\documentclass[11pt]{amsart}
\usepackage{amsmath}
\usepackage{comment}

\usepackage{tikz}

\theoremstyle{plain}
\newtheorem{thm}{Theorem}[section]
\newtheorem{lem}[thm]{Lemma}
\newtheorem{prop}[thm]{Proposition}
\newtheorem{cor}[thm]{Corollary}

\theoremstyle{definition}

\newtheorem{rem}[thm]{Remark}

\numberwithin{equation}{section}

\newcommand{\lda}{\lambda}
\newcommand{\om}{\Omega}                \newcommand{\pa}{\partial}
\newcommand{\va}{\varepsilon}            \newcommand{\ud}{\mathrm{d}}
\newcommand{\be}{\begin{equation}}      \newcommand{\ee}{\end{equation}}
\newcommand{\w}{\omega}                 
\newcommand{\Lda}{\Lambda}              \newcommand{\B}{\mathcal{B}}

\begin{document}

\title[Asymptotic Analysis of Singular Harmonic Maps]
{Asymptotic Analysis of Harmonic Maps With Prescribed Singularities}

\author[Han]{Qing Han}
\address{Department of Mathematics\\
University of Notre Dame\\
Notre Dame, IN 46556, USA} \email{qhan@nd.edu}

\author[Khuri]{Marcus Khuri}
\address{Department of Mathematics\\
Stony Brook University\\
Stony Brook, NY 11794, USA}
\email{marcus.khuri@stonybrook.edu}

\author[Weinstein]{Gilbert Weinstein}
\address{Department of Mathematics and Department of Physics\\
Ariel University\\
Ariel, 40700, Israel\\
(Corresponding author)}
\email{gilbertw@ariel.ac.il}

\author[Xiong]{Jingang Xiong}
\address{School of Mathematical Sciences\\
Laboratory of Mathematics and Complex Systems, MOE\\
Beijing Normal University\\
Beijing 100875, China}
\email{jx@bnu.edu.cn}


\begin{abstract}
This is the first in a series of two papers to establish the mass-angular momentum inequality for multiple black holes.
We study singular harmonic maps from domains of 3-dimensional Euclidean space to the hyperbolic plane having bounded hyperbolic distance to extreme Kerr harmonic maps. We prove that every such harmonic map admits a unique tangent harmonic map at the extreme black hole horizon. The possible tangent maps are classified and shown to be shifted `extreme Kerr' geodesics in the hyperbolic plane that depend on two parameters, one determined by angular momentum and another by conical singularities. In addition, rates of convergence to the tangent map are established. Similarly, 
expansions in the asymptotically flat end are presented. 
These results, together with those of Li-Tian \cite{LT,LT93} and Weinstein \cite{Wei90,Wei92}, provide a complete regularity theory for harmonic maps from $\mathbb R^3\setminus z\text{-axis}$ to $\mathbb  H^2$ with these prescribed singularities. The analysis is additionally utilized to prove existence of the so called near horizon limit, and to compute the associated near horizon geometries of extreme black holes. 
\end{abstract}




\maketitle

\section{Introduction}


Stationary vacuum black hole solutions play a fundamental role in general relativity. In particular, they are expected to serve as the final state of gravitational collapse for isolated systems. They also saturate fundamental geometric inequalities involving the ADM mass, such as the Penrose inequality and the mass-angular momentum inequality. It is therefore of paramount importance to classify these solutions and understand their properties.  
In this regard, the classical black hole uniqueness `no hair' conjecture asserts that the only 4-dimensional asymptotically flat solution of this type is the Kerr black hole. An important milestone in this context is the Hawking rigidity theorem
\cite{HawkingEllis}, which guarantees that a stationary vacuum spacetime possesses a rotational Killing field making it axisymmetric, under the assumption of analyticity. More recent work by 
Alexakis, Ionescu, and Klainerman \cite{AlexakisIonescuKlainerman1,AlexakisIonescuKlainerman2,AlexakisIonescuKlainerman3}
has made significant progress towards removing the analyticity hypothesis. 
It is with the axisymmetry condition that the harmonic maps of interest make an entrance to this topic. In particular, the original uniqueness result of Robinson \cite{Ro} for a connected nondegenerate horizon relied on a divergence identity, which through the observations of Bunting \cite{Bunting} and Mazur \cite{Mazur} could be explained from the point of view of a harmonic map structure. In fact, it was earlier shown by Ernst \cite{Er} and Carter \cite{Carter} that the axisymmetric stationary vacuum Einstein equations reduce to an axisymmetric singular harmonic map system from Euclidean 3-space $\mathbb{R}^3$ to the hyperbolic plane $\mathbb{H}^2$.  See the article of Chru\'{s}ciel-Costa \cite{CC} (as well as \cite{ChruscielCostaHeusler}) for further details on the uniqueness result for a single black hole.

The purpose of the present work is to establish precise regularity estimates and asymptotics for singular harmonic maps arising from stationary vacuum spacetimes with degenerate horizons. The primary aim is to use these estimates to prove in a companion paper 
the conjectured mass-angular momentum inequality for multiple black holes. 
We shall analyze solutions in all relevant regimes, namely in neighborhoods of the punctures, and in the asymptotically flat end. Since the analysis in neighborhoods of axis points and poles have been previously treated by Li-Tian \cite{LT,LT93} and Weinstein \cite{Wei92}, this fills-in the last gap in the study of the regularity of this type of harmonic map. The notion of a puncture refers to the regime associated with a degenerate horizon (extreme black hole), while a pole represents the intersection of an axis rod with a nondegenerate horizon.  In particular, we classify tangent maps at the punctures and establish rates of convergence, while precise asymptotic expansions are obtained in the asymptotically flat end. As mentioned, the analysis at punctures is primarily motivated by the mass-angular momentum inequality studied by Chru\'{s}ciel-Li-Weinstein \cite{CLW}, Dain \cite{Dain}, and Schoen-Zhou \cite{SchoenZhou}; see \cite{KhuriWeinstein} for the charged case. This Penrose-type inequality is tied to the grand cosmic censorship conjecture \cite{Penrose}, and is open for the case of multiple horizons. In connection with this problem, Chru\'{s}ciel-Li-Weinstein 
\cite[Remark 2.2]{CLW} stated that it is of interest to investigate and determine the regularity of solutions in the neighborhood of punctures. As further consequences of our results, we establish existence of the so called near horizon limit and classify the associated near horizon geometries of extreme horizon punctures, proving that they depend on two parameters. One parameter represents the angular momentum of the black hole, while the other is shown to be given explicitly in terms of the change in angle defects between adjacent axis rods.

\subsection*{Acknowledgements}
The authors would like to thank the American Institute of Mathematics for its hospitality.

\section{Background and Statement of Results}
\label{sec2}

Any solution of the asymptotically flat axisymmetric stationary vacuum Einstein equations may be obtained by specifying a harmonic map from $\mathbb{R}^3\setminus\Gamma$ to $\mathbb{H}^2$ with prescribed singularities along $\Gamma$, where $\Gamma$ denotes the $z$-axis in $\mathbb{R}^3$ and $\mathbb{H}^2$ represents the hyperbolic plane of curvature $-2$. The Dirichlet energy density of the map in question takes the form
\[
e(u,v)= |\nabla u|^2+e^{4u} |\nabla v|^2,
\]
and the corresponding harmonic map equations are given by 
\begin{align}\label{eq-HM}\begin{split}
\Delta u-2e^{4u}|\nabla v|^2&=0,\\ \Delta v+4\nabla u\cdot\nabla v&=0.
\end{split}\end{align}
The Laplacian $\Delta$ is with respect to the flat metric on $\mathbb{R}^3$ parameterized in cylindrical coordinates $(\rho,z,\phi)$, while the hyperbolic metric is parameterized in horospherical coordinates
\begin{equation*}
g_{\mathbb{R}^3}=d\rho^2 +dz^2 +\rho^2 d\phi^2,\quad g_{\mathbb{H}^2}=du^2 +e^{4u}dv^2.
\end{equation*}
Due to the nature of the singularities on $\Gamma$, the total Dirichlet energy of these harmonic maps is always infinite.

A special role in our analysis will be played by the \textit{extreme Kerr harmonic map}. This map may be expressed in polar coordinates $(r,\theta)$, with $\theta\in[0,\pi]$ and such that $\rho=r\sin\theta$, $z=r\cos\theta$, by
\begin{align} \label{eq:kerr-solution}
\begin{split}
u_K&= -\ln \sin \theta-\frac12 \ln \Big((r+m)^2 +m^2 +\frac{2m^3  (r+m) \sin ^2 \theta}{\Sigma} \Big),\\
v_K&=m^2 \cos \theta (3-\cos^2 \theta ) +\frac{m^4 \sin^4 \theta \cos \theta}{\Sigma} ,
\end{split}
\end{align} 
where $m>0$ is a parameter representing mass (or equivalently angular momentum divided by mass) and 
$\Sigma= (r+m)^2 +m^2 \cos^2 \theta$.
Notice that since $\sin \theta =0$ on $\Gamma$ the function $u_K$ is singular there, however 
\[
u_K+\ln \sin \theta \quad \text{ is smooth on }\mathbb{R}^3 \setminus \{0\}.
\]
Moreover, $v_K$ has a jump discontinuity at the origin, and in fact 
\begin{equation} \label{eq:trace-v-k}
v_K\big|_{\Gamma_+}= 2m^2, \quad v_K \big|_{\Gamma_-}=- 2m^2, 
\end{equation}
where $\Gamma_{+}=\Gamma\cap \{z>0\}$ and $\Gamma_{-}=\Gamma\cap \{z<0\}$. 
Although the Dirichlet energy of the extreme Kerr harmonic map is infinite, it has a finite renormalized energy where the \emph{renormalized energy density} of a map $(u,v)$ is defined as
\[
e'(U,v)=|\nabla U|^2+e^{4u} |\nabla v|^2, 
\]
where $U= u+\ln (r\sin\theta)$. This type of behavior exhibited by the extreme Kerr map will serve as a model for general solutions. Furthermore, the multi-extreme black hole family of axially symmetric harmonic maps constructed by Chru\'{s}ciel-Li-Weinstein \cite{CLW} have bounded $\mathbb{H}^2$-distance to extreme Kerr harmonic maps near each {\it puncture}, the point on $\Gamma$ where the $v$-component has a jump discontinuity.

The first result concerns the asymptotics of harmonic maps, possessing prescribed singularities modeled on extreme Kerr, near punctures. Without assuming axisymmetry, it is shown that a tangent map exists and rates of convergence to the tangent map are provided. Here, a tangent map refers to a limit of maps restricted to a sequence of spheres with radii tending to zero. In what follows, $B_1$ will denote the ball of unit radius centered at the origin in $\mathbb{R}^3$, while $\mathbb{S}^2$ denotes the unit sphere parameterized by the angles $(\theta,\phi)$ with $N$ and $S$ representing the north and south poles, $\nabla_{\mathbb{S}^2}$ indicates covariant differentiation on the sphere, and $H^1$ will be the Sobolev space of square integrable functions with square integrable derivatives.

\begin{thm} \label{thm:main-1} 
Let $(u,v) \in H_{\mathrm{loc}}^{1}(B_1\setminus \Gamma)$ be a harmonic map 
from $B_1\setminus \Gamma$ to $\mathbb{H}^2$. Suppose that 
\begin{align*}
|u+\ln \sin\theta |&\le \Lda  \quad \text{in } B_1\setminus \Gamma,\\
e'(u+\ln (r\sin \theta), v)&\in L^1_{\mathrm{loc}} (B_1\setminus \{0\}),
\end{align*}
with
\begin{equation}\label{oinhiqoinh}
v=a\text{ on }B_1\cap \Gamma_+ \quad\text{ and }\quad v=-a\text{ on }B_1\cap\Gamma_- \quad\text{ in the trace sense},
\end{equation}
for some positive constants $\Lda$ and $a$. 
Then $( u+\ln \sin\theta, v) \in C^{3,\alpha}(B_1\setminus \{0\})$ for any $\alpha\in(0,1)$,
and there exists a harmonic map $(\bar u, \bar v)$ 
from $\mathbb{S}^2 \setminus \{N,S\} \to \mathbb{H}^2 $ satisfying 
\begin{equation*}
(\bar u+\ln \sin\theta, \bar v) \in C^{3,\alpha}(\mathbb{S}^2) , \qquad \bar v(N)=a, \qquad \bar v(S)= -a,
\end{equation*}
such that 
\begin{equation}\label{qaohoihqhq}
\max_{\mathbb{S}^2}\left(|(r\partial_r)^l \nabla_{\mathbb{S}^2}^k (u(r)-\bar{u})|
+e^{(3+\alpha-k)\bar{u}}|(r\partial_r)^l \nabla_{\mathbb{S}^2}^k (v(r)-\bar{v})|\right)
\le C r^{\beta},
\end{equation}
for $l+k\leq 3$, where $C$ and $\beta$ are positive constants
and $u(r)=u(r,\cdot)$, $v(r)=v(r,\cdot)$.
\end{thm} 

Observe that the weight $e^{4u(x)} \sim \mathrm{dist}(x,\Gamma)^{-4}$ in $e'$ is sufficiently singular to guarantee that the trace of $v$ on $\Gamma_{\pm}$ makes sense. We note that the asserted $C^{3,\alpha}$-regularity in the punctured ball follows from Li-Tian \cite{LT} and Weinstein \cite{Wei90, Wei92}. The primary contribution of Theorem \ref{thm:main-1} is the existence and uniqueness of the limit $(\bar u, \bar v)$, along with the rate of convergence as stated in \eqref{qaohoihqhq}. It is an interesting point that although axisymmetry is not assumed a priori for $(u,v)$, the tangent map must always admit the axial symmetry as is shown below. It should also be noted that $ e'(u+\ln (r\sin \theta), v)\in L^1 (B_{r})$ for some $r>0$ is a conclusion rather than a hypothesis.
The next result classifies all the tangent harmonic maps in the above theorem, and proves their integrability in the sense of Allard-Almgren \cite{AA}.

\begin{thm}\label{thm:classification-non-degeneracy}  
Let $(\bar u, \bar v) \in H^1_{loc}(\mathbb{S}^2\setminus \{N,S\})$ be a harmonic map from $\mathbb{S}^2 \setminus \{N,S\} \to \mathbb{H}^2 $ satisfying 
\begin{equation}\label{oaihfoihnah}
\begin{split}
&|\bar u+\ln \sin\theta|\le   \Lda \quad \mbox{on }\mathbb{S}^2,\\&  |\nabla_{\mathbb{S}^2}( \bar u+\ln \sin\theta)|^2+ e^{4\bar u} |\nabla_{\mathbb{S}^2}  \bar v|^2 \in L^{1}(\mathbb{S}^2) ,\\&
\bar v(N)=a \quad\text{ and }\quad  \bar v(S)= -a \quad\text{ in the trace sense},
\end{split}
\end{equation}
for some positive constants $\Lda$ and $a$.  Then 
there exists a constant $b\in (-1,1)$ such that
\begin{equation}\label{eq:general-kerr-geodesic}
\bar{u}(s)=\frac{1}{2}\ln\left[\frac{1}{2a}\cosh(-2s+\tanh^{-1} b)\right],\quad\quad \bar{v}(s)= a \tanh (- 2s+\tanh^{-1} b), 
\end{equation}
where $s= \frac{1}{2} \ln \frac{1-\cos \theta}{1+\cos \theta}$ is arclength parameterization in $\mathbb{H}^2$. Moreover, the kernel of the linearized harmonic map at $(\bar u, \bar v)$ is 1-dimensional, and is generated by $(\pa_b\bar{u}, \pa_b \bar{v})$.
\end{thm} 

Interpretations may be given for the parameters defining the tangent map. Namely, $a=2\mathcal{J}$ is twice the angular momentum of the black hole, and as is described in detail below, the parameter $b$ encodes the difference of angle defects associated with $\Gamma_+$ and $\Gamma_-$. Additionally, the problem is invariant under $v\mapsto\pm v+c$ for any constant $c$. Thus, the boundary conditions $a$ and $-a$ in \eqref{oaihfoihnah} can be replaced with any two numbers whose difference is $\pm 2a$.
For any choice of parameters $a$ and $b$, the tangent map expression \eqref{eq:general-kerr-geodesic} represents a geodesic in hyperbolic space parameterized by arclength, which is a translation of the tangent map of an extreme Kerr, and agrees with the extreme Kerr geodesic when $b=0$.

The proof of Theorem \ref{thm:main-1} may be described as follows. Once regularity is established for the renormalization
$(u+\ln \sin\theta, v)$, the theorem consists of studying an isolated singularity problem for the renormalized map, which shares some common features with the tangent map problem for harmonic maps, 
see for instance Adams-Simon \cite{AS}, Gulliver-White \cite{GW}, Hardt \cite{Ha}, Schoen-Uhlenbeck \cite{SU}, Simon \cite{Simon83, Simon85, simoneth}, White \cite{White}. 
In these works, a monotonicity formula with respect to the energy density integrated 
on balls centered at singular points is crucial. 
However,
such a monotonicity formula is absent for the renormalized energy density in our setting. 
Nevertheless, due to the negative curvature of the target space $\mathbb{H}^2$,  
we instead can adapt the method of Schoen-Uhlenbeck \cite{SU} and Li-Tian \cite{LT} 
to establish uniform bounds for derivatives of the renormalized map, away from the isolated singular points.
Then, after developing some important weighted estimates, we adapt the method of Simon \cite{Simon83, Simon85, simoneth} for establishing asymptotics   
of solutions to elliptic equations with gradient type structure. 
Due to the last conclusion of Theorem \ref{thm:classification-non-degeneracy},
the integrability hypothesis of Allard-Almgren \cite{AA} (see also \cite{simoneth}) is satisfied within the context of the current problem, which yields the rate $r^{\beta}$ instead of the slower one $|\ln r|^{-\gamma}$ for some $\gamma>0$. Here the constant $\beta$ can be chosen optimally, depending on the second eigenvalue of the linearized operator at $(\bar u,\bar v)$ for maps from $\mathbb{S}^2 \setminus \{N,S\} \to \mathbb{H}^2 $ with prescribed singularity.

Having studied the asymptotics of solutions at punctures, we now establish expansions at infinity. These will imply that the associated stationary vacuum spacetime is asymptotically flat.
The proof, however, is quite different from Theorem~\ref{thm:main-1}. More precisely, we show that the harmonic map system essentially becomes decoupled as $r \to \infty$, allowing for a separate analysis of each equation via eigenfunction expansions.
Some related results were obtained by Beig-Simon \cite{BeigSimon} (see also Weinstein \cite[Proposition 5]{Wei92}) for stationary vacuum solutions of the Einstein equations. In what follows, the complement of the closed ball in $\mathbb{R}^3$ will be denoted by $\overline{B}_1^c$. 

\begin{thm} \label{thm:main-2} 
Let $(u,v) \in H_{\mathrm{loc}}^{1}(\overline{B}_1^c \setminus \Gamma)$ be a harmonic map 
from $\overline{B}_1^c \setminus \Gamma$ to $\mathbb{H}^2$. Suppose that 
\begin{align*}
|u+\ln \rho |+|v|&\le \Lda  \quad \text{in } \overline{B}_1^c \setminus \Gamma,\\  
e'(u+\ln \rho, v)&\in L^1_{\mathrm{loc}} (\overline{B}_1^c ), 
\end{align*}
with
$$v=a\text{ on }\overline{B}_1^c \cap \Gamma_+ \quad
\text{ and } \quad v=-a\text{ on }\overline{B}_1^c\cap\Gamma_- \quad\text{ in the trace sense},$$
for some positive constants $\Lda$ and $a$. 
Then $( u+\ln \rho, v) \in C^{3,\alpha}(\overline{B}_1^c)$ for all $\alpha\in(0,1)$, and for all $r\geq 1$, 
\begin{equation*}
\big|u(r)+\ln \rho -c_0-c_1 r^{-1}-Y_1 r^{-2}-Y_2 r^{-3}\big|_{C^3(\mathbb{S}^2)} \leq Cr^{-3-\beta},
\end{equation*}
\begin{equation*}
\big|v(r)- \tfrac12a\cos \theta (3-\cos^2 \theta )-c_2 r^{-1} \sin^ 4 \theta \big |_{C^3(\mathbb{S}^2)}
\leq Cr^{-1-\beta}(\sin\theta)^{3+\alpha},  
\end{equation*}
for some $\beta\in(0,1)$, where $c_0$, $c_1$, $c_2$, $C$ are constants and $Y_1$, $Y_2$ are degree 1 and 2 spherical harmonics, respectively. Moreover, corresponding asymptotics are
valid for the $r$-derivatives of $(u,v)$.
\end{thm}

Theorems \ref{thm:main-1} and \ref{thm:main-2} address the asymptotics of harmonic maps at punctures
and in the asymptotically flat end, respectively. 
The question of regularity in the vicinity of axis points, and of poles, was previously studied by Li-Tian and Weinstein. In this situation, the origin of coordinates may be taken at the axis point or pole, with the nonnegative $z$-axis representing an axis rod, and the nonpositive $z$-axis representing either another portion of the axis or the horizon rod for the latter situation. Since blow-up of the harmonic map should then only occur on both the nonnegative and nonpositive $z$-axis, or just the nonnegative $z$-axis, the harmonic functions $\ln \rho$ or $\ln\sqrt{r-z}$ are used to renormalize the map, respectively. In contrast with the case of punctures, the tangent map at  axis points and poles is trivial. Li-Tian \cite[Theorem 1.1]{LT}, \cite[Theorem 2.2]{LT93} and Weinstein \cite[Theorem 5]{Wei90}, \cite[Corollary 1]{Wei92} were able to obtain $C^{3,\alpha}$ regularity in neighborhoods of axis points and $C^{1,\alpha}$ regularity in neighborhoods of poles for the renormalized maps, using the small energy regularity approach. Moreover, they assumed separate extra hypotheses, such as a minimizing property \cite{LT,LT93} or axisymmetry
and bounded renormalized energy up to the origin \cite{Wei90,Wei92}.
Theorem \ref{thm:main-1}, \ref{thm:classification-non-degeneracy}, and \ref{thm:main-2}, together with the cited results due to Li-Tian \cite{LT,LT93} and Weinstein \cite{Wei90,Wei92}, provide a complete regularity theory for harmonic maps from $\mathbb R^3\setminus\Gamma$ to $\mathbb  H^2$ with prescribed singularities.

We now discuss a geometric interpretation of the parameter $b$ appearing in Theorem \ref{thm:classification-non-degeneracy}, as well as an application of the asymptotic analysis at punctures. Consider the domain of outer communication $\mathcal{M}^4$
of a stationary axisymmetric 4-dimensional spacetime. Stationarity and axisymmetry imply that the isometry group admits $\mathbb{R}\times U(1)$ as a subgroup. Under reasonable hypotheses \cite{CC}, the orbit space 
$\mathcal{M}^4/[\mathbb{R}\times U(1)]$ is homeomorphic to the right half plane $\{(\rho,z)\mid \rho>0\}$, and the spacetime metric may be expressed in Weyl-Papapetrou coordinates as
\begin{equation}\label{mcg11}
\mathbf{g} = - e^{2U}d \tau^2 + \rho^2 e^{-2U}  (d \phi +wd\tau)^2 +  e^{-2U+2 \boldsymbol{\alpha}} (d \rho^2 + dz^2),
\end{equation}
where the $U(1)$ symmetry is generated by $\partial_{\phi}$ for $\phi\in[0,2\pi)$ and the time translation symmetry is generated by $\partial_{\tau}$. By setting $u=U-\ln \rho$, the vacuum Einstein equations \cite{Wei90} may be expressed as the harmonic map equations for $(u,v):\mathbb{R}^3 \setminus \Gamma \rightarrow\mathbb{H}^2$, 
and a set of quadrature equations all given by 
\begin{equation}\label{equationaklfahg1}
\Delta u-2e^{4u}|\nabla v|^2=0,\quad \Delta v+4\nabla u\cdot\nabla v=0,
\end{equation}
\begin{equation}\label{wequations1}
w_{\rho}=2\rho e^{4u}v_z,\quad w_{z}=-2\rho e^{4u}v_{\rho},
\end{equation}
\begin{equation}\label{alphaequation1}
\rho^{-1}(\boldsymbol{\alpha}_{\rho}-2u_{\rho}-\rho^{-1})=u_{\rho}^2-u_z^2+e^{4u}(v_{\rho}^2-v_z^2),
\quad \rho^{-1}(\boldsymbol{\alpha}_z -2u_z)=2u_{\rho}u_z +2e^{4u}v_{\rho}v_z.
\end{equation}
Observe that the integrability conditions for \eqref{wequations1} and \eqref{alphaequation1} are equivalent to the harmonic map equations \eqref{equationaklfahg1}. Moreover, the $z$-axis is decomposed into a sequence of intervals $\{\Gamma_l\}_{l\in \mathcal{I}}$ for some index set $\mathcal{I}$, called \textit{rods}. There are two types of rods, those on which $|\partial_{\phi}|$ vanishes are designated as \textit{axis rods}, while all others are referred to as \textit{horizon rods}. It should be noted that a horizon rod $\Gamma_h$ may also be identified by the vanishing in norm of a Killing field, namely $\partial_{\tau}+\Omega_h\partial_{\phi}$ where $\Omega_h=-w|_{\Gamma_h}$ is a constant indicating the angular velocity of the horizon. The intersection point of an axis rod with a horizon rod is called a \textit{pole}, while the intersection point of two axis rod closures is referred to as a \textit{puncture}. These latter points represent the horizons of extreme black holes. The rod structure for the extreme Kerr harmonic map given by \eqref{eq:kerr-solution} consists of two semi-infinite axis rods, and a single puncture at the origin indicating a degenerate horizon.

An important issue when constructing the spacetime metric $\mathbf{g}$ in \eqref{mcg11} from a harmonic map, is the possible presence of conical singularities along axis rods. A conical singularity at point $(0,z_0)$ on an axis rod $\Gamma_{l}$ has an associated \textit{angle defect} $\theta\in(-\infty,2\pi)$ arising from the 2-dimensional cone formed by the orbits of $\partial_{\phi}$ over the line $z=z_0$. More precisely
\begin{equation}\label{alphaequation111}
\frac{2\pi}{2\pi-\theta}=\lim_{\rho\rightarrow 0}\frac{2\pi\cdot\mathrm{Radius}}{\mathrm{Circumference}}
=\lim_{\rho\rightarrow 0}
\frac{\int_{0}^{\rho}e^{\boldsymbol{\alpha}-U}}{\rho e^{-U}}
=e^{\boldsymbol{\alpha}(0,z_0)}.
\end{equation}
The existence of this limit as well as the fact that this quantity is constant along each axis rod, follows from the interior axis
regularity established by Li-Tian \cite{LT,LT93} and Weinstein \cite{Wei90,Wei92}; specifically the constancy is a consequence of \eqref{alphaequation1}.  We will denote the \textit{logarithmic angle defect} $\ln\left(\frac{2\pi}{2\pi-\theta}\right)$, on an axis rod $\Gamma_l$, by $\mathbf{b}_l$.
Moreover, the absence of a conical singularity corresponds to a zero logarithmic angle defect. In this case, the metric is smoothly extendable across the axis, which may be checked via a change to Cartesian coordinates. The conical singularity on $\Gamma_l$ is said to have an \textit{angle deficit} if $\mathbf{b}_l>0$, and an \textit{angle surplus} if $\mathbf{b}_l<0$. Physically, the logarithmic angle defect's sign determines the character of the force associated with the axis rod, through the following formula computed in \cite[pg. 921]{Wei90}, namely
\begin{equation*}
\text{Force}=\frac14 \left(e^{-\mathbf{b}_l} -1\right).
\end{equation*}
We are able to give a precise relation between the tangent map parameter $b$, and the logarithmic angle defects in the spacetime metric. It should be noted that the conical singularity difference is independent of the angular momentum parameter in the tangent map.

\begin{thm}\label{conicalsing}
Let $(u,v)$ be as in Theorem \ref{thm:main-1} and axially symmetric, and consider the associated stationary vacuum spacetime $(\mathcal{M}^4,\mathbf{g})$. If the two axis rods $\Gamma_{l+1}$ and $\Gamma_{l}$ lying directly to the north and south of the puncture $p_l =0$ each have logarithmic angle defect $\mathbf{b}_{l+1}$ and $\mathbf{b}_{l}$, respectively, then 
\begin{equation*}
\mathbf{b}_{l+1}-\mathbf{b}_{l}=\ln\left(\frac{1+b}{1-b}\right).
\end{equation*}
\end{thm}

Another concept associated with extreme black holes is that of the \textit{near horizon geometry} (NHG). These spacetimes arise through a limiting process in the vicinity of a degenerate black hole (puncture), which infinitely magnifies the spacetime geometry at the horizon, see \cite{KunduriLucietti} for a detailed review and \cite{AKKNHG} for the use of harmonic maps to construct these solutions. In the physics literature, a rigorous justification of the so called near horizon limit that gives rise to the NHG, is often not addressed. Here we show that this limit exists as a consequence of the asymptotic analysis studied in this work, and compute the limit. More precisely, utilizing polar coordinates centered at a puncture, the NHG metric is obtained from the scaling $r=\epsilon \bar{r}$, $\tau=\epsilon^{-1} \bar{\tau}$,  $\phi=\bar{\phi}+\Omega \epsilon^{-1}\bar{\tau}$ by letting $\epsilon\rightarrow 0$, where $\Omega$ is the angular velocity of the black hole. A relevant example is the extremal Kerr throat metric \cite{BHorowitz}, which is derived from the extreme Kerr solution.

\begin{thm}\label{NHG}
Let $(u,v)$ be as in Theorem \ref{thm:main-1} and axially symmetric, and consider the associated stationary vacuum spacetime $(\mathcal{M}^4,\mathbf{g})$. Then the near horizon limit exists, and gives rise to the near horizon metric
\begin{align*}
\mathbf{g}_{NH} 
=&-\bar{r}^2\left(\frac{1+\cos^2 \theta +2b\cos\theta}{2a\sqrt{1-b^2}}\right)d\bar{\tau}^2
+\frac{2a\sqrt{1-b^2}\sin^2 \theta}{1+\cos^2 \theta+2b\cos\theta}\left(d\bar{\phi}+\frac{\bar{r}}{a}d\bar{\tau}\right)^2\\
&+\frac{a\sqrt{1-b^2}}{2}(1+\cos^2 \theta+2b\cos\theta)\left(\frac{d\bar{r}^2}{\bar{r}^2}+d\theta^2 \right).
\end{align*}
In particular, this agrees with the extremal Kerr throat metric when $a=2\mathcal{J}$ and $b=0$, where $\mathcal{J}$ denotes angular momentum.
\end{thm}

The paper is organized as follows. In Section \ref{s:regularity}, we prove regularity 
and uniform estimates for the renormalized maps on a portion of the cylinder $I\times \mathbb{S}^2$, where $I$ is an interval. This is inspired by Li-Tian \cite{LT} and Schoen-Uhlenbeck \cite{SU}, with the primary goal consisting of certain linear form estimates that are important for later applications. 
Section \ref{s:classification} is dedicated to the proof of Theorem \ref{thm:classification-non-degeneracy}.  
The proof of Theorem \ref{thm:main-1} is presented in Section \ref{s:convergence-1}, by adapting the method of Simon \cite{Simon83, Simon85}, while the proof of 
Theorem \ref{thm:main-2} is given in Section \ref{s:convergence-2}. Lastly, conical singularities and near horizon geometries are studied, and Theorems \ref{conicalsing} and \ref{NHG} are established, in Section \ref{sec7}.



\smallskip

\noindent \textbf{Acknowledgements.} The authors would like to thank YanYan Li for helpful discussions.

\section{Regularity and Uniform Estimates on the Cylinder}
\label{s:regularity}

In this section, we discuss the regularity of harmonic maps 
away from the origin and derive necessary estimates of 
the $L^\infty$-norms of harmonic maps and their derivatives in terms of the $L^2$-norm 
of a specific combination of derivatives. The analysis here will be applied to study both the convergence near 
black holes and in the asymptotically flat end of associated stationary vacuum spacetimes. It should be pointed out
that the results of this section do not require axisymmetry.

Let $(r, \phi, \theta)$ be spherical coordinate on $\mathbb R^3$
with $r>0$, $\phi\in [0,2\pi)$, $\theta\in [0,\pi]$, and set  
$$t=-\ln r.$$ 
Note that the flat metric may be expressed with respect to the two radial coordinates as
\[
dr^2+r^2 g_0=  e^{-2t} (dt^2+g_0), 
\]
where 
\[
 g_0=d\theta^2+ \sin^2 \theta d\phi^2 
\]
is the round metric on $\mathbb{S}^2$.  We will denote the product metric on $\mathbb{R}\times \mathbb{S}^2$ by
\begin{equation}\label{eq:product-metric} 
g= d  t^2+g_0. 
\end{equation} 
In order to write the harmonic map equations 
as a differential system 
on the cylinder, let $\nabla_g$ and $\Delta_g$ be the gradient operator and Laplace-Beltrami operator 
with respect to the metric $g$, that is
\[
\nabla_{g}=(\pa_t, \nabla_{g_0}), \quad\quad \Delta_g=\pa_t^2+\Delta_{g_0}.
\]
By a straightforward computation, the harmonic map system \eqref{equationaklfahg1}
reduces to the following set of equations on $\mathbb R\times \mathbb S^2$:
\begin{align}\label{eq:HM-cylinder}
\begin{split}
\Delta_gu-\partial_tu -2 e^{4u} |\nabla_g v|^2 &=0, \\
\Delta_gv-\partial_tv +4\nabla_{g}u\cdot   \nabla_{g} v &=0,
\end{split} 
\end{align}
where $\cdot$ represents the inner product given by $g$. Furthermore, by setting 
\begin{equation} \label{eq:L-LB}
L
= \Delta_g-\partial_t=\pa_t^2- \pa_t  +\Delta_{g_0},
\ee
the system of equations \eqref{eq:HM-cylinder} may be rewritten as
\begin{align*}
L u -2 e^{4u} |\nabla_{g} v|^2 &=0, \\
Lv +4\nabla_{g}u\cdot   \nabla_{g} v &=0. 
\end{align*}

Let $N$ be the north pole on $\mathbb S^2$ and $S=-N$ the south pole. 
We will consider the model function $\ln \sin \theta$  on $\mathbb S^2$.  
Note that $\ln \sin \theta$ is singular only at $N$ and $S$.  
A simple computation shows 
$$L(\ln \sin \theta)=-1.$$ 
Let $\w\in C^\infty([-2, 2] \times \mathbb{S}^2\setminus \{N,S\})$ be a positive function  satisfying 
\begin{equation} \label{eq:bds-omega-1}
\|\ln \w- \ln \sin \theta\|_{C^{10}([-2, 2] \times \mathbb{S}^2)}\le A_1,
\ee
for some constant $A_1\ge 1$. 
We can take a  larger $A_1$, if necessary, to have  
\begin{equation} \label{eq:bds-omega-3}
\|L\ln \w\|_{C^{8}([-2, 2] \times \mathbb{S}^2)}\le A_1. 
\end{equation} 
In the rest of the section, $\w$ is fixed to satisfy \eqref{eq:bds-omega-1}. The reason for carrying out the analysis with a general singular function instead of restricting to the model, is that different models will be needed when applying the results below to study expansions near horizons and in asymptotically flat ends of stationary vacuum spacetimes.

In what follows we introduce a renormalized function $\Phi$ given by
\begin{equation}\label{eq:definition-Phi}
u=\Phi-\ln \w,
\end{equation}
and consider weak solutions of the renormalized harmonic map system 
\begin{equation}\label{eq:main-unified-2} 
\begin{split}
L \Phi -2 e^{4u} |\nabla_{g} v|^2  &=L\ln\w, \\
Lv +4\nabla_{g}u\cdot   \nabla_{g} v&=0.
\end{split} 
\end{equation}
It will be assumed that $(\Phi,v)\in H^1((-2, 2)\times \mathbb{S}^2)$ satisfies
\begin{align} \label{eq:ubd-unified-1} 
&|\Phi|  \le \Lda \quad \text{ on }(-2,2) \times \mathbb{S}^2, \\
\label{eq:ea-unified-1}
\int_{-2}^{2} \int_{\mathbb{S}^2} &(|\nabla_{g} \Phi|^2 + \w^{-4}  |\nabla_{g} v|^2 )\, d vol_{g_0}  d  t <\infty, 
\end{align}
for some constant $\Lda\ge 1$, and in trace sense
\begin{equation} \label{eq:trace-unified-1}
\begin{split}
v=a  \text{ on  }\mathbb (-2,2) \times \{N\},\quad \quad
v=-a  \text{ on  }\mathbb (-2,2) \times \{S\},
\end{split}
\end{equation} 
for some positive constant $a$. 
Here $H^1$ represents the Sobolev space of $L^2$ functions having square integrable derivatives.
Moreover, by referring to $(\Phi,v)$ as a weak solution this should be understood in the distributional sense, see Li-Tian \cite{LT} for further details. 

\begin{rem}\label{rem:homlize}
Consider the Dirichlet problem 
\begin{align}\label{eq:homlize}\begin{split}
L\xi&= L\ln \w \quad \text{ on }(-2,2) \times \mathbb{S}^2, \\
\xi&=0\quad \text{ on }\{-2,2\}\times \mathbb{S}^2.
\end{split}\end{align} 
By \eqref{eq:bds-omega-3}, there exists a unique solution $\xi\in C^9([-2,2] \times \mathbb{S}^2)$ of \eqref{eq:homlize} satisfying 
$$\|\xi\|_{C^9([-2,2] \times \mathbb{S}^2)}\le CA_1,$$ 
for some universal constant $C>0$.  It may be checked that $(\Phi-\xi, v)$ then 
satisfies \eqref{eq:main-unified-2} with a homogeneous right-hand side, and with a different $\omega$ in \eqref{eq:definition-Phi}. Therefore, in studying the regularity of $(\Phi,v)$, it may be assumed without loss of generality that $L\ln \w=0$. 
\end{rem}

We we begin with an energy estimate for the density
$$ e_{g}(\Phi,v)= |\nabla_{g} \Phi|^2  + e^{4u}  |\nabla_{g}  v|^2. $$
Notice that in the proof below, only the first equation of \eqref{eq:main-unified-2} will be
utilized.

\begin{lem}\label{lem:local-bound} Let $(\Phi,v)\in H^{1}((-2, 2)\times \mathbb{S}^2)$ 
be a weak solution of \eqref{eq:main-unified-2} satisfying \eqref{eq:ubd-unified-1}, \eqref{eq:ea-unified-1}, 
and \eqref{eq:trace-unified-1} for some positive constants $\Lambda$ and $a$. 
Suppose that $L\ln\w=0$.  Then
\[
\int_{-R}^{R}\int_{\mathbb{S}^2}e_{g}(\Phi,v) \, d  vol_{g_0}    d  t \le \frac{5000\Lda^2}{ (2-R)^2}
\]
for any $0<R<2$.
\end{lem} 

\begin{proof}  Let $\eta=\eta(t)\in C_c^2(-2,2)$  be a smooth cut-off function satisfying $0\le \eta \le 1$,
\[
\eta= 1 \quad \text{on } (-R,R)\quad \text{and}\quad |\eta'|+(2-R)|\eta''| \le \frac{10}{2-R}. 
\]
Using $\eta^2$ as a test function in the first equation of \eqref{eq:main-unified-2}, we have 
\begin{align*}
\int_{-2}^{2} \int_{\mathbb{S}^2}\Big[\Phi\big(\frac{ d ^2}{ d  t^2}\eta^2+ \frac{ d }{ d  t}\eta^2\big) 
- 2 \eta^2 e^{4u} |\nabla_{g}  v|^2  \Big] \, dvol_{g_0}  d  t =0. 
\end{align*}
Since $|\Phi|\le \Lda$, it follows that
\begin{align} \label{eq:energy-1}
\int_{-2}^{2} \int_{\mathbb{S}^2} \eta^2 e^{4u} |\nabla_{g}  v|^2 \, d  vol_{g_0}  d  t  \le \frac{1000\Lda}{ (2-R)^2}. 
\end{align}
Next, using $\eta^2 \Phi$ as a test function  in the first equation of \eqref{eq:main-unified-2}, we have 
\begin{align*}
\int_{-2}^{2} \int_{\mathbb{S}^2}&\Big[\pa_t \Phi \pa_t (\Phi \eta^2 ) 
+ \eta^2 |\nabla_{g_0} \Phi|^2 - \frac12{\Phi^2} \frac{ d  }{ d  t} \eta^2  
+2 \eta^2 \Phi e^{4u} |\nabla_{g}  v|^2 \Big] \, d  vol_{g_0}  d  t =0,
\end{align*}
and by the Cauchy inequality
\[
\pa_t \Phi \pa_t (\Phi \eta^2 )  \ge \frac12 |\pa_t \Phi|^2 \eta^2 - 4 \Phi^2 |\pa_t \eta|^2. 
\]
Hence
\[
\int_{-2}^{2} \int_{\mathbb{S}^2} \eta^2 |\nabla_{g}  \Phi|^2 \, d  vol_{g_0}  d  t 
\le  \frac{1000\Lda^2}{ (2-R)^2} 
+4  \int_{-2}^{2} \int_{\mathbb{S}^2} \eta^2 e^{4u} |\nabla_{g}  v|^2 \, d  vol_{g_0}  d  t. 
\]
Combining the above inequality with \eqref{eq:energy-1} produces the desired result. 
\end{proof}
 

We now proceed to study the regularity and uniform estimates of $(\Phi, v)$.
Li-Tian \cite{LT} proved that $(\Phi,  v)\in C^{3,\alpha}$ for any $0<\alpha<1$, if it is a local minimizer.  
The minimality was used only in the proof of a monotonicity-type formula, 
which can be derived from a Caccioppoli-type (or reverse Poincar\'{e}) inequality in our setting as Nguyen \cite[pg. 424]{N} pointed out. 
For any $P_0\in \mathbb{S}^2$, we denote by $\B_{R}(P_0)$ the open geodesic ball centered at $P_0$ 
with radius $R>0$, and set 
\[
Q_R(t_0,P_0)= (t_0-R, t_0+R) \times \B_{R}(P_0). 
\]
The following Caccioppoli-type inequality may be viewed as corresponding to Lemma 4.1 of Li-Tian \cite{LT}.

\begin{lem}\label{lem:monotonicity} 
Let $(\Phi,v)\in H^{1}((-2, 2)\times \mathbb{S}^2)$ be a weak solution 
of \eqref{eq:main-unified-2} satisfying \eqref{eq:ubd-unified-1}, \eqref{eq:ea-unified-1}, 
and \eqref{eq:trace-unified-1} for some positive constants $\Lambda$ and $a$. Suppose that $L\ln\w=0$,
$c_1\in [-\Lda, \Lda]$ is an arbitrary constant, and $c_2=a$ if $N\in \B_{2\delta}(P_0) $ while
$c_2=-a$ if $S \in \B_{2\delta}(P_0)$ otherwise $c_2$ is an arbitrary constant. 
Then for any $Q_\delta(t_0,P_0) \subset (-2,2) \times \mathbb{S}^2$ with $0<\delta<1/2$ it holds that
\begin{align*}
\int_{Q_{\delta/2}(t_0,P_0)} e_{g}(  \Phi, v)\, d  vol_{g} 
\le  \frac{C}{\delta^2} \int_{Q_{\delta}(t_0,P_0)\setminus Q_{\delta/2}(t_0,P_0)} 
\big(  |  \Phi-c_1|^2+\w^{-4} | v-c_2|^2\big)\, d  vol_{g},
\end{align*}
where $C$ is a positive constant depending only on $\Lda$ and $A_1$.
\end{lem}

\begin{proof} 
Set $\bar g=e^{-2t}g$. A simple computation shows that   
\[
|\nabla_{\bar g} f|^2= e^{2t}|\nabla_{g}f |^2, \quad\quad\quad \Delta_{\bar g}f=e^{2t}( \Delta_gf-\partial_tf)=e^{2t}Lf,
\] 
for any $f\in C^2(\mathbb R\times \mathbb S^2)$. Moreover, if $L\ln\w=0$ then \eqref{eq:main-unified-2} is equivalent to  
\begin{equation}\label{eq:main-unified-4} 
\begin{cases}
\Delta_{\bar g}  \Phi -2 e^{4 u}  |\nabla_{\bar g}  v|^2 =0, \\
\Delta_{\bar g}  v+4\nabla_{\bar g}  u \cdot \nabla_{\bar g}  v =0.
\end{cases} 
\end{equation} 
We now have that $(\Phi, v)$ is a weak solution of \eqref{eq:main-unified-4} on $(-2,2) \times \mathbb{S}^2$.

Let $\eta\in C_c^2(Q_{\delta}(t_0,P_0))$ be a cutoff function with $0\le \eta\le 1$, 
$\eta=1$ on $Q_{\delta/2}(t_0,P_0)$, and $|\pa_t^k \eta|+ |\nabla_{g_0} ^k\eta| \le 10 \delta^{-k}$ 
for $k=1,2$. Note that the second equation of \eqref{eq:main-unified-4} can be written 
in the divergence form $\mathrm{div}(e^{4u} \nabla_{\bar g} v)=0$.  
Then using $\eta^2e^{4 u} ( v-c_2)$ as a test function in this equation, 
and applying the Cauchy inequality produces
\begin{align*}
0&= \int_{Q_{\delta}(t_0, P_0)} 
e^{4 u}  \nabla_{\bar g}v \cdot  \nabla_{\bar g}(\eta^2( v-c_2)) \, d  vol_{\bar g}\\&=
\int_{Q_{\delta}(t_0, P_0)}\Big(  e^{4 u} | \nabla_{\bar g}  v|^2 \eta^2 
+ 2 \eta (v-c_2) e^{4 u}  \nabla_{\bar g} v \cdot  \nabla_{\bar g} \eta\Big) \, d  vol_{\bar g} \\&
\ge \frac{1}{2}\int_{Q_{\delta}(t_0, P_0)}  e^{4 u} | \nabla_{\bar g}  v|^2 \eta^2\, d  vol_{\bar g}  
-4\int_{Q_{\delta}(t_0, P_0) \setminus Q_{\delta/2}(t_0, P_0) }  
e^{4u}|\nabla_{\bar g}\eta|^2 ( v-c_2)^2 \, d  vol_{\bar g}. 
\end{align*}
It follows that 
\begin{equation} \label{eq:mono-v-1}
\int_{Q_{\delta}(t_0, P_0)}  e^{4 u} | \nabla_{\bar g}  v|^2 \eta^2\, d  vol_{\bar g} 
\le \frac{80}{\delta^2} \int_{Q_{\delta}(t_0, P_0) \setminus Q_{\delta/2}(t_0, P_0) } 
e^{4 u}   ( v-c_2)^2\, d  vol_{\bar g}. 
\ee
Using $\eta^2(  \Phi-c_1)$ as a test function in the first equation of \eqref{eq:main-unified-4} yields
\begin{align*}
&\int_{Q_{\delta}(t_0, P_0)}   \nabla_{\bar g} \Phi \cdot  \nabla_{\bar g}(\eta^2(  \Phi-c_1)) \, d  vol_{\bar g} \\
&\qquad= -2 \int_{Q_{\delta}(t_0, P_0)}  e^{4 u} | \nabla_{\bar g}  v|^2 \eta^2(  \Phi-c_1)\, d  vol_{\bar g}  \\
&
\qquad\le 4 \Lda  \int_{Q_{\delta}(t_0, P_0)}  e^{4 u} | \nabla_{\bar g}  v|^2 \eta^2\, d  vol_{\bar g}. 
\end{align*} 
By the  Cauchy  inequality again, and \eqref{eq:mono-v-1},  we conclude that
\[
\int_{Q_{\delta}(t_0, P_0)} |  \nabla_{\bar g}\Phi|^2 \eta^2\, d  vol_{\bar g} 
\le \frac{C}{\delta^2} \int_{Q_{\delta}(t_0, P_0) \setminus Q_{\delta/2}(t_0, P_0) } 
\big((  \Phi-c_1)^2+ e^{4 u}   ( v-c_2)^2\big)\, d  vol_{\bar g}.
\]
Now observe that $| u+\ln \w| \le \Lda$ implies
\begin{align*}&\int_{Q_{\delta}(t_0, P_0)} 
\big(|  \nabla_{\bar g}\Phi|^2+e^{4 u} | \nabla_{\bar g}  v|^2\big) \eta^2\, d  vol_{\bar g}\\
&\qquad \le \frac{C}{\delta^2} \int_{Q_{\delta}(t_0, P_0) \setminus Q_{\delta/2}(t_0, P_0) } 
\big((  \Phi-c_1)^2+ \omega^{-4}   ( v-c_2)^2\big)\, d  vol_{\bar g}.
\end{align*}
Since $e^{-2t}$ has positive upper and lower bounds on $(-2,2) \times \mathbb{S}^2$, 
the gradients and the volume forms can be changed from $\bar g=e^{-2t}g$ to $g$. 
This gives the desired inequality.
\end{proof}

\begin{thm}\label{thm:reg}
Let $(\Phi,v)\in H^{1}((-2, 2)\times \mathbb{S}^2)$ be a weak solution of \eqref{eq:main-unified-2}
satisfying \eqref{eq:ubd-unified-1},  \eqref{eq:ea-unified-1}, and \eqref{eq:trace-unified-1} for some positive constants $\Lambda$ and $a$.  
Then
$\pa_t^m\pa_\phi^n(\Phi,v)\in {C^{3,\alpha}((-2,2) \times \mathbb{S}^2 )}$ for any $\alpha\in (0,1)$ 
and $ m+n\le 3$.   
Moreover, for any $1<R<2$ the following estimates hold
\[
\sum_{m+n\le 3}\| \pa_t^m\pa_\phi^n(  \Phi,  v)\|_{C^{3,\alpha}((-R,R) \times \mathbb{S}^2)}  \le C, 
\]
and 
\begin{equation} \label{eq:tan-estimates}
\begin{split}
\sum_{l=0}^{3}\sum_{m,n\le 3}  \w^l  \left|\pa_\theta^l\pa_t^m\pa_\phi^n( v- a)\right| 
& \le C\w^{3+\alpha} \quad \mbox{ on }(-R,R) \times \mathbb{S}^2 _{+}, \\
\sum_{l=0}^{3}\sum_{m,n\le 3}  \w^l  \left|\pa_\theta^l\pa_t^m\pa_\phi^n( v+ a)\right| 
& \le C\w^{3+\alpha} \quad \mbox{ on }(-R,R) \times \mathbb{S}^2 _{-}, 
\end{split}
\ee
where $\mathbb{S}^2_+ =\{0\le \theta<\pi/2\}$, $\mathbb{S}^2_- =\{\pi/2<\theta\le \pi\}$, 
and $C$ is a positive constant depending only on $\Lda, A_1, a, \alpha, R$. 
\end{thm}

\begin{proof} By Remark \ref{rem:homlize}, we can assume that $L\ln\w=0$.  
Given Lemmas \ref{lem:local-bound} and \ref{lem:monotonicity}, 
the proof of the desired result follows from the proof of Theorems 1.1 and 5.1 of Li-Tian \cite{LT}. 
Although not stated here, it should be noted that $(\Phi,v)$ is in fact smooth in the $t$ and $\phi$ variables. 
Away from the singular set of $\ln \w$, these results were established by Schoen-Uhlenbeck \cite{SU}. 
\end{proof}

Theorem \ref{thm:reg} gives the basic regularity and estimates of weak solutions $(\Phi, v)$ 
to the harmonic map equations \eqref{eq:main-unified-2}, up to a certain order. 
This, however, is not sufficient to adapt the method of Simon \cite{Simon83} concerning 
analysis of asymptotics. For this purpose we will establish, in Proposition \ref{prop:partial-t-linear} below, $L^\infty$-estimates for derivatives of $(\Phi, v)$ in terms of a linear form involving certain $L^2$-norms of derivatives, 
with constants depending on $\Lda $ and $A_1$ from \eqref{eq:bds-omega-1}.  If $\delta<\theta<\pi-\delta$ for some $\delta>0$, 
such linear estimates follow from standard elliptic theory for linear equations 
with smooth coefficients. Extra work is needed if $\theta $ is close to $0$ 
and $\pi$, due to the singular behavior of $u$. We will incorporate some ideas from Li-Tian \cite{LT}. 

Let us first make a simple but important observation. Note that the second equation of \eqref{eq:main-unified-2} 
may be regarded as a linear equation for $v$. Indeed, upon introducing the corresponding linear operator  
$\mathcal L=L+4\nabla_gu\cdot\nabla_g$, we have
$$\mathcal Lh=Lh-4\big(\nabla_g\ln\w-\nabla_g\Phi\big)\cdot\nabla_gh,$$ 
according to \eqref{eq:definition-Phi}. Notice that
$\Phi\in C^{3,\alpha}((-2,2) \times \mathbb{S}^2 )$ for any $\alpha\in (0,1)$ by Theorem \ref{thm:reg},
and $\nabla_g\ln\w$ is singular at the north and south poles $N$, $S$. We begin with two preparatory lemmas.
The first rephrases Lemma 4.4 of \cite{LT} in our specific setting\footnote{The $\Delta \rho_\infty$ in (4.38) of \cite{LT} should be $\nabla \rho_\infty$, see their proof on page 20.  In addition, Lemma 4.4 was misquoted as Lemma 4.3 at the start of its proof.}. 

\begin{lem}
\label{lem:LT} 
Let $(\Phi,v)\in H^{1}((-2, 2)\times \mathbb{S}^2)$ be a weak solution of \eqref{eq:main-unified-2} 
satisfying \eqref{eq:ubd-unified-1},  \eqref{eq:ea-unified-1}, and \eqref{eq:trace-unified-1}, and let 
$f$ be continuous on $(-2,2)\times \mathbb{S}^2_+ $ with 
$$|f(t,\phi,\theta)|\le K(\sin \theta)^{1+\alpha}\quad\text{ on }(-2 ,2)\times \mathbb{S}^2_+,$$ 
for some constants $\alpha\in (0,1)$ and $K>0$. Suppose that  $h\in C^2((-2,2) \times \mathbb{S}^2_+)$ 
is a solution of the equation
\[
Lh +4\nabla_gu\cdot  \nabla_{g}h =f \quad \text{ on }(-2,2)\times \mathbb{S}^2_+, 
\]
and satisfies $h=0$ on $(-2,2) \times \{N\}$. Then
\begin{equation}\label{akfniiqh}
|h(t,\phi,\theta)|\le C(\sin \theta)^{3+\alpha}
\big(K+\|h\|_{L^\infty((-2,2)\times \mathbb S^2_+)}\big)\quad \mbox{ on }(-1,1)\times \mathbb{S}^2_+, 
\end{equation}
where $C$ is a positive constant depending only on $\alpha$, $A_1$,
and the $C^1$-norm of $\Phi$ on $(-2,2) \times \mathbb{S}^2_+$.
\end{lem}

The proof of Lemma \ref{lem:LT} is based on a barrier function argument. 
The coefficient 4 in the gradient term is responsible for the constraint on the largest possible order of decay near the north pole, as expressed in \eqref{akfniiqh}.  We refer to Lemma 4.4 of \cite{LT} for further details. 

\begin{lem}\label{lem:w2p-}  
Let $h\in C^2((-1,1)\times \mathbb{S}^2_+)$ be a  solution of 
\[
Lh +(\sin \theta)^{-1} b \cdot \nabla_{g} h=(\sin \theta)^{-2} f \quad \text{ on }(-1,1)\times \mathbb{S}^2_+.
\]

$\mathrm{(i)}$ Assume that $b\in L^\infty((-1,1)\times \mathbb{S}^2_+)$ 
and $f\in L^p((-1,1)\times \mathbb{S}^2_+)$ for some $p>3$. 
Then for any $-1/2<t<1/2$ and  $P=(\phi,\theta)\in \mathbb{S}^2_+$ with $0<\theta<\pi/4$ it holds that
\[
|\nabla_{g} h(t,P)|\le C (\sin \theta )^{-1}\big(\|h\|_{L^\infty(Q_{\delta}(t,P))}
+(\sin \theta )^{-\frac{3}{p}}\|f\|_{L^p(Q_{\delta}(t,P))}\big),
\]
where $\delta= \frac{1}{2}\min\{dist_{g_0}(P, N),1/2\}$  
and $C$ is a positive constant depending only on $p$ and 
the $L^\infty$-norm of $b$ on $(-1,1) \times \mathbb{S}^2_+$. 

$\mathrm{(ii)}$ Assume that $b, f\in C^{k,\beta}((-1,1)\times \mathbb{S}^2_+)$ for some $k\ge 0$ 
and $\beta \in (0,1)$. Then for any $-1/2<t<1/2$ and
$P=(\phi,\theta)\in \mathbb{S}^2_+$ with $0<\theta<\pi/4$ it holds that
\[
|\nabla_{g}^{l} h(t,P)|\le C (\sin \theta )^{-l}\big(\|h\|_{L^\infty(Q_{\delta}(t,P))}
+\|f\|^*_{C^{k,\beta}(Q_{\delta}(t,P))}\big),\quad l=1,\dots, k+2,
\]
where $\delta= \frac{1}{2}\min\{dist_{g_0}(P, N),1/2\}$, the notation $\|\cdot\|^*_{C^{k,\beta}(Q_{\delta}(t,P))}$ 
indicates scaled 
H\"older norms, and $C$ is a positive constant depending only on $k$, $\beta$,  
and the $C^{k,\beta}$-norm of $b$ on $(-1,1)\times \mathbb{S}^2_+$.  
\end{lem}

\begin{proof} The first conclusion follows from the scaled $W^{2,p}$ estimates for linear elliptic equations, and the second one follows from scaled Schauder estimates. 
\end{proof}

The next result is a preliminary form of a portion of the estimates needed for the asymptotic analysis relying on Simon's method. Notice that the right-hand side of the estimates depend on all derivatives of $v$. Ultimately, estimates which effectively only depend on $t$-derivatives are needed for applications, and will be achieved subsequently.

\begin{prop}\label{prop:v-linear} 
Let $(\Phi,v)\in H^{1}((-2, 2)\times \mathbb{S}^2)$ be a weak solution of \eqref{eq:main-unified-2} 
satisfying \eqref{eq:ubd-unified-1},  \eqref{eq:ea-unified-1}, and \eqref{eq:trace-unified-1}
for some positive constants $\Lambda$ and $a$.  Then 
for any $0<\alpha<1$,  $1\le R_1<R_2<2$, and $l,m,n=0,1,2,3$
the following estimates hold
\begin{align*}
 \left|\pa_\theta^l\pa_t^m\pa_\phi^n(v- a)\right|  &\le C\w^{3+\alpha-l } 
 \|\w^{-2}  \nabla_{g} v \|_{L^2((-R_2,R_2) \times \mathbb{S}^2)}
\quad \mbox{ on }(-R_1,R_1) \times \mathbb{S}^2 _{+}, \\
 \left|\pa_\theta^l\pa_t^m\pa_\phi^n(v+ a)\right|  &\le C\w^{3+\alpha-l } 
 \|\w^{-2}  \nabla_{g} v \|_{L^2((-R_2,R_2) \times \mathbb{S}^2)}
\quad \mbox{ on }(-R_1,R_1) \times \mathbb{S}^2 _{-}, \end{align*}
where $C$ is a positive constant depending only on $a,\alpha,\Lda, A_1, R_1$, and $R_2$. 
\end{prop}

\begin{proof} We give arguments only for $\mathbb S^2_+$, as the proof for $\mathbb{S}^2_-$ is analogous. 
By Corollary 4.1 of \cite{LT}, we have the weighted Poincar\'{e} inequality
\[
\|\w^{-2} ( v-a) \|_{L^2((-R_2,R_2) \times (\mathbb{S}^2\cap \{\theta<3/4\}))}   
\le C \|\w^{-2}  \nabla_{g} v \|_{L^2((-R_2,R_2) \times (\mathbb{S}^2\cap \{\theta<3/4\}))},  
\]
where $C$ is a positive constant depending only on $A_1$ and $R_2$. 
We now view the second equation of \eqref{eq:main-unified-2} as a linear equation of $v-a$. 
Let $\om \subset \subset  (-R_2,R_2)\times (\mathbb{S}^2\cap \{\theta<5/8\})$ and $0<\alpha<1$.
Then Theorem \ref{thm:reg} above and Theorem 1 of \cite{N} yield
\[
\sup_{\om} \w^{-(3+\alpha)} |(v-a)| \le C \|\w^{-2} ( v-a) \|_{L^2((-R_2,R_2) \times (\mathbb{S}^2\cap \{\theta<3/4\}))} ,
\]
where $C$ is a positive constant depending only on $a,\alpha,\Lda$, and  $\om$.  
We point out that this estimate is proved with a De Giorgi iteration. 
The difficulty arises from the singularity of the coefficient $\nabla_gu$ 
in the second equation of \eqref{eq:main-unified-2}. 
We refer to the proof of Theorem 1 in \cite{N} for further details.
The higher order estimates follow from Lemma \ref{lem:w2p-}.  
\end{proof}

We are ready to establish the main estimate in this section. These pointwise bounds for the radial derivatives will play a fundamental role in the asymptotic analysis near black holes.

\begin{prop}\label{prop:partial-t-linear} 
Let $(\Phi,v)\in H^{1}((-2, 2)\times \mathbb{S}^2)$ be a weak solution of \eqref{eq:main-unified-2} 
satisfying \eqref{eq:ubd-unified-1},  \eqref{eq:ea-unified-1}, and \eqref{eq:trace-unified-1}
for some positive constants $\Lambda$ and $a$.  
Suppose that $\nabla_{g_0} \pa_t \ln \w = 0$ and $\pa_t (L\ln \w)=0$. Then on $[-1,1]\times \mathbb{S}^2$ the following
estimate holds
\begin{align*}
& |\pa_t \Phi| +|\pa_t v|  \le C\Big(\int_{-2}^{ 2}\int_{\mathbb{S}^2} \big [|\pa_t \Phi|^2
+e^{4u} ( |\pa_t v|^2 +|\pa_t\ln \w|^2 |\nabla_{g_0} v|^2 )\big] \, d  vol_{g_0}  d  t \Big)^{1/2}, 
\end{align*}
and for any $\alpha\in (0,1)$, $k=1,2,3$, and $j=0,1,2,3$ the higher order derivatives satisfy
\begin{align*}
& |\pa_t^k \nabla_{g_0}^j \Phi| +\w^{j-3-\alpha}|\pa_t^k \nabla_{g_0}^j v| \\
&\qquad \le  C\Big( \int_{-2}^{2} \int_{\mathbb{S}^2}
\big( | \pa_t \Phi|^2 +e^{4u} |\pa_t v|^2\big) \, d  vol_{g_0}  d  t \Big)^{1/2}
 \\ 
&\qquad\qquad +C \|\pa_t\ln \w \|_{C^5 ((-2, 2)\times \mathbb{S}^2) } 
\Big(\int_{-2}^{ 2}\int_{\mathbb{S}^2} e^{4u}  |\nabla_{g_0} v|^2 \, d  vol_{g_0}  d  t \Big)^{1/2}, 
\end{align*}
where $C$ is a positive constant depending only on 
$a,\alpha,\Lda$, and $A_1$.
\end{prop}

\begin{proof}  We will write $(\Phi_t,v_t)=(\pa_t \Phi, \pa_t v)$ for brevity. 
Since  $\pa_t (L\ln \w)=0$, differentiating equations \eqref{eq:main-unified-2} with respect to  $t$ produces
\begin{align}\label{eq:diff_t-pre}\begin{split}
L \Phi_t-4  e^{4u} \nabla_{g} v \cdot \nabla_{g} v_t -8 e^{4u} |\nabla_{g} v|^2 \Phi_t
+8 e^{4u} |\nabla_{g} v|^2\pa_t \ln \w &=0,\\
Lv_t +4\nabla_{g}  u \cdot  \nabla_{g} v_t+4 \nabla_{g}\Phi_t\cdot\nabla_{g} v
-4 \nabla_{g} \pa_t \ln \w \cdot\nabla_{g} v&=0.
\end{split}\end{align}
For the last term in the first equation, write $|\nabla_{g} v|^2=|\partial_tv|^2+|\nabla_{g_0} v|^2$ 
and move the portion with $|\nabla_{g_0} v|^2$ to the right-hand side. 
Moreover, note that in the last term of the second equation, 
the portion involving $\nabla_{g_0} v$ is absent since $\nabla_{g_0} \pa_t \ln \w=0$. 
Equations \eqref{eq:diff_t-pre} may then be rewritten as  
\begin{equation} \label{eq:diff_t}
\begin{split}
L \Phi_t  +b_{11} \cdot \nabla_{g} v_t + b_{12} \Phi_t +b_{13}v_t&=b_0,
 \\
 Lv_t +4\nabla_{g}  u \cdot  \nabla_{g} v_t +b_{21} \cdot \nabla_{g} \Phi_t  + b_{22} v_t &=0,
\end{split}
\ee
where \begin{align*}
&b_{11}= -4  e^{4u} \nabla_{g} v ,\qquad   b_{12}= -8 e^{4u} |\nabla_{g} v|^2, \\&  
b_{13}= 8 e^{4u}  \pa_t v \pa_t \ln \w   , \qquad b_0=  -8 e^{4u} |\nabla_{g_0} v|^2  \pa_t \ln \w, 
\end{align*}
and 
\[
b_{21}= 4 \nabla_{g} v , \qquad  b_{22}= -4\pa_t^2 \ln \w.
\]
By Theorem  \ref{thm:reg} and the conditions on $\w$, we have 
\begin{equation} \label{eq:partial--structure}
\begin{split}
&|b_{11}|\le C\w^{\alpha} e^{2u}, \qquad |b_{12}|\le  C \w^{2\alpha}, \qquad |b_{13}|\le  C\w^{\alpha} e^{2u},  \\
&|b_0|\le C\w^{\alpha}  |\pa_t \ln \w| e^{2u} |\nabla_{g_0} v| ,   \qquad |b_{21}|\le C\w^{2+\alpha}, 
\qquad |b_{22}|\le C. 
\end{split}
\end{equation}

\emph{Step 1.  Energy estimates.}  
Fix a constant $t_1\in (1,2)$ and a smooth cutoff function  $\eta(t)\in C^2_c(-2, 2)$  
with $\eta=1$ on $(-t_1,  t_1)$. Take $\eta^2 \Phi_t$ as a test function for the first equation of \eqref{eq:diff_t}. 
By the Cauchy inequality and \eqref{eq:partial--structure}, for any $\va>0$ we have
\begin{align*}
&\int_{-2}^{ 2}\int_{\mathbb{S}^2}  |\nabla_{g} \Phi_t|^2 \eta^2 \, d  vol_{g_0}  d  t  
\le   \va \int_{-2}^{ 2}\int_{\mathbb{S}^2} e^{4u} |\nabla_{g} v_t|^2 \eta^2 \, d  vol_{g_0}  d  t   \\
&\qquad \qquad + C(\va)  \int_{-2}^{ 2}\int_{\mathbb{S}^2}  \big( |\Phi_t|^2 +e^{4u}|v_t|^2
+e^{4u}|\pa_t\ln \w|^2  |\nabla_{g_0} v|^2\big)\, d  vol_{g_0}  d  t. 
\end{align*}
We may write the second equation of \eqref{eq:diff_t} as 
$$ e^{-4u} \mathrm{div}_g(e^{4u} \nabla_{g} v_t) -\partial_tv_t+b_{21} \cdot \nabla_{g} \Phi_t  + b_{22} v_t =0,$$
and take $\eta^2 e^{4u} v_t$ as a test function. Again, the Cauchy inequality and \eqref{eq:partial--structure} imply
\begin{align*}
\int_{-2}^{ 2}\int_{\mathbb{S}^2}  e^{4u}|\nabla_{g} v_t|^2 \eta^2 \, d  vol_{g_0}  d  t 
& \le   \va \int_{-2}^{ 2}\int_{\mathbb{S}^2}  |\nabla_{g} \Phi_t|^2 \eta^2 \, d  vol_{g_0}  d  t\\&\qquad  
+ C(\va)  \int_{-2}^{ 2}\int_{\mathbb{S}^2}  e^{4u}  |v_t|^2 \, d  vol_{g_0}  d  t.
\end{align*}
Choosing $\va=1/2$ and summing the above two inequalities produces
\begin{equation}\label{eq:tan-energy}
\int_{-t_1}^{t_1}\int_{\mathbb{S}^2}  (|\nabla_{g} \Phi_t|^2)
+ e^{4u} |\nabla_{g} v_t|^2)  \, d  vol_{g_0}  d  t\leq C\Xi^2,
\end{equation}
where for brevity  we have set
\[
\Xi= \Big(\int_{-2}^{ 2}\int_{\mathbb{S}^2} \big [|\Phi_t|^2
+e^{4u} |v_t|^2 +e^{4u} |\pa_t\ln \w|^2 |\nabla_{g_0} v|^2 \big] \, d  vol_{g_0}  d  t \Big)^{1/2}.
\]

\emph{Step 2. $H^{2}$-estimates and $C^0$-estimates.} 
Observe that \eqref{eq:diff_t} and \eqref{eq:partial--structure} yield
\begin{align*} 
|L \Phi_t|&\le C(e^{2u}|\nabla_{g} v_t| + |\Phi_t| +e^{2u}|v_t|+e^{2u}|\pa_t \ln \w|  |\nabla_{g_0} v|),
 \\
|Lv_t|&\le C(e^u|\nabla_{g} v_t| +|\nabla_{g} \Phi_t|  + |v_t|).
\end{align*}
and by \eqref{eq:tan-energy} it follows that
\begin{align*}
\int_{-t_1}^{t_1}\int_{\mathbb{S}^2} (|L \Phi_t|^2 + |L v_t|^2 ) \, d  vol_{g_0}  d  t 
\le  C \Xi^2. 
\end{align*}
Next, fix a constant $t_2\in (1,t_1)$ and note that the interior $H^2$-estimates then give
\begin{align*}
\|( \Phi_t,  v_t)\|_{H^{2}([-t_2, t_2]\times \mathbb{S}^2)} 
&\le C\big(
\|( \Phi_t,  v_t)\|_{L^2([-t_1, t_1]\times \mathbb{S}^2)}
+\|( L\Phi_t,  Lv_t)\|_{L^2([-t_1, t_1]\times \mathbb{S}^2)}\big)\\
&\le C\Xi. 
\end{align*}
Therefore, by the Sobolev embedding theorem we obtain 
\be \label{eq:L-infty}
\max_{[-t_2, t_2]\times \mathbb{S}^2} (|\Phi_t|+ | v_t|)  \le C\Xi. 
\ee
This implies the first inequality in the proposition. 

\medskip 

\emph{Step 3.}  By Step 2 and the Sobolev embedding theorem,  we also have 
\be \label{eq:L-six}
\|\nabla_{g} v_t\|_{L^6([-t_2, t_2]\times \mathbb{S}^2)} \le C\Xi. 
\ee   
Now write the first equation of \eqref{eq:diff_t} as 
$$L\Phi_t=(\sin \theta)^{-2} f,$$
where
\[
f= \sin^2\theta (-b_{11} \cdot \nabla_{g} v_t - b_{12} \Phi_t -b_{13}v_t +b_0).
\] 
Fix a constant $t_3\in (1,t_2)$, let $P=(\phi,\theta)\in \mathbb S^2_+\setminus\{N\}$, and take 
$\delta<t_2-t_3$ with $\delta<\text{dist}_{g_0}(P,N)/2$. Then for any $t\in [-t_3,t_3]$, we find by Lemma \ref{lem:w2p-} $(\mathrm{i})$ with $p=6$ that 
\[
|\nabla_{g} \Phi_t(t,P)|
\le C (\sin \theta )^{-1}\big(\|\Phi_t\|_{L^\infty(Q_{\delta}(t,P))}
+(\sin \theta )^{-\frac12}\|f\|_{L^6(Q_{\delta}(t,P))}\big).
\]
In addition, take any $\alpha\in (0,1)$ and observe that \eqref{eq:partial--structure} shows
\[
|f|\le C\w^\alpha (|\nabla_{g} v_t| +|\Phi_t| +|v_t| +|(\pa_t \ln \w)e^{2u}\nabla_{g_0} v|),
\] 
so that \eqref{eq:L-infty} and \eqref{eq:L-six} yield
$$\|f\|_{L^6(Q_{\delta}(t,P))}\le C\w^\alpha 
\big(\Xi+\|(\pa_t \ln \w)e^{2u}\nabla_{g_0} v\|_{L^\infty ([-t_2, t_2]\times \mathbb{S}^2) }\big).$$
Therefore, by choosing $\alpha=1/2$ we obtain
\begin{equation}\label{anhfpq}
|\nabla_{g} \Phi_t(t,P) |
\le C\w^{-1} \big(\Xi+\|(\pa_t \ln \w)e^{2u}\nabla_{g_0} v\|_{L^\infty ([-t_2, t_2]\times \mathbb{S}^2) }\big). 
\end{equation}

In the following, $\alpha\in (0,1)$ will again be arbitrary. 
By Proposition \ref{prop:v-linear}, for any $t\in  [-t_2, t_2]$ we have 
\begin{align*}
|v_t|&\le C\w^{3+\alpha}\|\w^{-2}\nabla_gv\|_{L^2([-t_1,t_1]\times \mathbb S^2)}\\
&\le C\w^{3+\alpha}\big(\|e^{2u}\partial_tv\|_{L^2([-t_1,t_1]\times \mathbb S^2)}
+\|e^{2u}\nabla_{g_0}v\|_{L^2([-t_1,t_1]\times \mathbb S^2)}\big).\end{align*}
Notice that the first term of the second line in this inequality is a part of $\Xi$, and therefore
\be\label{eq:estimate-v-t}
|v_t|
\le C\w^{3+\alpha}\big(\Xi+\|e^{2u}\nabla_{g_0}v\|_{L^2([-2,2]\times \mathbb S^2)}\big).\ee
Similarly, for any $t\in  [-t_2, t_2]$ it holds that
\be\label{eq:estimate-v-0}|\nabla_{g_0}v|
\le C\w^{2+\alpha}\big(\Xi+\|e^{2u}\nabla_{g_0}v\|_{L^2([-2,2]\times \mathbb S^2)}\big).\ee
Using this in \eqref{anhfpq} then shows that 
\be\label{eq:estimate-Phi-t}
|\nabla_{g} \Phi_t |
\le C\w^{-1} \big(\Xi+\|\pa_t \ln \w\|_{L^\infty ([-2, 2]\times \mathbb{S}^2) }
\|e^{2u}\nabla_{g_0}v\|_{L^2([-2,2]\times \mathbb S^2)}\big),
\ee
for any $t\in [-t_3,t_3]$. We point out that the decay order $3+\alpha$ in \eqref{eq:estimate-v-t} is optimal. 
However, \eqref{eq:estimate-v-t} is not the desired estimate due to the absence of $\partial_t \ln \omega$ with 
the $L^2$-norm of $e^{2u}\nabla_{g_0}v$. 

Next, write the second equation of \eqref{eq:diff_t} as 
$$Lv_t +4\nabla_{g}  u \cdot  \nabla_{g} v_t =-b_{21} \cdot \nabla_{g} \Phi_t  -b_{22} v_t.$$
By combining \eqref{eq:estimate-Phi-t} with the estimate of $b_{21}$ in \eqref{eq:partial--structure}, 
for any $t\in [-t_3,  t_3]$ we get
\begin{align}\label{eq:estimate-Phi-1-nabla}
|b_{21} \cdot \nabla_{g} \Phi_t  |
\le C \w^{1+\alpha}
\big(\Xi+\|\pa_t \ln \w\|_{L^\infty ([-2, 2]\times \mathbb{S}^2) }
\|e^{2u}\nabla_{g_0}v\|_{L^2([-2,2]\times \mathbb S^2)}\big). 
\end{align}
Also in this domain, \eqref{eq:estimate-v-t} implies that 
\begin{align*}
|b_{22}v_t| = 4 |\pa_t^2 \ln \w| |v_t| \le C \w^{3+\alpha}
\big(\Xi+\|\pa^2_t \ln \w\|_{L^\infty ([-2, 2]\times \mathbb{S}^2) }
\|e^{2u}\nabla_{g_0}v\|_{L^2([-2,2]\times \mathbb S^2)}\big), 
\end{align*}
where $C$ is a constant depending only on $\Lda, A_1 $, and $\alpha$. 
Moreover, let $t_4\in (1,t_3)$ and note that 
Lemma \ref{lem:LT} together with \eqref{eq:L-infty} produces
\be\label{eq:estimate-v-t-improved}
|v_t|\le  C \w^{3+\alpha} \big(\Xi+\|\pa_t \ln \w\|_{C^1 ([-2, 2]\times \mathbb{S}^2) }
\|e^{2u}\nabla_{g_0}v\|_{L^2([-2,2]\times \mathbb S^2)} \big)
\ee
for any $t\in [-t_4,  t_4]$, where again $C$ is a positive constant depending only on $\Lda, A_1$,  and $\alpha$. 
Observe that \eqref{eq:estimate-v-t-improved}, rather than \eqref{eq:estimate-v-t}, is the desired estimate for $v_t$. 

In order to estimate derivatives of $v_t$, write the second equation of \eqref{eq:diff_t} as 
$$Lv_t +4\nabla_{g}  u \cdot  \nabla_{g} v_t =-(\sin \theta)^{-2}(\sin \theta)^{2}(b_{21} 
\cdot \nabla_{g} \Phi_t  +b_{22} v_t).$$
Then \eqref{eq:estimate-Phi-1-nabla} and \eqref{eq:estimate-v-t-improved} show that the right-hand side is controlled for $t\in [-t_4,  t_4]$ by
$$(\sin \theta)^{2}|b_{21} \cdot \nabla_{g} \Phi_t  +b_{22} v_t|\le 
C \w^{3+\alpha} \big(\Xi+\|\pa_t \ln \w\|_{C^1 ([-2, 2]\times \mathbb{S}^2) }
\|e^{2u}\nabla_{g_0}v\|_{L^2([-2,2]\times \mathbb S^2)}\big).$$
We may now fix $t_5\in (1,t_4)$ and apply the first part of Lemma \ref{lem:w2p-} to obtain 
\be\label{eq:estimate-v-t-g}
 |\nabla_{g} v_t|\le C \w^{2+\alpha}  \big(\Xi+\|\pa_t \ln \w\|_{C^1 ([-2, 2]\times \mathbb{S}^2) }
 \|e^{2u}\nabla_{g_0}v\|_{L^2([-2,2]\times \mathbb S^2)}\big),
\ee
for any $t\in [-t_5,t_5]$. This gives the desired estimate for $\nabla_{g_0}v_t$, but not $\partial_tv_t$. 

In order to improve the derivative estimate for $\Phi_t$, write the first equation of \eqref{eq:diff_t} as 
$$L\Phi_t=\widetilde f, \quad\quad\quad \tilde  f= b_0- b_{11} \cdot \nabla_{g} v_t - b_{12} \Phi_t -b_{13}v_t.$$ 
Combining \eqref{eq:partial--structure}, \eqref{eq:estimate-v-0}, \eqref{eq:estimate-v-t-improved}, 
and \eqref{eq:estimate-v-t-g} yields 
\begin{align*}
|\tilde  f| \le C\big(\Xi+\|\pa_t \ln \w\|_{C^1 ([-2, 2]\times \mathbb{S}^2) }
\|e^{2u}\nabla_{g_0}v\|_{L^2([-2,2]\times \mathbb S^2)}\big),
\end{align*}
for any $t\in [-t_5,t_5]$. Fix $t_6\in (1,t_5)$, and observe that the interior $W^{2,p}$-estimates along with 
\eqref{eq:L-infty} produce
\begin{align}\label{eq:estimate-Phi-t-improved}\begin{split}
|\nabla_{g}\Phi_t  |& \le C\big( \|\Phi_t\|_{L^\infty([-t_5,t_5]\times\mathbb S^2)}
+\|\tilde f\|_{L^\infty([-t_5,t_5]\times\mathbb S^2)}\big)\\
&\le C\big(\Xi+\|\pa_t \ln \w\|_{C^1 ([-2, 2]\times \mathbb{S}^2) }
\|e^{2u}\nabla_{g_0}v\|_{L^2([-2,2]\times \mathbb S^2)}\big),
\end{split}\end{align} 
for any $t\in [-1/2,  1/2]$. This is the desired estimate for $\nabla_{g}\Phi_t$; compare \eqref{eq:estimate-Phi-t-improved} 
with \eqref{eq:estimate-Phi-t}.

In summary, we have derived the desired estimates for $\pa_t\Phi$, $\pa_tv$, $\pa_t\nabla_{g_0}v$, 
and $\pa_t\nabla_{g}\Phi$, 
as expressed in \eqref{eq:L-infty}, \eqref{eq:estimate-v-t-improved}, \eqref{eq:estimate-v-t-g}, 
and \eqref{eq:estimate-Phi-t-improved}, respectively.
We point out that \eqref{eq:estimate-v-t-g} does not yield the appropriate estimate for $\partial_t^2v$. 
By a bootstrap argument using the second part of Lemma \ref{lem:w2p-}, 
one can establish the necessary estimates of higher order derivatives. 
We refer to Section 5 of Li-Tian \cite{LT} for more details concerning the bootstrap procedure.   
\end{proof} 
 
In a similar manner, we are able to estimate difference of two solutions. 

\begin{prop}\label{prop:difference-linear}  
Let  $(\Phi,v), (\tilde  \Phi, \tilde v)\in H^{1}((-2, 2)\times \mathbb{S}^2)$ be weak solutions 
of \eqref{eq:main-unified-2} satisfying \eqref{eq:ubd-unified-1},  \eqref{eq:ea-unified-1}, 
and \eqref{eq:trace-unified-1} for some positive constants $\Lambda$ and $a$.  
Then on $[-1,1]\times \mathbb{S}^2$, for any $\alpha\in (0,1)$ and $j,k=0,1,2,3$, it holds that
\begin{align*}
|\pa_t^k \nabla_{g_0}^j (\Phi&-\tilde \Phi) | +\w^{j-3-\alpha}|\pa_t^k \nabla_{g_0}^j (v-\tilde v) |\\
& \le C\Big( \int_{-2}^{2} \int_{\mathbb{S}^2} |\Phi-\tilde\Phi |^2 
+\w^{-4} | v-\tilde v|^2 \, d  vol_{g_0}  d  s \Big)^{1/2},
\end{align*}
where $C$ is a positive constant depending only on 
$a,\alpha, \Lda$, and $A_1$.
\end{prop} 

\begin{proof} Set $(w_1,w_2)= (\Phi-\tilde \Phi, v-\tilde v)$ and $\tilde u=\tilde \Phi -\ln \w$. 
By \eqref{eq:main-unified-2}, a simple subtraction yields 
\begin{align*}
L w_1 -2 e^{4u} |\nabla_{g} v|^2 +2 e^{4\tilde u} |\nabla_{g} \tilde v|^2  &=0, \\
Lw_2+4\nabla_{g}u\cdot   \nabla_{g} v-4\nabla_{g}\tilde u\cdot   \nabla_{g} \tilde v&=0.
\end{align*}
Furthermore, a straightforward computation produces
$$e^{4u} |\nabla_{g} v|^2 - e^{4\tilde u} |\nabla_{g} \tilde v|^2
=e^{4u}\nabla_g(v+\tilde v)\cdot\nabla_gw_2+e^{4\tilde u}|\nabla_g\tilde v|^2(e^{4w_1}-1),$$
and 
$$\nabla_{g}u\cdot   \nabla_{g} v-\nabla_{g}\tilde u\cdot   \nabla_{g} \tilde v
=\nabla_{g}u\cdot   \nabla_{g} w_2+\nabla_{g}\tilde v\cdot   \nabla_{g} w_1.$$
Therefore
\begin{align} \label{eq:substract}\begin{split}
L w_1 -2e^{4u}\nabla_g(v+\tilde v)\cdot\nabla_gw_2-2e^{4\tilde u}|\nabla_g\tilde v|^2(e^{4w_1}-1)&=0, \\
 Lw_2 +4\nabla_{g}u\cdot   \nabla_{g} w_2+4\nabla_{g}\tilde v\cdot   \nabla_{g} w_1&=0.
\end{split}
\end{align}
We note that the term involving $e^{4w_1}-1$ in the first equation can be viewed as a linear term in $w_1$. 
Hence, the equation \eqref{eq:substract} has a similar structure as \eqref{eq:diff_t-pre}, 
with the terms involving $\ln\w$ absent here.  
We may proceed similarly as in the proof of Proposition \ref{prop:partial-t-linear}; details are omitted. 
\end{proof}

\section{Singular Harmonic Maps on the Sphere}
\label{s:classification}

In this section, we will classify singular harmonic maps on the sphere as well as their Jacobi fields. 
The harmonic map equations from open sets of $\mathbb{S}^2 $ to $ \mathbb{H}^2$ can be written as 
\begin{align}\label{ai9tgjiqihq}
\begin{split}
\Delta_{g_0}   u -2 e^{4  u}  |\nabla_{g_0}    v|^2 &=0, \\
\Delta_{g_0}   v +4 \nabla_{g_0}   u \cdot \nabla_{g_0}   v &=0, 
\end{split}
\end{align} 
or equivalently, 
\begin{align*}
\mathrm{div}_{g_0}(\nabla_{g_0}   u -2 e^{4  u} v  \nabla_{g_0} v  ) &=0, \\
\mathrm{div}_{g_0}( e^{4  u}\nabla_{g_0}   v ) &=0. 
\end{align*} 
Let $a,\Lambda > 0$ be constants and set $\w=\sin \theta$. It will be assumed that $(u+\ln\omega,v)\in H^1(\mathbb{S}^2)$ satisfies
\begin{equation} \label{eq:main-s2-1b} 
|u+\ln\w|\leq\Lambda\quad\quad \text{ on}\quad\mathbb{S}^2,
\end{equation} 
\begin{equation}\label{qiiqihju}
\int_{\mathbb{S}^2}\left(|\nabla_{g_0}(u+\ln\w)|^2 +\w^{-4}|\nabla_{g_0} v|^2\right) d  vol_{g_0}<\infty,
\end{equation}
and
\begin{equation} \label{eq:main-s2-1c}
   v(N)=a \quad\text{ and }\quad   v(S)= -a \quad\text{ in the trace sense}.
\end{equation}
Under the additional hypothesis of axisymmetry, meaning that the map is independent of the coordinate $\phi$, by analyzing the system of ODE all such solutions are explicitly given and parameterized by two parameters $a$ and $b$. 
When $b=0$, this corresponds to the extreme Kerr near horizon geometry map. In the case that the harmonic map arises from the near horizon limit of a smooth axisymmetric stationary vacuum spacetime, this result is known to the physics community; see the survey article by Kunduri-Lucietti \cite[Theorem 4.3]{KunduriLucietti} and the references therein. Similarly, when the harmonic map is assumed to be axisymmetric, Chru\'sciel-Li-Weinstein \cite[Appendix B]{CLW} obtain the same conclusion. Below, we will show that even without the symmetry assumption, the same conclusion holds.

\begin{prop}\label{prop:classification}  
Consider a weak solution of \eqref{ai9tgjiqihq} with $(u+\ln\omega,v)\in H^1(\mathbb{S}^2)$
satisfying  \eqref{eq:main-s2-1b}, \eqref{qiiqihju}, and \eqref{eq:main-s2-1c}. Then, $(u,v)=(u_{a,b},v_{a,b})$ where
\begin{align*}
u_{a,b}(\theta)&= -\ln \sin \theta -\frac{1}{2} \ln \frac{2a \sqrt{1-b^2}}{1+\cos^2 \theta+2b \cos \theta}, \\
v_{a,b}(\theta)&= a\cdot \frac{b+b \cos ^2 \theta +2\cos \theta}{1+ \cos ^2 \theta +2b \cos \theta},
\end{align*}
for some constant $b\in (-1,1)$.  
\end{prop}

\begin{proof}  Set $U=u+\ln\w$, and observe that the map $(U,v)$ may be viewed as $t$-independent on 
$(-2,2)\times\mathbb{S}^2$, so that it is a weak solution of 
\eqref{eq:main-unified-2} satisfying \eqref{eq:ubd-unified-1}, \eqref{eq:ea-unified-1}, and \eqref{eq:trace-unified-1}. Therefore, Theorem \ref{thm:reg} implies that $(U,v)\in C^3(\mathbb{S}^2)$.

We first work under the additional hypothesis of axisymmetry. The harmonic map $( u, v)$ is now regular in 
$\mathbb{S}^2 \setminus \{N, S\}$, and by assumption only depends on $\theta$.  For brevity, write $'=\frac{ d }{ d  \theta}$ and note that the harmonic map equations become
\begin{equation*}
[\sin \theta (u'-2e^{4u} vv') ]'=0,\quad\quad
[\sin \theta e^{4u} v' ]' =0,
\end{equation*} 
for $0< \theta < \pi$.  This implies that 
\begin{equation}\label{eq:first-int}
u'-2e^{4u} vv'=\frac{c_1}{\sin \theta}, \quad\quad e^{4u} v'  =\frac{c_2}{\sin \theta},
\ee
for some constants $c_1$ and $c_2$. Inserting the second equation of \eqref{eq:first-int} into the first produces 
\[
U'=\frac{1}{\sin \theta} (\cos \theta +2 c_2 v +c_1). 
\]
Since $U\in C^3(\mathbb{S}^2)$ we must have $U'(0)=U'(\pi)=0$, so that
\begin{equation*}\label{eq:algebra}
1+2c_2a +c_1=0, \quad\quad -1-2c_2 a+c_1=0,
\end{equation*}
or rather $c_1=0$ and $c_2=-1/(2a)$.
The first equation of \eqref{eq:first-int} then yields
\begin{equation} \label{eq:u-v-nonlinear}
\frac{1}{4} e^{-4u} +v^2 = c_3^2,
\ee
for some constant $c_3>0$. 
Then solving for $e^{4u}$, and substituting the result into the second equation of \eqref{eq:first-int} shows that
\[
 \frac{1}{4 (c_3^2- v^2)} v'+\frac{1}{2a\sin \theta}= 0. 
\]
We may now integrate the above first order ODE to find
\[
\frac{1}{8c_3} \ln \frac{c_3+v}{c_3-v} +\frac{1}{4a} \ln \frac{1-\cos \theta}{1+\cos \theta}= \tilde{c}_4,\quad
\text{ which implies }\text{ } \quad \frac{c_3+v}{c_3-v} \Big( \frac{1-\cos \theta}{1+\cos \theta}\Big)^{\frac{2c_3}{a}}=c_4^2,
\]
for some constants $\tilde{c}_4$ and $c_4>0$. 
Next, notice that by sending $\theta \to0, \pi$ we conclude that $c_3=a$, and hence solving for $v$ gives rise to
\begin{equation} \label{eq:geodesic-2}
v=  a \frac{c_4^2 (1+\cos \theta )^2- (1-\cos \theta)^2}{c_4^2 (1+\cos \theta )^2+(1-\cos \theta)^2}.
\end{equation}
Now use \eqref{eq:u-v-nonlinear} to compute
\begin{align}\label{alirhioqhw}
\begin{split}
u&=-\frac{1}{4} \ln [4(a^2-v^2)]\\&
=-\frac12 \ln \frac{4ac_4(1+\cos \theta)(1-\cos \theta)}{c_4^2 (1+\cos \theta)^2+(1-\cos \theta)^2} \\&
=-\ln \sin \theta -\frac12 \ln\frac{4ac_4}{c_4^2 (1+\cos \theta)^2+(1-\cos \theta)^2}.
\end{split}
\end{align}
Lastly, by setting  $b= \frac{c_4^2-1}{c_4^2+1}$ the desired result follows, that is $(u,v)=(u_{a,b},v_{a,b})$.

In order to treat the general case, we will show that any harmonic map satisfying the hypotheses of this proposition must in fact be axisymmetric.  Denote $\Psi=(u,v)$, and let $f:\mathbb{S}^2 \rightarrow\mathbb{S}^2$ be a rotation of the $\phi$ angular variable. Then since $f$ is an isometry of the sphere, we have that $f^* \Psi:\mathbb{S}^2 \setminus\{N,S\}\rightarrow\mathbb{H}^2$ is also harmonic. It follows that the distance function $w=d_{\mathbb{H}^2}(\Psi,f^* \Psi)$ is subharmonic \cite[Lemma 2]{WeinsteinHadamard},
\begin{equation}\label{aoiohioqohy}
\Delta_{g_0}w\geq 0 \quad\text{ on }\mathbb{S}^2 \setminus\{N,S\}.
\end{equation}

We now claim that the distance function vanishes at the north and south poles. For brevity write $f^* \Psi=(u_*,v_*)$ and $U_*=u_* +\ln\w$, then using a basic formula \cite[page 1192]{Wei92} for the distance between points in the hyperbolic plane yields
\begin{align*}
\begin{split}
\cosh\left(2w\right)=&\cosh(2(u-u_*))
+2e^{2(u+u_*)}\left(v-v_*\right)^2\\
=&\cosh(2(U-U_*))
+\frac{2e^{2(U+U_*)}}{\sin^4 \theta}\left(v-v_*\right)^2.
\end{split}
\end{align*}
According the Theorem \ref{thm:reg}, and the fact that the north and south poles are fixed points for the rotation $f$, we have the expansions
\begin{equation*}
U, U_* =c_{\pm} +O(\sin\theta),\quad\quad v, v_* = \pm a+O(\sin^{3+\alpha} \theta)\quad \text{ as } \theta\rightarrow 0,\pi,
\end{equation*}
for some constants $\alpha\in(0,1)$ and $c_{\pm}$, where the $\pm$ corresponds to the limits $\theta\rightarrow 0,\pi$ respectively. The claim now follows, and since $w$ is a continuous function on a compact manifold it is uniformly bounded, that is
\begin{equation}\label{q9ijh9jq9h9qw}
w\leq C \quad\text{ on }\mathbb{S}^2,\quad\quad w(N)=w(S)=0.
\end{equation}

Properties \eqref{aoiohioqohy} and \eqref{q9ijh9jq9h9qw} combine to show that $w\equiv 0$. This may be verified by either using
the expansions for $v$, $v_*$ above and the fact that $(U,v)$ and $(U_*,v_*)$ are $C^3(\mathbb{S}^2)$ to conclude that $w\in C^2(\mathbb{S}^2)$ which allows for an application of the strong maximum principle, or with an adaptation of the argument in \cite[Lemma 8]{WeiDuke}.
It follows that $\Psi=f^* \Psi$ on $\mathbb{S}^2 \setminus\{N,S\}$, and since $f$ was an arbitrary rotation the harmonic map must be axisymmetric.
\end{proof}

\begin{rem}\label{rem:geodesic} 
If we let $s= \frac{1}{2} \ln \frac{1-\cos \theta}{1+\cos \theta}$, 
then the axisymmetric harmonic map equations (or equivalently the geodesic equation in the hyperbolic plane) becomes
\[
u'' -2e^{4u} (v')^2=0, \quad\quad  v'' +4 u' v'=0,
\]
where $'=\frac{ d }{ d  s}$ here. In the new arclength variable $s$, the expressions \eqref{eq:geodesic-2} and \eqref{alirhioqhw} may be rewritten as 
\[
u=\frac{1}{2}\ln\left[\frac{1}{2a}\cosh(-2s+\lambda)\right],\quad\quad v= a \tanh (- 2s+\lda), 
\]
where $\lda=\ln c_4 $. Moreover, the constant $b=\tanh \lda$ is associated with translations in the $s$-parameter. 
\end{rem}





We will now characterize Jacobi fields, or rather solutions of the homogeneous linearized harmonic map 
equations from $\mathbb{S}^2 \setminus \{N,S\}$ to $\mathbb{H}^2$. 
In particular, it will be shown that any solution admitting appropriate homogeneous Dirichlet boundary 
conditions at the north and south poles, must arise as the first variation of a 1-parameter family of
singular harmonic maps. This fact is important, as it implies that the integrability condition of 
Allard-Almgren \cite{AA} is satisfied, allowing for an improved rate of convergence to the tangent
map.

Let $( u_{a,b}, v_{a,b})$ be a solution of \eqref{ai9tgjiqihq} 
as in Proposition \ref{prop:classification}, for some constants $a>0$ and $b\in (0,1)$. 
Consider, for any $\varphi=(\varphi_1, \varphi_2)$  with $\varphi_1, \varphi_2\in C^2(\mathbb{S}^2 \setminus \{N,S\})$, 
the equations
\begin{equation} \label{eq:linearized-HM-operator}
\begin{split}
\mathcal{T}_{1} \varphi &= \Delta_{g_0}  \varphi_1 -8 e^{4 u_{a,b}} |\nabla_{g_0}    v_{a,b} |^2 \varphi_1 -4 e^{4 u_{a,b}} \nabla_{g_0}   v_{a,b} \cdot\nabla_{g_0}  \varphi_2,\\
\mathcal{T}_{2} \varphi&= e^{-4u_{a,b}}  \mathrm{div}_{g_0} (e^{4 u_{a,b}} \nabla_{g_0}  \varphi_2))   +4  \nabla_{g_0}  v_{a,b} \cdot\nabla_{g_0}  \varphi_1.
\end{split}  
\end{equation}
The operator $\mathcal{T}= (\mathcal{T}_1, \mathcal{T}_2)$ arises from the linearization of the 
harmonic map equation \eqref{ai9tgjiqihq} at the solution $( u_{a,b}, v_{a,b})$. 
Since $( u_{a,b}, v_{a,b})$ is independent of $\phi$, we have 
\begin{equation}\label{eq:general-linearized}
\mathcal{T}_1 = \frac{1}{\sin ^2\theta} \pa_\phi^2 +L_{a,b,1}, \qquad \mathcal{T}_2 = \frac{1}{\sin ^2\theta} \pa_\phi^2 + L_{a,b,2},
\end{equation}
where $L_{a,b}= (L_{a,b,1}, L_{a,b,2})$ is the axisymmetric linearized harmonic map operator at $(u_{a,b}, v_{a,b})$, given by
\begin{align*}
L_{a,b,1} \varphi&=\frac{1}{\sin \theta}(\sin \theta \varphi_1')' 
-8 e^{4u_{a, b}}| v_{a,b}'|^2 \varphi_1-4e^{4u_{a, b}} v_{a,b}'  \varphi_2', \\
L_{a,b,2} \varphi &= \frac{1}{\sin \theta}(\sin \theta \varphi_2')' 
+4 u_{a,b}'  \varphi_2' +4  v_{a,b}'  \varphi_1', 
\end{align*}
with $'=\frac{\partial }{\partial  \theta}$.  In \eqref{eq:linearized-HM-operator}, the coefficients of $\varphi_1$ are regular, but the coefficients of $\varphi_2$ have singularities at $N$ and $S$. Thus, additional requirements are to be imposed on $\varphi_2$ at $N$ and $S$.  In the following, we will study 
\begin{equation} \label{eq:linearized-HM-1}
\begin{split}
\mathcal{T}\varphi=0\quad\text{on }\mathbb{S}^2 \setminus \{N,S\},
\end{split}  
\end{equation}
with 
\begin{equation}\label{eq:linearized-HM-2}
\varphi_2(N)=   \varphi_2(S)= 0.
\end{equation}

We now demonstrate that the operator $\mathcal{T}$ is self-adjoint in appropriate $L^2$-spaces. Denote by $L^2(\mathbb{S}^2, e^{4u_{a,b}})$ the subspace of $L^2(\mathbb{S}^2)$ 
consisting of functions $f$ with the bounded norm 
\[
\|f\|_{L^{2}(\mathbb{S}^2, e^{4u_{a,b}} )}=\Big(\int_{\mathbb{S}^2} e^{4u_{a,b}}  f^2\,  d  vol_{g_0} \Big)^{1/2},
\] 
and by $ H_0^{1}(\mathbb{S}^2, e^{4u_{a,b}} )$ the closure of 
$C_c^\infty(\mathbb{S}^2\setminus \{N,S\})$ under the norm 
\[
\|f\|_{H^{1}(\mathbb{S}^2, e^{4u_{a,b}} )} 
=\Big(\int_{\mathbb{S}^2}e^{4u_{a,b}}  (|\nabla_{g_0} f|^2 +f^2)\,  d  vol_{g_0} \Big)^{1/2}. 
\]
The inner products associated to these norms will be denoted with the braces $\langle\cdot, \cdot\rangle$. Introduce the bilinear form 
\begin{align*}
\mathcal{B}[\varphi, \psi]&= \int_{\mathbb{S}^2} \big(\nabla_{g_0} \varphi_1\cdot   \nabla_{g_0} \psi_1
+e^{4u_{a,b}} \nabla_{g_0} \varphi_2 \cdot  \nabla_{g_0} \psi_2
+8 e^{4 u_{a,b}} |\nabla_{g_0}    v_{a,b}|^2 \varphi_1 \psi_1\\
&\qquad\qquad  +4 e^{4u_{a,b}}\psi_1 \nabla_{g_0}  v_{a,b}\cdot \nabla_{g_0} \varphi_2   -4 e^{4u_{a,b}} \psi_2\nabla_{g_0}  v_{a,b}\cdot\nabla_{g_0}  \varphi_1 \big)\, d vol_{g_0},  
\end{align*}
for any $\varphi=(\varphi_1, \varphi_2), \psi=(\psi_1, \psi_2)\in H^1(\mathbb{S}^2)\times H_0^{1}(\mathbb{S}^2, e^{4u_{a,b}})$. If $\varphi_1\in C^2(\mathbb{S}^2)$ and $\varphi_2\in C^2_c(\mathbb{S}^2\setminus\{N,S\})$, then
\[
\mathcal{B}[\varphi, \psi]= -\langle\mathcal{T} _1 \varphi, \psi_1\rangle_{L^2(\mathbb{S}^2)} - \langle \mathcal{T}_2 \varphi, \psi_2\rangle_{L^2(\mathbb{S}^2, e^{4u_{a,b}} )}.
\]

\begin{lem}\label{lemma:sym-bilinear}$\mathcal{B}[\cdot, \cdot]$ is symmetric and nonnegative on $H^1(\mathbb{S}^2)\times H_0^{1}(\mathbb{S}^2, e^{4u_{a,b}})$.
\end{lem}

\begin{proof} For brevity, we write $\nabla=\nabla_{g_0}$ and $(u,v)=( u_{a,b}, v_{a,b})$. Then
for any $\varphi=(\varphi_1, \varphi_2)$ and $\psi=(\psi_1, \psi_2)\in H^1(\mathbb{S}^2)\times H_0^{1}(\mathbb{S}^2, e^{4u})$, the quadratic form becomes
\begin{align*}
\mathcal{B}[\varphi, \psi]&= \int_{\mathbb{S}^2} \big(\nabla \varphi_1\cdot   \nabla \psi_1
+e^{4u} \nabla \varphi_2 \cdot  \nabla \psi_2
+8 e^{4 u} |\nabla    v|^2 \varphi_1 \psi_1\\
&\qquad\qquad  +4 e^{4u} \psi_1\nabla  v\cdot \nabla  \varphi_2   -4 e^{4u} \psi_2\nabla  v \cdot\nabla  \varphi_1 \big)\, d vol_{g_0}.  
\end{align*}
There are five terms in the integrand, and the first three are symmetric in $\varphi$ and $\psi$, so
consider the last two terms. Set 
$$I[\varphi, \psi]=e^{4u}\psi_1  \nabla  v \cdot\nabla  \varphi_2  - e^{4u}\psi_2 \nabla  v\cdot \nabla  \varphi_1 ,$$
and observe that a simple computation yields 
$$I[\varphi, \psi]-I[\psi, \varphi]=e^{4u} \nabla  v\cdot \nabla(\varphi_2 \psi_1)  - e^{4u} \nabla  v\cdot \nabla  (\varphi_1 \psi_2).$$
Thus, integrating by parts and using $\mathrm{div}(e^{4u}  \nabla v)=0$ produces
\begin{align*}
&\int_{\mathbb{S}^2} I[\varphi, \psi]\, dvol_{g_0}
=\int_{\mathbb{S}^2} I[\psi, \varphi]\, dvol_{g_0},
\end{align*}
which yields $\mathcal{B}[\varphi, \psi]=\mathcal{B}[\psi, \varphi]$. 

Next, take a $\varphi=(\varphi_1, \varphi_2)\in H^1(\mathbb{S}^2)\times H_0^{1}(\mathbb{S}^2, e^{4u})$ 
and set 
\begin{align*}
J[\varphi]= |\nabla \varphi_1|^2
+e^{4u} |\nabla \varphi_2 |^2
+8 e^{4 u} |\nabla    v|^2 \varphi_1^2 
+4 e^{4u}\varphi_1 \nabla  v\cdot \nabla  \varphi_2   -4 e^{4u}\varphi_2 \nabla  v\cdot \nabla  \varphi_1 .  
\end{align*}
A straightforward computation yields 
\begin{align*}
J[\varphi]=J_1[\varphi]+J_2[\varphi],\end{align*}
where 
\begin{align*}
J_1[\varphi]= |\nabla \varphi_1-2 e^{4u} \nabla  v  \varphi_2|^2
+e^{4u} |\nabla \varphi_2+2 \nabla  v  \varphi_1+2  \nabla  u  \varphi_2|^2
+4 e^{4 u} |\nabla    v\varphi_1-\nabla    u\varphi_2|^2,\end{align*}
and 
\begin{align*}
J_2[\varphi]
=-  4 e^{4u} \varphi_2\nabla  u\cdot \nabla  \varphi_2 -8 e^{4u} |\nabla  u|^2 \varphi_2^2  -4 e^{8u} |\nabla  v|^2 \varphi_2^2.  
\end{align*}
Note that $J_1[\varphi]\ge 0$. 
By using the first equation in \eqref{ai9tgjiqihq}, we have
\begin{align*}J_2[\varphi]
=-  2 e^{4u} \nabla  u\cdot \nabla  \varphi^2_2-8 e^{4u} |\nabla  u|^2 \varphi_2^2  -4 e^{8u} |\nabla  v|^2 \varphi_2^2
=  -  2 \mathrm{div}(e^{4u} \nabla  u  \varphi_2^2),
\end{align*}
and hence 
\begin{align*}\int_{\mathbb{S}^2}J_2[\varphi]\, dvol_{g_0}=0.\end{align*}
As a consequence, we have $\mathcal{B}[\varphi , \varphi] \ge 0$. 
\end{proof}

For the nonnegativity of $\mathcal{B}$, we proved that its integrand can be decomposed 
as a sum of squares modulo a divergence term. This actually follows from the fact that 
the target space $\mathbb{H}^2$ is of negative curvature. Indeed, the proof above simply utilizes an 
explicit form of the standard index form obtained from the Jacobi equation. 
We will say that $\varphi\in H^1(\mathbb{S}^2)\times H_0^{1}(\mathbb{S}^2, e^{4u_{a,b}})$ is a weak solution of \eqref{eq:linearized-HM-1}-\eqref{eq:linearized-HM-2} if 
\[
\mathcal{B}[\varphi, \psi]=0, 
\]
for all $\psi\in H^1(\mathbb{S}^2)\times H_0^{1}(\mathbb{S}^2, e^{4u_{a,b}})$.

\begin{prop} \label{prop:kerr-kernel}
Let $(u_{a,b}, v_{a,b})$ be as in Proposition \ref{prop:classification}, for some $a>0$ and $b\in (0,1)$. Suppose that  $\varphi\in H^1(\mathbb{S}^2)\times H_0^{1}(\mathbb{S}^2, e^{4u_{a,b}})$ is a weak solution of \eqref{eq:linearized-HM-1}-\eqref{eq:linearized-HM-2}. Then $\varphi= \lda (\pa_b u_{a,b}, \pa_b v_{a,b})$ for some constant $\lda$.
\end{prop}

\begin{proof} 
In a manner similar to the proof of Proposition \ref{prop:partial-t-linear}, we can prove that $\varphi \in C^{3,\alpha} (\mathbb{S}^2)$ for any $\alpha \in (0,1)$, and  
$$\|(\sin \theta)^{j-3-\alpha}\nabla_{g_0}^j \varphi_2\|_{L^\infty(\mathbb{S}^2)}\le C$$ 
for $j=0,1,2,3$ where $C$ is a positive constant depending on $\varphi$. In particular, this implies that 
$\varphi_2(N)=\varphi_2(S)=0$. 

Next we claim that $\varphi$ is axisymmetric, that is, independent of the variable $\phi$. 
To see this, expand in a Fourier series 
$$\varphi(\theta, \phi)=  \xi_0(\theta)+\sum_{m=1}^\infty\big[\xi_m (\theta) \cos m \phi+\eta_m (\theta) \sin m \phi\big],$$
for some functions $\xi_m(\theta)=(\xi_{m,1} (\theta), \xi_{m,2} (\theta))$ with $m\ge 0$ and 
$\eta_m(\theta)=(\eta_{m,1} (\theta), \eta_{m,2} (\theta))$ 
with $m\ge 1$. A straightforward computation produces
\begin{align*}
\mathcal{B}[\varphi,  \varphi] &=\sum_{m=1}^\infty m^2 \pi \int_{0}^\pi  \big((\xi_{m,1} ^2+\eta_{m,1} ^2) + e^{4u_{a,b}}(\xi_{m,2} ^2 +\eta_{m,2} ^2) \big)\,\frac{d \theta}{\sin \theta}\\
&\qquad +\mathcal{B}[\xi_0, \xi_0]+\frac12\sum_{m=1}^\infty\big(\mathcal{B}[\xi_m, \xi_m]+\mathcal{B}[\eta_m, \eta_m]\big),
\end{align*}
and thus Lemma \ref{lemma:sym-bilinear} implies
\begin{align*}
\mathcal{B}[\varphi,  \varphi] \ge \sum_{m=1}^\infty m^2 \pi \int_{0}^\pi  \big((\xi_{m,1} ^2+\eta_{m,1} ^2) + e^{4u_{a,b}} (\xi_{m,2} ^2 +\eta_{m,2} ^2)\big)\,\frac{d \theta}{\sin \theta}\ge 0.\end{align*}
Furthermore, since $\varphi$ is a weak solution we have that $\mathcal{B}[\varphi,  \varphi]=0$. 
It follows that all $\xi_m$ and $\eta_m$ are identically zero for $m\ge 1$. 
Therefore $\varphi(\theta, \phi)=  \xi_0(\theta)$ is a function independent of $\phi$.

Now, $\varphi\in C^2 ([0,\pi])$ is a solution of the following boundary value problem:
\begin{equation} \label{eq:jacobi}\begin{split}
L_{a,b,1}\varphi=L_{a,b,2}\varphi&=0 \quad \text{on }[0,\pi],\\
\varphi_2(0)=\varphi_2(\pi)&=0. 
\end{split}\end{equation}
Set $\varphi_*=(\varphi_{*1}, \varphi_{*2})= (\pa_bu_{a,b}, \pa_bv_{a,b})$ and observe that 
$L_{a,b}\varphi_*=0$. Moreover, a straightforward computation yields 
\begin{align*}
\varphi_*=  \Big(\frac{b}{2(1-b^2)} +\frac{\cos \theta }{1+\cos^2 \theta +2b \cos \theta},  
\frac{a\sin^4 \theta }{(1+ \cos ^2 \theta +2b  \cos \theta)^2}\Big), 
\end{align*} 
so that $\varphi_{*2}(0)=\varphi_{*2}(\pi)=0$. Hence, $\varphi_*$ satisfies the boundary value problem \eqref{eq:jacobi}.

Let us examine \eqref{eq:jacobi}. 
Notice that the second equation $L_{a,b,2}\varphi=0$ gives 
\[
(e^{4 u_{a,b}} \sin \theta \varphi_2')'+4 e^{4 u_{a,b}} \sin \theta v_{a,b}'  \varphi_1'=0,
\]
and a calculation produces 
\begin{align*}
e^{4 u_{a,b}} \sin  \theta  v_{a,b}'
= - \frac{1}{2a}. 
\end{align*} 
Therefore 
\begin{align} 
\label{eq:ODE-varphi-2}
e^{4 u_{a,b}} \sin \theta \varphi_2' =\frac{2}{a} \varphi_1 +c_0,
\end{align}
for some constant $c_0$. Substituting this into the first  equation $L_{a,b,1} \varphi =0$ shows that
\begin{align*}
0&=(\sin \theta \varphi_1')' -8\sin\theta e^{4u_{a, b}}| v_{a,b}'|^2 
\varphi_1-4\left(\frac{2}{a} \varphi_1 +c_0 \right) v_{a,b}' \\&
= (\sin \theta \varphi_1')' -8v_{a,b}' \left( \sin\theta e^{4u_{a, b}} v_{a,b}'  +\frac{1}{a}\right)\varphi_1 -4 c_0v_{a,b}'\\&
=(\sin \theta \varphi_1')' -\frac{4}{a} v_{a,b}' \varphi_1 -4 c_0v_{a,b}'.
\end{align*} 
In particular, we conclude that $\varphi_1=-ac_0 $ is a solution. 

Now write $\psi=\varphi_1+ac_0$, and observe that this satisfies the homogenous equation
\be\label{eq:Jacobi-eq-1}
(\sin \theta \psi')' -\frac{4}{a} v_{a,b}' \psi=0.
\ee
We claim that $\varphi_{*1}$ is a solution of \eqref{eq:Jacobi-eq-1}. 
To see this, one may simply compute the left-hand side of \eqref{eq:Jacobi-eq-1} 
for $\psi=\varphi_{*1}$ to find it is zero.  However, a quicker method is to compute 
the constant $c_0$ given by \eqref{eq:ODE-varphi-2} 
for $(\varphi_{*1}, \varphi_{*2})$. In fact we have 
\begin{align*}
c_0&= e^{4 u_{a,b}} \sin \theta \varphi_{*2}'  - \frac{2}{a} \varphi_{*1} = 
\lim_{\theta \to 0}\left(e^{4 u_{a,b}} \sin \theta \varphi_{*2}'  - \frac{2}{a} \varphi_{*1} \right)\\&
=\frac{1}{a(1-b^2)}- \frac{b}{a(1-b^2)} -\frac{1}{a(1+b)} =0,
\end{align*}
and the claim follows.

To find another linearly independent solution of \eqref{eq:Jacobi-eq-1}, change variables to $s(\theta)=\frac12\ln \frac{1-\cos \theta}{1+\cos \theta}$, and note that $s'= \frac{1}{\sin \theta}$. Then \eqref{eq:Jacobi-eq-1} becomes
\[
 \frac{ d  ^2}{ d  s^2} \psi -\Big(\frac{4}{a} v_{a,b}'\sin \theta \Big) \psi=0 \quad\text{for }  s\in (-\infty,\infty).
\]
By the standard variation of parameters, another linearly independent solution is given by $C(s) \varphi_{*1} $ where $ \frac{ d  }{ d  s}  C= |\varphi_{*1} |^{-2}$, or rather 
\[
C(s(\theta))= \int_{\frac{\pi}{2}}^{\theta} |\varphi_{*1}(\tau) |^{-2} \frac{1}{\sin\tau}\, d  \tau
\]
up to an additive constant. Since $|\varphi_{*1} (0)|^{-1} \neq 0$, this solution blows-up as $\theta \to 0$. 
Thus, any $C^2$ solution of \eqref{eq:Jacobi-eq-1} on $[0,\pi]$ is given by $\psi=c_1\varphi_{*1}$, 
for some constant $c_1$. As a consequence, we have 
\[
\varphi_1= -c_0 a + c_1 \varphi_{*1}.
\]
Substituting the above expression into \eqref{eq:ODE-varphi-2} yields
\begin{align*}
e^{4 u_{a,b}} \sin \theta \varphi_2' &=\frac{2}{a} (-c_0 a + c_1 \varphi_{*1} ) +c_0 
=\frac{2c_1}{a} \varphi_{*1}  -c_0\\&
= c_1e^{4 u_{a,b}} \sin \theta \varphi_{*2}' -c_0,
\end{align*}
and hence
\[
 \varphi_2- c_1\varphi_{*2} =-c_0 \int_0^{\theta} e^{-4 u_{a,b}} \sin^{-1} \tau\, d  \tau +c_2,
\] 
for some constant $c_2$. By \eqref{eq:jacobi}, $\varphi_2- c_1\varphi_{*2}$ is zero at $0$ 
and $\pi$ simultaneously, which implies that $c_0=c_2=0$. Therefore, $\varphi_1=  c_1 \varphi_{*1}  $ and $\varphi_2=  c_1 \varphi_{*2}$. 
\end{proof}

Proposition \ref{prop:kerr-kernel} asserts that zero is the least eigenvalue of $\mathcal{T}$ and that it is simple, under the condition \eqref{eq:linearized-HM-2}. In fact, it can be shown that there exists an orthonormal basis of $L^2(\mathbb S^2)\times L^2(\mathbb S^2, e^{4u_{a,b}})$ formed by the eigenfunctions of $\mathcal{T}$.

\begin{rem}\label{rem:jacobi-field}
We note that $ (\pa_au_{a,b}, \pa_av_{a,b})$  satisfies the equations  in \eqref{eq:jacobi}. 
A straightforward computation produces
\[
(\pa_a u_{a,b}, \pa_a v_{a,b}) =\Big( -\frac1{2a}, \frac{b +b  \cos ^2 \theta +2\cos \theta}{1+ \cos ^2 \theta 
+2b  \cos \theta}\Big),
\] 
showing that $\pa_a v_{a,b}(0)=1$ and $\pa_a v_{a,b}(\pi) =-1$.
Thus, $(\pa_a u_{a,b}, \pa_a v_{a,b})$ does not satisfy the boundary conditions in \eqref{eq:jacobi}. 
\end{rem}

\begin{proof}[Proof of Theorem \ref{thm:classification-non-degeneracy}] 
This result follows directly from Propositions \ref{prop:classification} and \ref{prop:kerr-kernel}.
\end{proof}

\section{Convergence: Translation Invariant Renormalization}
\label{s:convergence-1}

In this section we will study the asymptotic behaviors of harmonic maps near the
prescribed singularity. Throughout this section it will always be assumed that $\w$ is a positive smooth function on $\mathbb{S}^2\setminus \{N,S\}$, in particular independent of $t$, that satisfies
\begin{equation} \label{eq:NS-singular}
\ln \w-\ln \sin\theta \in C^{10} ( \mathbb{S}^2).
\end{equation} 
Consider weak solutions 
$(\Phi,v)\in H^{1}_{\mathrm{loc}}( \mathbb R_+ \times \mathbb{S}^2)$  
of the renormalized  harmonic map system 
\begin{align}\label{eq:main-2} \begin{split}
L \Phi -2 e^{4u} |\nabla_{g} v|^2  &=L\ln\w, \\
Lv +4\nabla_{g}u\cdot   \nabla_{g} v&=0.
\end{split}\end{align}
We are interested in the behavior as $t\to\infty$, which in the realm of applications to
stationary vacuum spacetimes corresponds to the approach towards the degenerate black hole
horizon. As before we set $u=\Phi-\ln \w$. In addition, it is assumed that there are positive constants
$\Lambda$ and $a$ such that
\begin{equation} \label{eq:ubd-1} 
|\Phi| \le \Lda \quad \text{ on }\mathbb R_+ \times \mathbb{S}^2,
\end{equation}
for any finite interval $I\subset \mathbb R_+$ we have
\begin{equation}\label{eq:ea-1}
\int_{I} \int_{\mathbb{S}^2} (| \nabla_{g}\Phi|^2 + \w^{-4} | \nabla_{g}v|^2 )\, d  vol_{g_0}  d  t <\infty,
\end{equation}
and in the trace sense
\begin{equation} \label{eq:trace-1}
v= a  \text{ on  }\mathbb R_+ \times \{N\}, \quad \quad
v=-a  \text{ on  }\mathbb R_+ \times \{S\}.
\end{equation} 
 
Observe that since $\w$ is independent of $t$, the system \eqref{eq:main-2} is translation invariant in $t$, 
and Theorem \ref{thm:reg} applies to yield regularity and uniform bounds for $(\Phi(t),v(t))$ on the time interval $[2,\infty)$. 
Our first goal is to establish a crucial identity involving the renormalized harmonic map energy on the 2-sphere
\begin{align}\label{eq:def-E}
\mathcal{E}(\Phi,v)=\frac12\int_{\mathbb{S}^2}  (|\nabla_{g_0} \Phi|^2 
+e^{4u}|\nabla _{g_0} v|^2+2(L\ln\w) \Phi  ) \, d  vol_{g_0}. 
\end{align}
It should be noted that $L\ln\w=\Delta_{g_0}\w$, as $\w$ is independent of $t$.  The following result may be
viewed as yielding a type of monotonicity formula.

\begin{prop}\label{prop:MF}  
Let $\omega$ satisfy \eqref{eq:NS-singular} 
and $(\Phi,v)$ be a weak solution of \eqref{eq:main-2} on $\mathbb R_+ \times \mathbb{S}^2$,  
satisfying \eqref{eq:ubd-1}, \eqref{eq:ea-1}, and \eqref{eq:trace-1} 
for some positive constants $\Lambda$ and $a$.  Then
\begin{align}\label{eq:fund-iden}
\begin{split}
&\frac{ d  }{ d  t} \Big[\frac12\int_{\mathbb{S}^2} ( |\pa_t \Phi|^2  
+e^{4u} |\pa_t v| ^2)(t)\, d  vol_{g_0}
- \mathcal{E}(\Phi(t),v(t))\Big]\\&
\qquad= \int_{\mathbb{S}^2} ( |\pa_t \Phi|^2  +e^{4u} |\pa_t v| ^2)(t)\, d  vol_{g_0}
\end{split}
\end{align}
 for all $t\in \mathbb R_+$.
\end{prop}

\begin{proof}  
Multiply the first equation of \eqref{eq:main-2} by $\pa_t \Phi$ to find
\begin{align*} 
|\pa_t \Phi|^2-\frac12\pa_t|\pa_t \Phi|^2+\frac12\pa_t|\nabla_{g_0} \Phi|^2+(L\ln\w)\pa_t \Phi& \\
-\ \mathrm{div}_{g_0}(\pa_t \Phi\nabla_{g_0} \Phi)+2e^{4u}\pa_t \Phi |\nabla_gv|^2&=0.
\end{align*}
Next, multiply the second equation of \eqref{eq:main-2} by $e^{4u}\pa_t v$ and use that $\pa_t u=\pa_t \Phi$ to produce 
\begin{align*} 
e^{4u}|\pa_t v|^2-\frac12\pa_t(e^{4u}|\pa_t v|^2)+\frac12\pa_t(e^{4u}|\nabla_{g_0}v|^2)& \\
-\ \mathrm{div}_{g_0}(e^{4u}\pa_t v\nabla_{g_0} v)-2e^{4u}\pa_t \Phi |\nabla_gv|^2&=0.
\end{align*}
Adding the two previous equations then yields
\begin{align*} 
|\pa_t \Phi|^2+e^{4u}|\pa_t v|^2-\frac12\pa_t(|\pa_t \Phi|^2+e^{4u}|\pa_t v|^2)&\\
+\frac12\pa_t\big(|\nabla_{g_0} \Phi|^2+e^{4u}|\nabla_{g_0} v|^2+2(L\ln\w)\Phi\big)& \\
-\ \mathrm{div}_{g_0}(\pa_t \Phi\nabla_{g_0} \Phi+e^{4u}\pa_t v\nabla_{g_0} v)&=0.
\end{align*}
Finally, integrating this expression over $\mathbb S^2$ produces the desired formula \eqref{eq:fund-iden}.
\end{proof} 

As a consequence we find that the $t$-derivative portion of the renormalized energy is globally finite, and also obtain
an initial decay statement for the derivatives of the renormalized map.

\begin{cor}\label{cor:L2} 
Let $\omega$ satisfy \eqref{eq:NS-singular} 
and $(\Phi,v)$ be a weak solution of \eqref{eq:main-2} on $\mathbb R_+ \times \mathbb{S}^2$,  
satisfying \eqref{eq:ubd-1}, \eqref{eq:ea-1}, and \eqref{eq:trace-1} 
for some positive constants $\Lambda$ and $a$.   Then
\[
\int_{2}^\infty\int_{\mathbb{S}^2} \big(| \pa_t \Phi|^2 
+e^{4u} |\pa_t v|^2\big)\, d  vol_{g_0}  d  t <\infty,
\]
and 
\[
\lim_{t\to \infty} \max_{\mathbb{S}^2} \left(| \nabla_{g_0}^l \pa_t^k\big (\Phi(t,\cdot), v(t,\cdot)\big)| 
+e^{2u(t,\cdot)} | \pa_t^k v(t,\cdot)|\right)=0
\]
for all $k=1,2,3$ and $l=0,1,2,3$.
\end{cor}

\begin{proof} By integrating \eqref{eq:fund-iden} over $(2,\sigma)$ for any $\sigma>2$, 
using Theorem  \ref{thm:reg}, and letting $\sigma\to \infty$ produces the first conclusion. 
Since $\w$ is independent of $t$ and \eqref{eq:main-2} is translation invariant in $t$, 
the second conclusion then follows from Proposition \ref{prop:partial-t-linear} 
and the first conclusion. 
\end{proof}

The next proposition shows a sequential convergence of energy to that of a renormalized tangent map.
 
\begin{prop}\label{prop:sub-conv} 
Let $\omega$ satisfy \eqref{eq:NS-singular} 
and $(\Phi,v)$ be a weak solution of \eqref{eq:main-2} 
on $\mathbb R_+ \times \mathbb{S}^2$, satisfying \eqref{eq:ubd-1}, \eqref{eq:ea-1}, 
and \eqref{eq:trace-1} for some positive constants $\Lambda$ and $a$.    
Then there is a sequence $t_i\to \infty$ as $i\to \infty$, such that 
$ (  \Phi(t_{i}), v(t_{i})) $  converges to some $(\bar \Phi, \bar v)$ in $C^3(\mathbb{S}^2)$, and 
\[
\lim_{i\to \infty} \mathcal{E}(\Phi(t_i),v(t_i))= \mathcal{E}(\bar \Phi,\bar v).
\] 
Moreover,  $(\bar u, \bar v):= (\bar\Phi-\ln \w , \bar v) $ satisfies 
\begin{align}\label{eq:main-s2-1}\begin{split} 
\Delta_{g_0} \bar \Phi -2 e^{4\bar u}  |\nabla_{g_0} \bar  v|^2 &=L\ln\w, \\
\Delta_{g_0} \bar v +4 \nabla_{g_0} \bar u \cdot \nabla_{g_0}\bar  v &=0,
\end{split}\end{align}
on $\mathbb{S}^2\setminus \{N,S\}$, and 
\begin{align}\label{eq:main-s2-1-poles}\bar v(N)=a,\quad\quad \bar v(S)=-a.
\end{align}
\end{prop} 

\begin{proof} Take an arbitrary $\alpha\in (0,1)$. 
By Theorem  \ref{thm:reg}, for any $t\in (2,\infty)$ we have
$$\|(\Phi(t,\cdot), v(t,\cdot))\|_{C^{3,\alpha}(\mathbb{S}^2)} \le C,$$ 
for some constant $C$ independent of $t$. 
The Arzela-Ascoli theorem then implies the first conclusion of the proposition, as well
as \eqref{eq:main-s2-1-poles}. 
By Corollary \ref{cor:L2} and equation \eqref{eq:main-2}, it follows that $(\bar\Phi-\ln \w, \bar v)$ 
satisfies \eqref{eq:main-s2-1}.
\end{proof}

A priori, different sequences of times may lead to different limit energies. However, we show that this
in fact cannot happen with the following result.

\begin{prop}\label{prop:funct-conv} Let 
$\w$ satisfy \eqref{eq:NS-singular} and  
$(\Phi,v)$ be a weak solution of \eqref{eq:main-2} on $\mathbb R_+ \times \mathbb{S}^2$, 
satisfying \eqref{eq:ubd-1}, \eqref{eq:ea-1},  and \eqref{eq:trace-1} 
for some positive constants $\Lambda$ and $a$, 
and $(\bar \Phi, \bar v)$ be as in Proposition \ref{prop:sub-conv}. 
Then   
\[
\mathcal{E}(\Phi(t),v(t)) \to \mathcal{E}(\bar \Phi,\bar v) \quad \mbox{as }t\to \infty.
\]  
Moreover, for any $t>2$ the following energy identity holds
\begin{align}\label{eq:energy-identity-infty}
\begin{split}
\mathcal{E}(\Phi(t),v(t))- \mathcal{E}(\bar \Phi,\bar v)
&= \int_{t}^{\infty}\int_{\mathbb{S}^2}\big( | \pa_t \Phi|^2 
+e^{4u} |\pa_t v|^2\big)(s)\, d  vol_{g_0}  d  s\\
&\qquad+ \frac12\int_{\mathbb{S}^2}\big( | \pa_t \Phi|^2 
+e^{4u} |\pa_t v|^2\big)(t)\, d  vol_{g_0}.  
\end{split}
\end{align}
\end{prop}

\begin{proof} 
By Proposition \ref{prop:MF}, for any $5<t_1<t_2$ we have
\begin{align*}
&| \mathcal{E}(\Phi,v)(t_1)-  \mathcal{E}(\Phi,v)(t_2)| \\
&\qquad\le \int_{t_1}^{t_2} \int_{\mathbb{S}^2}\big(  | \pa_t \Phi|^2 
+e^{4u} |\pa_t v|^2)(s)\, d  vol_{g_0}  d  s \\
& \qquad\quad + \frac12\int_{\mathbb{S}^2} \big( | \pa_t \Phi|^2 
+e^{4u} |\pa_t v|^2\big)(t_1)\, d  vol_{g_0} 
+ \frac12\int_{\mathbb{S}^2}\big(  | \pa_t \Phi|^2 +e^{4u} |\pa_t v|^2\big)(t_2)\, d  vol_{g_0}. 
\end{align*}
Then Corollary \ref{cor:L2} implies that $\mathcal{E}(\Phi(t),v(t)) $ converges as $t\to\infty$, and
Proposition  \ref{prop:sub-conv} ensures that the limit is $\mathcal{E}(\bar \Phi,\bar v)$. 
Furthermore, identity \eqref{eq:energy-identity-infty} may be obtained by integrating \eqref{eq:fund-iden} over $(t,\sigma)$, using Corollary \ref{cor:L2}, and letting $\sigma\to\infty$.
\end{proof}

Having shown that the renormalized energies converge upon approach to the singularity, our next task will be to
improve this by estimating the convergence rate; the fundamental estimate needed for this purpose is referred to as the {\L}ojasiewicz-Simon inequality. 
In order to achieve this goal, we first introduce certain weighted Banach spaces.
For any integer $k\ge 0$ and $\alpha\in(0,1)$, denote by $C^{k}_*(\mathbb{S}^2)$ 
the space of vector-valued functions $\varphi=(\varphi_1,\varphi_2)\in C^k(\mathbb{S}^2)$ 
with the bounded norm 
\begin{equation*}
|\varphi|_{k}^{*}= \sum_{0\le j\le k} \left\{\max_{\mathbb{S}^2}
|\nabla_{g_0}^j \varphi_1|+\max_{\mathbb{S}^2}\left(\w^{j-3-\frac12}| \nabla_{g_0}^j \varphi_2|\right) \right\},
\end{equation*}
and denote by $C^{k,\alpha}_*(\mathbb{S}^2)$ the space of functions  
$\varphi\in C^{k}_*(\mathbb{S}^2)\cap C^{k,\alpha}(\mathbb{S}^2)$ with the bounded norm 
\[
|\varphi|_{k,\alpha}^{*}= |\varphi|_{k}^{*} +\|\nabla_{g_0}^k\varphi\|_{C^{\alpha}(\mathbb{S}^2)}. 
\]
Furthermore, let $L^2(\mathbb{S}^2, \w^{-4})$ be the subspace of $L^2(\mathbb{S}^2)$ 
which consists of functions with the bounded norm 
\[
\|f\|_{L^{2}(\mathbb{S}^2, \w^{-4})}=\Big(\int_{\mathbb{S}^2} \w^{-4} f^2\,  d  vol_{g_0} \Big)^{1/2},
\] 
and let $ H_0^{1}(\mathbb{S}^2, \w^{-4})$ be the closure of 
$C_c^\infty(\mathbb{S}^2\setminus \{N,S\})$ under the norm 
\[
\|f\|_{H^{1}(\mathbb{S}^2, \w^{-4})} 
=\Big(\int_{\mathbb{S}^2} \w^{-4} (|\nabla_{g_0} f|^2 +f^2)\,  d  vol_{g_0} \Big)^{1/2}. 
\]

Consider a weak solution $(\Phi,v)$ of \eqref{eq:main-2},
and $(\bar \Phi, \bar v)$ be a solution of \eqref{eq:main-s2-1}. 
For convenience, we will write 
$$w=(w_1, w_2)=(\Phi,v),  \quad \bar w=(\bar w_1, \bar w_2)=(\bar\Phi, \bar v).$$
Then equation \eqref{eq:main-2} may be rewritten as
\begin{align}\label{eq:main-w} 
\begin{split}
\pa_t^2 w_1- \pa_t w_1 -2e^{4u}|\pa_t w_2|^2 +\mathcal{M}_1(w)&=0,\\
\pa_t^2 w_2- \pa_t w_2  +4 \pa_t u \pa_t w_2 +\mathcal{M}_2(w)&=0, 
\end{split} 
\end{align}
where 
\begin{align}\label{eq:main-mathcal-M}\begin{split}
\mathcal{M}_1(w)&= \Delta_{g_0} w_1 -2 e^{4u} |\nabla_{g_0}  w_2 |^2-L\ln\w, \\
\mathcal{M}_2(w)&=\Delta_{g_0} w_2 +4\nabla_{g_0}  u\cdot \nabla_{g_0}  w_2. 
\end{split} \end{align}
Observe that the energy $\mathcal E$ of \eqref{eq:def-E} becomes
\begin{equation} \label{eq:F-defn}
\mathcal{E}(w)=\frac12\int_{\mathbb{S}^2}  (|\nabla_{g_0} w_1|^2 
+e^{4u}|\nabla _{g_0} w_2|^2+2(L\ln\w) w_1 ) \, d  vol_{g_0},
\end{equation} 
and since $\bar w=(\bar\Phi, \bar v)$ is a solution of \eqref{eq:main-s2-1} we have that
$\mathcal M(\bar w)=(\mathcal{M}_1(\bar w),\mathcal{M}_2(\bar w))=0$. 

We will make use of two related weighted inner products, and for this reason we introduce 
the following notation. Let $h$ be a given positive function on $\mathbb S^2\setminus\{N, S\}$. 
For any $\varphi=(\varphi_1, \varphi_2), \zeta=(\zeta_1, \zeta_2)
\in L^2(\mathbb{S}^2) \times L^2(\mathbb{S}^2, \w^{-4})$
define 
\begin{align*}
\langle\varphi, \zeta\rangle_{L^2(\mathbb{S}^2, 1\times h)} 
=\int_{\mathbb{S}^2}\big(\varphi_1\zeta_1+ h \varphi_2\zeta_2)\,  d  vol_{g_0}, 
\end{align*}
and 
\begin{align*}
\|\varphi\|_{L^2(\mathbb{S}^2, 1\times h)} 
=\Big(\int_{\mathbb{S}^2}\big(\varphi_1^2+ h \varphi_2^2)\,  d  vol_{g_0} \Big)^{1/2}. 
\end{align*}
Here $1\times h$ means that 1 is the weight for the first component, and 
$h$ is the weight for the second component. Below we will take either $h=\w^{-4}$ or $h=e^{4u}$. 
Notice that with $|u+\ln \w|\le \Lambda$ as in \eqref{eq:ubd-1}, 
the $L^2(\mathbb{S}^2, 1\times \w^{-4})$-norms are equivalent to 
the $L^2(\mathbb{S}^2, 1\times e^{4u})$-norms, that is 
\begin{equation}\label{aoirhoihnfghgh}
e^{-2\Lambda}\|\varphi\|_{L^2(\mathbb{S}^2, 1\times \w^{-4})} 
\le \|\varphi\|_{L^2(\mathbb{S}^2, 1\times e^{4u})} 
\le e^{2\Lambda}\|\varphi\|_{L^2(\mathbb{S}^2, 1\times \w^{-4})}.
\end{equation}
Although the weight $h=\w^{-4}$ is independent of $t$ and seems simpler, 
in certain situations it is more advantageous to use the weight 
$h=e^{4u}$, since it is directly related to equation \eqref{eq:main-w}. 
For example, we may write \eqref{eq:energy-identity-infty} as 
\begin{align}\label{eq:energy-identity-infty-v2}
\mathcal{E}(w(t))- \mathcal{E}(\bar w)
= \int_{t}^{\infty}\|\partial_t w(s)\|^2_{L^2(\mathbb{S}^2, 1\times e^{4u})}\, d  s
+ \frac12\|\partial_t w(t)\|^2_{L^2(\mathbb{S}^2, 1\times e^{4u})}.  
\end{align}
This fundamental energy identity, as well as the next result concerning the {\L}ojasiewicz-Simon inequality, play  essential roles in our study of the uniqueness of tangent maps. 

\begin{prop} \label{prop:functional-F} Let $\alpha\in (0,1)$ and 
$(\bar \Phi, \bar v)\in C^{3}(\mathbb S^2)$ be a solution of 
\eqref{eq:main-s2-1}-\eqref{eq:main-s2-1-poles}. 
Then there exist constants $\eta>0$ and $\bar C>0$,
depending only on $\alpha$ and $(\bar \Phi,\bar v)$, such that  
\begin{equation} \label{eq:LS-ineq}
|\mathcal{E} (w)- \mathcal{E} (\bar w)|^{1/2} 
\le  \bar C \|\mathcal{M}(w)\|_{L^2(\mathbb S^2, 1\times \w^{-4})}
\ee
for all $w-\bar w\in C^{2,\alpha}_*(\mathbb{S}^2)$ with $|w-\bar w|_{2,\alpha}^*<\eta$.
\end{prop}

\begin{proof} We first show that $\mathcal M=(\mathcal M_1, \mathcal M_2)$
is the Euler-Lagrange operator of $\mathcal E$. 
To see this, take any $w-\bar w\in C^{2,\alpha}_*(\mathbb{S}^2)$ and 
$\zeta=(\zeta_1, \zeta_2)\in H^{1}(\mathbb{S}^2) \times  H^{1}(\mathbb{S}^2, \w^{-4})$,
then the Fr\'echet derivative of $\mathcal{E}$ at $w$ and applied to $\zeta$ is given by 
\begin{align*}
\mathcal{E}'(w)[\zeta]&= \frac{ d \ }{ d  s}\Big|_{s=0} \mathcal{E}(w+s\zeta) \\
&=  \int_{\mathbb{S}^2}  \left[\nabla_{g_0}w_1 \cdot\nabla_{g_0} \zeta_1 
+(L\ln\w) \zeta_1 +2e^{4u}|\nabla _{g_0} w_2|^2\zeta_1 \right] \, d  vol_{g_0} 
\\& \qquad +   \int_{\mathbb{S}^2}  e^{4u} \nabla _{g_0} w_2 \cdot\nabla _{g_0} \zeta_2 \, d  vol_{g_0}\\
&
= - \int_{\mathbb{S}^2} \big[ \mathcal{M}_1(w) \zeta_1 
+e^{4u} \mathcal{M}_2(w) \zeta_2\big] \, d  vol_{g_0}.
\end{align*}
Therefore
\begin{align*}
\mathcal{E}'(w)[\zeta]=-\langle \mathcal{M}(w), \zeta\rangle_{L^2(\mathbb{S}^2, 1\times e^{4u})},
\end{align*} 
and by the Cauchy-Schwarz inequality we obtain
\begin{align*}
\big|\mathcal{E}'(w)[\zeta]\big|\le 
C\|\mathcal{M}(w)\|_{L^2(\mathbb{S}^2, 1\times \w^{-4})} 
\|\zeta\|_{L^2(\mathbb{S}^2, 1\times \w^{-4})}.
\end{align*}
Note that the constant $C$ depends on $w_1$, since $u=w_1-\ln\omega$ and the weight $1\times e^{4u}$ is replaced by $1\times \omega^4$ analogously to \eqref{aoirhoihnfghgh}, however this dependence is removed with uniform control on $w-\bar{w}$ as expressed in the hypotheses.
Furthermore, observe that $\mathcal{E}$ is an analytic functional on $C^{2,\alpha}_*(\mathbb{S}^2)$ 
and that $\bar w$ is a critical point, namely $\mathcal M(\bar w)=0$.  
Hence, by the {\L}ojasiewicz-Simon gradient inequality there exist constants 
$\eta>0, \vartheta\in (0,1/2]$, and $\bar C>0$ such that   
$$|\mathcal{E} (w)- \mathcal{E} (\bar w)|^{1-\vartheta} \le  
\bar C \|\mathcal{M}(w)\|_{L^2(\mathbb S^2, 1\times \w^{-4})}
$$
for all $w-\bar w\in C^{2,\alpha}_*(\mathbb{S}^2)$ with $|w-\bar w|_{2,\alpha}^*<\eta$. 
By Theorem \ref{thm:classification-non-degeneracy} we must have $\vartheta=1/2$.
For details we refer to Theorem 3.1 of Simon \cite{Simon83}, and also Theorem 3.10 of Chill \cite{Chill}.
\end{proof}

We are now in a position to apply the method of Simon \cite{Simon83} 
to obtain convergence to the tangent map.  
In the proof below we will adapt the arguments of \cite{J}. 

\begin{lem}\label{lem:step1} Let 
$\w$ satisfy \eqref{eq:NS-singular},  
$w=(\Phi,v)$ be a weak solution of \eqref{eq:main-2} 
in $\mathbb R_+ \times \mathbb{S}^2$ 
satisfying \eqref{eq:ubd-1}, \eqref{eq:ea-1},  
and \eqref{eq:trace-1} for some positive constants $\Lambda$ and $a$, 
let $\bar w=(\bar \Phi, \bar v)\in C^{3}(\mathbb{S}^2)$ be a solution of 
\eqref{eq:main-s2-1}-\eqref{eq:main-s2-1-poles} as in Proposition \ref{prop:sub-conv}.
Then there exist positive constants $\gamma$, $\beta$, and $ C_*$
depending only on $\bar w$ and $\Lda$, such that 
for any $t_*\ge 5$ and $\sigma>t_*+5$ the following holds.
If for all $t \in(t_*, \sigma]$ the estimate 
\begin{equation}\label{eq10-PDE-order2}
\| w(t) - \bar w \|_{L^2(\mathbb S^2, 1\times \w^{-4})}< \gamma
\end{equation}
is valid, then  
\begin{equation}\label{eq11-PDE-order2}
\|w(t)-w(t')\|_{L^2(\mathbb S^2, 1\times \w^{-4})}\le 
C_*e^{-\beta(t-t_*)} 
\end{equation} 
for all $t, t'\in(t_*+1,\sigma-1]$ with $t\leq t'$.
\end{lem}

\begin{proof} 
The proof consists of two steps. In the first step \eqref{eq11-PDE-order2}
will be established under a stronger assumption than \eqref{eq10-PDE-order2}, and in the second step
it will be shown that the stronger assumption may be relaxed to the desired hypothesis. 
Moreover, the weight $1\times e^{4u}$ will be used in the first step, while
the weight $1\times \w^{-4}$ will be used in the second step. 
In what follows we will write $f'$ for the derivative of $f$ with respect to $t$. 
As a preliminary observation, note that since $\w$ is independent of $t$ we find for any $t_*<t\leq t'\leq\sigma-1$ that
\begin{align*}
\|w(t)-w(t')\|_{L^2(\mathbb S^2, 1\times \w^{-4})}
&=\Big\|\int_t^{t'} w'(s)ds\Big\|_{L^2(\mathbb S^2, 1\times \w^{-4})}\\
&\le \int_t^{\sigma-1} \|w'(s)\|_{L^2(\mathbb S^2, 1\times \w^{-4})}ds.
\end{align*}
Therefore it suffices to prove 
\begin{equation}\label{eq11-L1-bound-PDE-order2}
\int_t^{\sigma-1}\|w'(s)\|_{L^2(\mathbb S^2, 1\times \w^{-4})}ds\le 
C_*e^{-\beta(t-t_*)} 
\end{equation} 
for all $t \in(t_*+1,\sigma-1]$, in order to obtain \eqref{eq11-PDE-order2}. 

{\it Step 1.} Fix a constant $\alpha\in (0,1)$, and let
$\eta>0$ be as in Proposition \ref{prop:functional-F}.
We will assume that for all $t \in(t_*, \sigma]$ the estimate
\begin{equation}\label{eq10-C2alpha-PDE-order2}
|w(t) - \bar w |^*_{2,\alpha}< \eta
\end{equation}
is valid, and then show that
\begin{equation}\label{eq11-L1-bound-PDE-order2-v2}
\int_t^{\sigma}\|w'(s)\|_{L^2(\mathbb S^2, 1\times e^{4u})}ds\le 
C_*e^{-\beta(t-t_*)} 
\end{equation} 
for all $t \in(t_*,\sigma]$. It should be pointed out that this inequality is 
slightly different from \eqref{eq11-L1-bound-PDE-order2}, in particular due to the use of weight $e^{4u}$ instead of $\w^{-4}$. 
For brevity, we will write 
$$\langle\cdot, \cdot\rangle_{*}
=\langle\cdot, \cdot\rangle_{L^2(\mathbb S^2, 1\times e^{4u})}, \quad \quad
\|\cdot\|_{*}
=\|\cdot \|_{L^2(\mathbb S^2, 1\times e^{4u})}.$$

The process to establish \eqref{eq11-L1-bound-PDE-order2-v2} begins with the definition of an entropy-type function  
\begin{align}\label{eq-definition-H-PDE}
H(t)=\int_t^\infty \|w'(s)\|^2_{*}ds
+\varepsilon\big[\mathcal E(w(t))-\mathcal E(\bar w)
+\langle\mathcal M(w(t)), w'(t)\rangle_{*}\big],
\end{align}
for $t> 0$ and a positive constant $\varepsilon$ to be fixed.
Observe that \eqref{eq:energy-identity-infty-v2} implies
\begin{align}\label{eq-definition-H-PDE-v2}
H(t)&=-\frac12\|w'(t)\|^2_{*}
+(1+\varepsilon)\big[\mathcal E(w(t))-\mathcal E(\bar w)\big]
+\varepsilon\langle\mathcal M(w(t)), w'(t)\rangle_{*}.
\end{align}
Furthermore, recall that $\Phi'(t)\to 0$ and $e^{2u(t)}v'(t)\to 0$ 
uniformly on $\mathbb S^2$ as $t\to\infty$ by Corollary \ref{cor:L2},
and that $\mathcal E(w(t))\to \mathcal E(\bar w)$ 
as $t\to\infty$ by Proposition \ref{prop:funct-conv}. 
Thus, $H(t)\to 0$  as $t\to\infty$. 
For simplicity of notation, we will suppress the dependence on $t$ below. 
Then a simple differentiation of \eqref{eq-definition-H-PDE} produces
\begin{align*}
H'&=-\|w'\|^2_{*}
+\varepsilon\big[-\langle\mathcal M(w), w'\rangle_{*}
+\langle\mathcal M(w), w''\rangle_{*}\\
&\qquad\qquad+\langle\mathcal M'(w)w', w'\rangle_{*}
+4(e^{4u}u'\mathcal M_2(w), w_2')_{L^2(\mathbb S^2)}\big],
\end{align*}
where $\mathcal M'(w)$ is the linearized operator of $\mathcal M$ at $w$, 
and the last term is given by the usual $L^2$-inner product on $\mathbb S^2$ 
and arises from differentiating the weight $e^{4u}$. We now write 
\begin{align}\label{eq-expression-H-der}
H'&=-\|w'\|^2_{*}
-\varepsilon \|\mathcal M(w)\|^2_{*} +\varepsilon[I_1+I_2+I_3],
\end{align}
where 
\begin{equation*}
I_1 =\langle\mathcal M(w), w''-w'+\mathcal M(w)\rangle_{*},\quad
I_2 =\langle\mathcal M'(w)w', w'\rangle_{*},\quad
I_3 =4(e^{4u}u'\mathcal M_2(w), w_2')_{L^2(\mathbb S^2)}. 
\end{equation*}
To estimate $I_3$, observe that since $\omega$ is independent of $t$ we have 
$u'=\Phi'$ and hence $|u'|\le C$ by Theorem  \ref{thm:reg}, so that
\begin{align}\label{eq-estimate-I3}\begin{split}
I_3&\le C(e^{4u}|\mathcal M_2(w)|, |w_2'|)_{L^2(\mathbb S^2)}
\le C\|e^{2u}w_2'\|_{L^2(\mathbb S^2)}
\|e^{2u}\mathcal M_2(w)\|_{L^2(\mathbb S^2)}
\\
&\le C\|w'\|_{*}
\|\mathcal M(w)\|_{*}.
\end{split}\end{align}
Next, by \eqref{eq:main-w} and the definition of the 
$L^2(\mathbb S^2, 1\times e^{4u})$-inner product, we find
$$I_1=\int_{\mathbb{S}^2}\big(2e^{4u}|v'|^2\mathcal M_1(w)
- 4e^{4u} u'v'\mathcal M_2(w)\big)\,  d  vol_{g_0}.$$
Furthermore, using Theorem \ref{thm:reg} again yields
$e^{2u}|v'|\le C$ and $|u'|\le C$, and hence 
\begin{align}\label{eq-estimate-I1}
I_1\le C\int_{\mathbb{S}^2}\big(e^{2u}|v'|\mathcal M_1(w)
+ e^{4u} |v'|\mathcal M_2(w)\big)\,  d  vol_{g_0}
\le C\|w'\|_{*}
\|\mathcal M(w)\|_{*}.
\end{align}
Lastly, in order to estimate $I_2$ note that \eqref{eq:main-mathcal-M} and $\w'=0$ produce
\begin{align*}
\mathcal{M}'_1(w)w'&= \Delta_{g_0} w'_1 -4 e^{4u} \nabla_{g_0}  w_2\cdot \nabla_{g_0}  w'_2 
-8 e^{4u} u'|\nabla_{g_0}  w_2 |^2, \\
\mathcal{M}'_2(w)w'&=\Delta_{g_0} w'_2 +4\nabla_{g_0}  u'\cdot \nabla_{g_0}  w_2
+4\nabla_{g_0}  u\cdot \nabla_{g_0}  w'_2,
\end{align*}
and write
\begin{align*}
I_2=\int_{\mathbb{S}^2}J\,  d  vol_{g_0}
\end{align*}
where 
\begin{align*}
J&= \big(\Delta_{g_0} w'_1 -4 e^{4u} \nabla_{g_0}  w_2\cdot \nabla_{g_0}  w'_2 
-8 e^{4u} u'|\nabla_{g_0}  w_2 |^2\big)w_1' \\
&\qquad+e^{4u}\big(\Delta_{g_0} w'_2 +4\nabla_{g_0}  u'\cdot \nabla_{g_0}  w_2
+4\nabla_{g_0}  u\cdot \nabla_{g_0}  w'_2\big)w_2'. 
\end{align*}
A straightforward computation gives
\begin{align*}
J&= \mathrm{div}_{g_0}\big(w_1'\nabla_{g_0} w'_1+e^{4u}w_2'\nabla_{g_0} w'_2\big) 
- |\nabla_{g_0}  w'_1 |^2-e^{4u} |\nabla_{g_0}  w'_2 |^2 \\
&\qquad-4 e^{4u} w_1'\nabla_{g_0}  w_2\cdot \nabla_{g_0}  w'_2 
-8 e^{4u} u'w_1'|\nabla_{g_0}  w_2 |^2
+4e^{4u}w_2'\nabla_{g_0}  u'\cdot \nabla_{g_0}  w_2. 
\end{align*}
Moreover since $u'=w_1'$, and $e^{2u} |\nabla_{g_0}  w_2|\le C$ by Theorem  \ref{thm:reg}, we have 
\begin{align*}
J&\le \mathrm{div}_{g_0}\big(w_1'\nabla_{g_0} w'_1+e^{4u}w_2'\nabla_{g_0} w'_2\big) 
- |\nabla_{g_0}  w'_1 |^2-e^{4u} |\nabla_{g_0}  w'_2 |^2 \\
&\qquad +C\big( e^{2u} |w_1'||\nabla_{g_0}  w'_2 |
+|w_1'|^2
+e^{2u}|w_2'||\nabla_{g_0}  w_1'|\big). 
\end{align*}
In addition, with the help of Young's inequality
\begin{align*}
J&\le \mathrm{div}_{g_0}\big(w_1'\nabla_{g_0} w'_1+e^{4u}w_2'\nabla_{g_0} w'_2\big) 
- \frac12\big(|\nabla_{g_0}  w'_1 |^2+e^{4u} |\nabla_{g_0}  w'_2 |^2\big) \\
&\qquad +C\big( |w_1'|^2+e^{4u}|w_2'|^2\big), 
\end{align*}
and hence
\begin{align}\label{eq-estimate-I2}
I_2\le C\int_{\mathbb{S}^2}\big( |w_1'|^2+e^{4u}|w_2'|^2\big)\,  d  vol_{g_0}
\le C\|w'\|_{*}^2.
\end{align}
By substituting \eqref{eq-estimate-I3}, \eqref{eq-estimate-I1}, 
and \eqref{eq-estimate-I2} into \eqref{eq-expression-H-der}, we obtain 
\begin{align*}
H'\le-\|w'\|^2_{*}
-\varepsilon \|\mathcal M(w)\|^2_{*} 
+\varepsilon C\big(\|w'\|_{*}
\|\mathcal M(w)\|_{*}
+\|w'\|_{*}^2\big).
\end{align*}
Furthermore, by fixing $\varepsilon$ sufficiently small it follows that
\begin{align}\label{eq-estimate-H-derivative-PDE}
H'\le -\frac12\|w'\|^2_{*}
-\frac\varepsilon2 \|\mathcal M(w)\|^2_{*}
\le -c_0\big(\|w'\|_{*}
+ \|\mathcal M(w)\|_{*}\big)^2,
\end{align}
where $c_0$ is a positive constant. 
In particular $H'\le 0$, and consequently
$H\ge 0$ on $(0,\infty)$ since $H(t)\to 0$ as $t\to\infty$. 

We now seek a differential inequality for $H$, that will allow us to determine its
rate of decay. Observe that \eqref{eq-definition-H-PDE-v2} implies
\begin{align*}
H&\le -\frac12\|w'\|^2_{*}
+2|\mathcal E(w)-\mathcal E(\bar w)|
+\|\mathcal M(w)\|_{*}\|w'\|_{*}\\&
\le 2|\mathcal E(w)-\mathcal E(\bar w)| + \|\mathcal M(w)\|_{*}^2,
\end{align*}
where we used Cauchy's inequality.  
We now restore $t$, and recall that $\eta>0$ is as in Proposition \ref{prop:functional-F}. 
The assumption \eqref{eq10-C2alpha-PDE-order2} then yields, via Proposition \ref{prop:functional-F},
that
$$|\mathcal E(w(t))-\mathcal E(\bar w)|^{1/2}
\le C\|\mathcal M(w(t))\|_{*}$$ 
for all $t\in(t_*,\sigma]$, and hence 
\begin{align}\label{eq-estimate-H-upper-PDE}
H(t)\le C\|\mathcal M(w(t))\|_{*}^2 \le C\big(\|w'(t)\|_{*}^2
+\|\mathcal M(w(t))\|_{*}^2 \big). 
\end{align}
Combining with \eqref{eq-estimate-H-derivative-PDE} gives rise to
\begin{align*}
H(t)\le C\big(\|w'(t)\|_{*}
+\|\mathcal M(w(t))\|_{*}\big)^2\le -CH'(t),\end{align*}
and thus 
\begin{align}\label{eq-differential-ineq-PDE-order2}
CH'(t)+H(t)\le 0.
\end{align}
Integrating this first-order differential inequality produces
\begin{equation}\label{eq-estimate-H-PDE-order2}H(t)\le 
C_*e^{-2\beta(t-t_*)} \quad \mbox{for all }t \in(t_* ,\sigma], 
\end{equation} 
where $\beta$ and $C_*$ are positive constants.
 
This decay for $H$ may be translated into decay for $w'$ as follows. Observe that an alternate combination 
of \eqref{eq-estimate-H-derivative-PDE} and \eqref{eq-estimate-H-upper-PDE}, for $t\in (t_*,\sigma]$, produces
\begin{align*}
-[H^{1/2}(t)]'=-\frac12 H^{-1/2}(t) H'(t)
\ge c\big(\|w'(t)\|_{*}
+\|\mathcal M(w(t))\|_{*}\big)
\ge c\|w'(t)\|_{*}.
\end{align*}
Furthermore, integrating from $t$ to $\sigma$ then yields
$$\int_{t}^\sigma \|w'(t)\|_{*}ds
\le C\big(H^{1/2}(t)-H^{1/2}(\sigma)\big)\le CH^{1/2}(t).$$
Therefore, \eqref{eq-estimate-H-PDE-order2} may be applied to obtain
\begin{equation*}
\int_{t}^\sigma \|w'(t)\|_{*}ds\le 
C_*e^{-\beta(t-t_*)}  \quad \mbox{for all }t\in (t_*,\sigma]. 
\end{equation*} 
This finishes the proof of \eqref{eq11-L1-bound-PDE-order2-v2}.

{\it Step 2.}  We now assume \eqref{eq10-PDE-order2}, and show how to choose $\gamma$ in order to achieve
\eqref{eq10-C2alpha-PDE-order2}.
Note that both $w=(\Phi,v)$ and $\bar w=(\bar \Phi,\bar v)$ satisfy equation \eqref{eq:main-2}.  
Therefore, by Proposition \ref{prop:difference-linear} and 
\eqref{eq10-PDE-order2}
we have 
\begin{equation*}| w(t) - \bar w |^*_{2,\alpha}\le 
C\Big(\int_{t-1}^{t+1}\|w(s)-\bar w\|_{L^2(\mathbb S^2, 1\times \w^{-4})}^2ds\Big)^{1/2}
< C\gamma,
\end{equation*}
for all $t\in [t_*+1, \sigma-1]$.
We may then choose $\gamma$ such that $C\gamma<\eta$, to find that \eqref{eq10-C2alpha-PDE-order2} holds
for $t\in (t_*+1, \sigma-1]$. The desired estimate \eqref{eq11-L1-bound-PDE-order2} now follows from Step 1. 
\end{proof}

The hypotheses of the previous result can be weakened. More precisely, the assumption \eqref{eq10-PDE-order2} requires closeness to the tangent map on a time interval, however this may be reduced to closeness at a single time.

\begin{lem}\label{lem:step2} 
Let $\w$ satisfy \eqref{eq:NS-singular},  
$w=(\Phi,v)$ be a weak solution of \eqref{eq:main-2} 
in $\mathbb R_+ \times \mathbb{S}^2$ 
satisfying \eqref{eq:ubd-1}, \eqref{eq:ea-1},  
and \eqref{eq:trace-1} for some positive constants $\Lambda$ and $a$, 
let $\bar w=(\bar \Phi, \bar v)\in C^{3}(\mathbb{S}^2)$ be a solution of 
\eqref{eq:main-s2-1}-\eqref{eq:main-s2-1-poles} as in Proposition \ref{prop:sub-conv}.
Then there exist positive constants $\delta$, $\beta$, and $ C$
depending only on $\bar w$ and $\Lda$, such that the following holds.
If for some $t_*\ge 5$ the estimate
\begin{equation}\label{eq-closeness-assumption-PDE-order2}
\| w(t_*) - \bar w \|_{L^2(\mathbb S^2, 1\times \w^{-4})}  +
| \mathcal E(w(t_*)) - \mathcal E(\bar w) |^{1/2}  < \delta
\end{equation}
is valid, then
\begin{equation}\label{eq-closeness-conclusion-PDE-order2}
\|w(\sigma)- w(t) \|_{L^2(\mathbb S^2, 1\times \w^{-4})}  \le 
Ce^{-\beta(t-t_*)}
\end{equation}
for all $\sigma\ge t> t_*+1$.
\end{lem}

\begin{proof} 
We will make use of the weight $1\times \w^{-4}$, and in particular
will write $\|\cdot\|_{*} =\|\cdot \|_{L^2(\mathbb S^2, 1\times \w^{-4})}$ for brevity.
It should be emphasized that here, $\|\cdot\|_{*}$ represents a different norm from that in the proof of 
Lemma \ref{lem:step1}, due to the choice of weight. First observe that if $\sigma\ge t\ge t_*$ then 
\begin{align*}
\|w(\sigma)-w(t)\|_{*}^2&=\Big\|\int_t^\sigma w'(s)ds\Big\|^2_{*}
\le (\sigma-t)\int_t^\sigma\|w'(s)\|^2_{*}ds\\
&\le (\sigma-t)\int_{t_*}^\infty\|w'(s)\|^2_{*}ds\\
&\le e^{4\Lda}(\sigma-t)\big(\mathcal E(w(t_*))-\mathcal E(\bar w)\big),
\end{align*}
where in the last inequality we replaced the weight 
$\w^{-4}$ by $e^{4u}$ and used \eqref{eq:energy-identity-infty-v2}. 
Furthermore, if in addition $\sigma-t\le T$ then 
\begin{align}\label{eq13-PDE-order2}
\|w(\sigma)-w(t)\|_{*}
\le e^{2\Lda}\sqrt{T} | \mathcal E(w(t_*))-\mathcal E(\bar w) |^{1/2}.
\end{align}

Let $\gamma$, $\beta$, and $C_*$ be as in Lemma \ref{lem:step1}. 
Choose $T>5$ such that
\begin{equation}\label{eq14-PDE-order2}
C_*e^{-\beta T}  < \frac14\gamma,
\end{equation}
and then choose $\delta$ (depending on $T$) sufficiently small so that \eqref{eq-closeness-assumption-PDE-order2} implies 
\begin{equation}\label{eq15-PDE-order2}
e^{2\Lda}\sqrt{T}|  \mathcal E(w(t_*)) - \mathcal E(\bar w) |^{1/2}<\frac14\gamma,
\end{equation}
and
\begin{equation}\label{eq16-PDE-order2}
\| w(t_*) - \bar w\|_{*}  < \frac14\gamma.
\end{equation} 
By \eqref{eq13-PDE-order2} and \eqref{eq15-PDE-order2} we then have
\begin{equation}\label{eq17-PDE-order2}
\| w(\sigma) - w(t)\|_{*} < \frac14\gamma,
\end{equation} 
for all $\sigma\geq t\ge t_*$ with $\sigma-t\le T$.
In particular, if $t  \in [t_*, t_*+T]$ then
\begin{equation*}
\| w(t) - w(t_*)\|_{*} < \frac14\gamma,
\end{equation*} 
and hence with the help of \eqref{eq16-PDE-order2} we obtain
\begin{align}\label{eq18-PDE-order2}
\| w(t) - \bar w \|_{*} \le \| w(t ) - w(t_*)\|_{*}  + \|w(t_*) - \bar w\|_{*} 
< \frac14\gamma  + \frac14\gamma   = \frac12\gamma.
\end{align}
We now claim that
\begin{equation}\label{eq19-PDE-order2}
\| w(t ) - \bar w\|_{*}< \frac34\gamma\quad\text{ for any }t\ge t_*.
\end{equation}
Notice that this shows that \eqref{eq10-PDE-order2} holds for any $t> t_*$, and therefore by \eqref{eq11-PDE-order2} of Lemma \ref{lem:step1} the desired estimate 
$$\| w(\sigma ) - w(t )\|_{*}
\le C_*e^{-\beta( t-t_*)},$$ 
is valid for all $\sigma\geq t> t_*+1$.

It remains to establish \eqref{eq19-PDE-order2}, which will be accomplished by a contradiction argument. Set 
$$\overline{t}=\sup\{t\ge t_* \mid \, \|w(s)-\bar w\|_{*}<3\gamma/4\ \text{ for all }s\in [t_*,t]\},$$
and note that according to \eqref{eq18-PDE-order2} we have that $\overline{t}$ is well-defined with $\overline{t}>t_*+T$. 
If $\overline{t}<\infty$ then 
$$\|w(t)-\bar w\|_{*}\le \frac34\gamma$$
for all $t\in [t_*,\overline{t}]$, and by
\eqref{eq17-PDE-order2} we find
\begin{equation*}
\| w(t) - w(\overline{t})\|_{*} < \frac14\gamma
\end{equation*} 
for all $t\in [\overline{t}, \overline{t}+T]$. For $t$ in this same range it follows that
\begin{equation*}
\|w(t)-\bar w\|_{*}
\le \| w(t) - w(\overline{t})\|_{*} +\|w(\overline{t})-\bar w\|_{*}
< \frac14\gamma+\frac34\gamma=\gamma.
\end{equation*} 
As a consequence, \eqref{eq10-PDE-order2} holds for $t \in [t_*, \overline{t}+T]$, and therefore we may apply
\eqref{eq11-PDE-order2} of Lemma \ref{lem:step1} with $t=t_*+T$ and an arbitrary 
$t'\in [\overline{t}, \overline{t}+T-1]$, together with \eqref{eq14-PDE-order2}, to obtain
\begin{equation}\label{eq-estimate-t-bar-PDE}
\| w(t') - w(t_*+T)\|_{*} \le
C_*e^{-\beta T} < \frac14\gamma.
\end{equation} 
Using this and \eqref{eq18-PDE-order2} with $t=t_*+T$ produces
\begin{align*}\|w(t') - \bar w\|_{*} 
&\le \| w(t') -w(t_*+T)\|_{*} +\| w(t_*+T) - \bar w\|_{*}\\
& < \frac14\gamma+\frac12\gamma=\frac34\gamma,
\end{align*}
for any $t' \in [t_*, \overline{t}+T-1]$. It follows that a contradiction is achieved with
the definition of $\overline{t}$. We conclude that $\overline{t}=\infty$, and hence
\eqref{eq19-PDE-order2} is established. 
\end{proof}

With all of the preparatory work now completed, we are ready 
establish the uniqueness of limit points for the harmonic map upon approach
to the prescribed singularity. 

\begin{thm}\label{thm:main-1a} 
Let $\w$ satisfy \eqref{eq:NS-singular}, and let
$w=(\Phi,v)$ be a weak solution of \eqref{eq:main-2} 
in $\mathbb R_+ \times \mathbb{S}^2$ 
satisfying \eqref{eq:ubd-1}, \eqref{eq:ea-1},  
and \eqref{eq:trace-1} for some positive constants $\Lambda$ and $a$.
Then $w(t)\to \bar w$ in $C^3(\mathbb S^2)$ as $t\to\infty$, 
for some solution $\bar w=(\bar \Phi,\bar v)\in C^{3}(\mathbb{S}^2)$ of 
\eqref{eq:main-s2-1}-\eqref{eq:main-s2-1-poles},
and there exist positive constants $\beta$, $C$
depending only on $\Lda$, $\bar w$
such that  
\[
|\big(\Phi(t), v(t)\big)-(\bar \Phi,\bar v)|_{2}^* \le 
 C e^{-\beta t} 
\]
for all $t\in (0,\infty)$.
\end{thm}

\begin{proof} 
By Theorem \ref{thm:reg} and Proposition \ref{prop:sub-conv}, 
there exists a sequence of positive times $t_i\to\infty$ such that 
$w(t_i)\to \bar w$ in $C_*^3(\mathbb S^2)$ and 
$\mathcal E(w(t_i))\to \mathcal E(\bar w)$ as $i\to\infty$, 
for some solution $\bar w=(\bar \Phi,\bar v)\in C^{3}(\mathbb{S}^2)$ of 
\eqref{eq:main-s2-1}, \eqref{eq:main-s2-1-poles}. 
Take $t_*=t_i$ for a sufficiently large $i$ such that 
\eqref{eq-closeness-assumption-PDE-order2} holds. 
It follows that \eqref{eq-closeness-conclusion-PDE-order2} is valid for all $\sigma\ge t> t_*+1$.
Therefore $w(t)$ converges to $\bar w$ in $L^2(\mathbb S^2, 1\times \w^{-4})$ as $t\to\infty$. 
Moreover, by taking $\sigma\to\infty$ in \eqref{eq-closeness-conclusion-PDE-order2} we find that
$$\|w(t)- \bar w\|_{L^2(\mathbb S^2, 1\times \w^{-4})}  \le 
Ce^{-\beta (t-t_*)} \quad \mbox{for all }t> t_*+1.
$$
As in Step 2 of the proof for Lemma \ref{lem:step1}, the weighted H\"{o}lder norm may be estimated by
\begin{equation*}| w(t) - \bar w|_{2}^*\le 
C\Big(\int_{t-1}^{t+1}\|w(s)-\bar w\|_{L^2(\mathbb S^2, 1\times \w^{-4})}^2ds\Big)^{1/2}
\end{equation*}
for all $t>t_*+2$, and hence for $t$ in the same range we have
$$| w(t) - \bar w|_{2}^*\le 
Ce^{-\beta (t-t_*)},
$$
yielding the desired result.
\end{proof}

\begin{proof}[Proof of Theorem \ref{thm:main-1}]  
Since $e'(u+\ln (r\sin \theta), v)\in L^1_{loc} (B_1\setminus \{0\})$,  
it may be verified that $(\Phi,v)=(u+\ln \sin\theta, v)$ is a weak solution 
of \eqref{eq:main-2} satisfying \eqref{eq:ubd-1}, \eqref{eq:ea-1}, and \eqref{eq:trace-1}. 
Then Theorem \ref{thm:main-1} follows from 
Proposition \ref{prop:difference-linear} and Theorem \ref{thm:main-1a}, by taking $\w=\sin \theta$ and setting $(\tilde{\Phi},\tilde{v})=(\bar{\Phi},\bar{v})$. 
\end{proof}

\section{Convergence: Linear Growth Renormalization}
\label{s:convergence-2}

The purpose of this section is to study asymptotics of the harmonic map at spatial infinity,
which corresponds to the asymptotically flat end of the associated stationary vacuum spacetime.
In terms of the notation established at the beginning of this work, the radial coordinate $t=-\ln r$ will then
be restricted to the range $\mathbb R_-=(-\infty,0)$, and we will use the renormalization function
\be\label{eq:t-gth-omega}
\w=e^{-t} \sin \theta,
\end{equation}
which in cylindrical coordinates is simply $\rho$. Notice that this renormalization function is not independent
of $t$, in contrast to the translation invariant renormalization used in the previous section. Moreover, this dependence on $t$
will require additional care when applying the results of Section \ref{s:regularity}. When the invariant renormalization
function is needed below it will be denoted by $\bar \w= \sin \theta$.

We will study the asymptotic behavior of solutions 
$(\Phi,v)\in H^{1}_{\mathrm{loc}}( \mathbb R_- \times \mathbb{S}^2)$  
to the renormalized harmonic map system 
\begin{align}\label{eq:main-2z} \begin{split}
L \Phi -2 e^{4u} |\nabla_{g} v|^2  &=L\ln\w , \\
Lv +4\nabla_{g}u\cdot   \nabla_{g} v&=0,
\end{split}\end{align}
as $t\to-\infty$. 
Observe that with this renormalization we have $L\ln\w =0$.
Here and hereafter, the unrenormalized and renormalized harmonic map component functions $u$ and $\Phi$ will be related by 
$u=\Phi-\ln \w$. Furthermore, it will be assumed that there are positive constants $\Lambda$ and $a$ such that
\begin{equation} \label{eq:ubd-1z} 
|\Phi|+|v| \le \Lda \quad \text{ in }\mathbb R_- \times \mathbb{S}^2,
\end{equation}
for any finite interval $I\subset \mathbb R_-$ we have
\begin{equation}
\label{eq:ea-1z}
\int_{I} \int_{\mathbb{S}^2} (| \nabla_{g}\Phi|^2 + \w^{-4} | \nabla_{g}v|^2 )\, d  vol_{g_0}  d  t <\infty,
\end{equation}
and in the trace sense 
\begin{equation} \label{eq:trace-1z}
v= a  \text{ on  }\mathbb R_- \times \{N\}, \quad \quad
v=-a  \text{ on  }\mathbb R_- \times \{S\},
\end{equation} 
where $N$ and $S$ represent the north and south poles of $\mathbb S^2$. 
The next result may be viewed as a compilation of Theorem \ref{thm:reg}, Proposition \ref{prop:v-linear},
and Proposition \ref{prop:partial-t-linear} tailored to the current setting.

\begin{prop}\label{prop:reg-exterior} 
Let $(\Phi,v)\in H^{1}_{\mathrm{loc}}(\mathbb R_-\times\mathbb S^2)$ be a weak solution of \eqref{eq:main-2z} 
satisfying \eqref{eq:ubd-1z}, \eqref{eq:ea-1z}, and \eqref{eq:trace-1z}
for some positive constants $\Lambda$ and $a$.  Then for any $T<-2$, integers $j,m=0,1,2,3$, 
$k=1,2,3$, and parameter $\alpha\in (0,1)$, the following estimates hold
\[
| \pa_t^m (\Phi, e^{2T} v)|_{C^{3}((T-1,T+1) \times \mathbb{S}^2)}  \le C, 
\]
\begin{align*}
\left|\nabla_{g_0}^j \pa_t^m ( v-a)\right|  
&\le C\bar\w^{3+\alpha-j} \|\bar\w^{-2}  \nabla_{g}  v \|_{L^2((T-2,T+2) \times \mathbb{S}^2)}
\quad\text{ on }[T-1,T+1] \times \mathbb{S}^2 _+,\\  
\left|\nabla_{g_0}^j \pa_t^m ( v+a)\right|  
&\le C\bar\w^{3+\alpha-j} \|\bar\w^{-2}  \nabla_{g}  v \|_{L^2((T-2,T+2) \times \mathbb{S}^2)}
\quad\text{on }[T-1,T+1] \times \mathbb{S}^2 _-,
\end{align*}
and 
\begin{align*}
& |\pa_t^k \nabla_{g_0}^j   \Phi| +e^{2T}\bar\w^{j-3-\alpha}|\pa_t^k \nabla_{g_0}^j  v|\\
&\qquad \le  C\Big( \int_{T-2}^{T+2} \int_{\mathbb{S}^2}\big( | \pa_t   \Phi|^2 
+e^{4u} |\nabla_{g}   v|^2\big) \, d  vol_{g_0}  d  t \Big)^{1/2}
\quad\text{ on }[T-1,T+1] \times \mathbb{S}^2,
\end{align*} 
for some positive constant $C$ depending only on $\alpha$, $a$, and $\Lda$.  
\end{prop}

\begin{proof}
For $t\in (-2,2)$ consider the functions
\be\label{eq:change-varibles-ex}
\tilde \Phi(t,\cdot)=\Phi(T+t,\cdot), \quad  \tilde v(t,\cdot)=e^{2T} v(T+t,\cdot), 
\quad \tilde u(t,\cdot)=\tilde \Phi(t,\cdot)-\ln \w(t,\cdot),
\end{equation} 
and observe that by \eqref{eq:main-2z} and $L\ln\w=0$ they satisfy the equations
\begin{align*}
L\tilde \Phi -2 e^{4\tilde  u} |\nabla_g \tilde v|^2 &=0, \\
L\tilde v +4\nabla_g  \tilde  u  \cdot \nabla_g  \tilde  v &=0,
\end{align*}
on $(-2,2) \times \mathbb{S}^2$. Furthermore, boundary conditions \eqref{eq:trace-1z} are satisfied 
with $a$ replaced by $e^{2T} a$. 
We may then apply Theorem \ref{thm:reg}, Proposition \ref{prop:v-linear}, 
and Proposition \ref{prop:partial-t-linear} to obtain 
\[
|\pa_t^m( \tilde \Phi, \tilde v) |_{C^{3}((-1,1) \times \mathbb{S}^2)}  \le C, 
\]
\begin{align*}
\left|\nabla_{g_0}^j \pa_t^m(\tilde v-e^{2T}a)\right| 
& \le C\bar\w^{3+\alpha-j} \|\bar\w^{-2}  \nabla_g \tilde v \|_{L^2((-2,2) \times \mathbb{S}^2)} 
\quad \text{ on }[-1,1] \times \mathbb{S}^2 _+, \\
\left|\nabla_{g_0}^j \pa_t^m(\tilde v+e^{2T}a)\right| 
&  \le C\bar\w^{3+\alpha-j} \| \bar\w^{-2}  \nabla_g \tilde v\|_{L^2((-2,2) \times \mathbb{S}^2)} 
\quad \text{ on }[-1,1]  \times \mathbb{S}^2 _-, 
\end{align*}
and
\begin{align*}
|\pa_t^k \nabla_{g_0}^j \tilde  \Phi| +\bar\w^{j-3-\alpha}|\pa_t^k \nabla_{g_0}^j \tilde v| 
\le  C\Big( \int_{-2}^{2} \int_{\mathbb{S}^2}\big( | \pa_t  \tilde \Phi|^2 
+e^{4 \tilde  u} |\nabla_g  \tilde v|^2\big) \, d  vol_{g_0}  d  t \Big)^{1/2}
\end{align*}
on $[-1,1]\times \mathbb{S}^2$, for $j,m=0,1,2,3$, $k=1,2,3$, and $\alpha\in (0,1)$. 
Note that the positive constant $C$ depends only on $\alpha$, $a$, and $\Lda$, and in order to satisfy the hypotheses of
Proposition \ref{prop:partial-t-linear} we have used the relations $\partial_t \ln \w =-1$ and $L\ln\w=0$. The desired result now follows by translating back to the original map $(\Phi,v)$.
\end{proof}

It is possible to use standard energy estimates to bound the weighted gradient of $v$, and thereby 
improve some of the estimates for the twist potential.

\begin{prop}\label{prop:reg-exterior-improved} 
Let $(\Phi,v)\in H^1_{\mathrm{loc}}(\mathbb{R}_-\times\mathbb S^2)$ be a weak solution of \eqref{eq:main-2z} 
satisfying \eqref{eq:ubd-1z}, \eqref{eq:ea-1z}, and \eqref{eq:trace-1z}
for some positive constants $\Lambda$ and $a$.  Then for any $T<-3$, integers
$j,m=0,1,2,3$, and parameter $\alpha\in (0,1)$,  the following estimates hold
\begin{equation} \label{eq:hat-v-est}
\begin{split}
\left|\nabla_{g_0}^j \pa_t^m ( v-a)\right| & \le C \bar\w^{3+\alpha-j} 
\quad \text{ on }[T-1,T+1] \times \mathbb{S}^2 _+, \\
\left|\nabla_{g_0}^j \pa_t^m ( v+a)\right|  &\le C\bar\w^{3+\alpha-j} 
\quad \text{ on }[T-1,T+1] \times \mathbb{S}^2 _-, 
\end{split}
\end{equation}
for some positive constant $C$ depending only on $\alpha$, $a$, and $\Lda$.  
\end{prop}

\begin{proof}   
Let $\tilde  \Phi(t,\cdot)$ and  $\tilde u(t,\cdot)$ be given as in \eqref{eq:change-varibles-ex} for $T<-3$ and $t\in (-3,3)$, 
and define $\hat v(t,\cdot)= v(T+t,\cdot)$. Then
\begin{equation} \label{eq:exter-final-1}
\mathrm{div}_g( e^{4\tilde u} \nabla_g \hat v)(t)
=  e^{-4T}\mathrm{div}_g( e^{4 u(T+t) } \nabla_g  v(T+t))= 0 \quad \mbox{ on }(-3,3)\times \mathbb{S}^2,
\ee
with $\hat v(t,N)=a$ and $\hat v(t,S) =-a$. By applying Lemma \ref{lem:LT} to $\hat v$, while using \eqref{eq:ubd-1z}
as well as the estimates for $\Phi$ in Proposition \ref{prop:reg-exterior}, we obtain
\begin{align*}
|(\hat v-a)| &\le C \bar \w^{3+\alpha} \quad \mbox{ on }\left({\scriptstyle -\frac{5}{2}},{\scriptstyle \frac{5}{2}}\right) \times \mathbb{S}^2_+,\\ 
|(\hat v+a)| &\le C \bar \w ^{3+\alpha} \quad \mbox{ on }\left({\scriptstyle -\frac{5}{2}},{\scriptstyle \frac{5}{2}}\right) \times \mathbb{S}^2_-, 
\end{align*}
where $\alpha\in (0,1)$ is arbitrary and $C$ is a positive constant depending only on $\alpha$, $a$, and $\Lda$.  
We may now use local energy estimates for \eqref{eq:exter-final-1}, viewed as a linear equation of $\hat v\pm a$, to obtain
bounds for weighted derivatives of $\hat{v}$. More precisely, multiplying the equation by $\hat v \pm a$ with an appropriate 
cut-off function and integrating by parts produces 
\[
\|\bar \w ^{-2} \nabla_g \hat v\|_{L^2((-2,2) \times \hat{\mathbb{S}}^2_{\pm})} 
\le C \|\bar \w^{-2} (\hat v \pm a)\|_{L^2\left(({\scriptscriptstyle -\frac{5}{2}},{\scriptscriptstyle \frac{5}{2}}) \times \mathbb{S}^2_{\pm}\right)}\le C, 
\]
where $\hat{\mathbb{S}}^2_{\pm}$ is a domain in the sphere that is slightly smaller than $\mathbb{S}^2_{\pm}$.
It follows that
\[
\|\bar \w ^{-2} \nabla_g  v\|_{L^2((T-2,T+2) \times \mathbb{S}^2)}\le C,
\]
and the desired result is then a consequence of Proposition \ref{prop:reg-exterior}.  
\end{proof}

The asymptotic analysis of harmonic maps near spatial infinity will reduce to the asymptotic
behavior of two associated linear equations, which are studied in the next two results. The fact that the analysis reduces to the study of two linear equations, indicates that the system effectively decouples in this regime.

\begin{lem}\label{lem:phi-linear-pde} 
Let $\Phi\in C^{2}(\mathbb R_-\times \mathbb{S}^2)$ be a bounded solution of 
\[
L\Phi=f \quad \mbox{ on }\mathbb R_-\times \mathbb{S}^2,
\]
where $f\in C^{0}(\mathbb R_- \times\mathbb{S}^2)$ satisfies $| f|\le Ke^{\beta t}$ for some positive constants $\beta$ and $K$. 
If $\beta$ is not an integer, then there exist a constant $c_0$ and degree $l$ spherical harmonics $Y_l$, 
for $l=0, \cdots, [\beta]-1$, such that for any $\alpha\in (0,1)$ the following expansion holds
\[
|\Phi(t) - c_0-Y_0 e^{t}-\cdots-Y_{[\beta]-1}e^{[\beta]t}|_{C^{1,\alpha}(\mathbb{S}^2)} \le Ce^{\beta t},
\]
where $C$ is a positive constant depending only on $\alpha$, $\beta$, $K$, 
and the $L^\infty$-norm of $\Phi$ on $\mathbb R_-\times \mathbb{S}^2$. 
\end{lem}

\begin{proof} 
This is based on a well-known argument, and thus we only outline the main points which are included for completeness. Let $\{X_m\}_{m=0}^\infty$ be an orthonormal basis of eigenfunctions for $L^2(\mathbb{S}^2)$, more precisely $-\Delta_{g_0} X_m=\lda_m X_m$ with 
$$\lda_0=0<\lda_1=\lda_2=\lda_3=2<\lda_4\le \cdots \rightarrow\infty.$$ 
We may expand $\Phi$ and $f$ according to these spherical harmonics as 
\[
\Phi(t)= \sum_{m=0}^{\infty} \Phi_{m}(t) X_{m},\quad\quad  f(t)=e^{\beta t}  \sum_{m=0}^{\infty} f_{m}(t) X_{m}, 
\]
where the $f_m(t)$ are uniformly bounded in $t$.  Then, the equation $L\Phi=f$ becomes
\[
\sum_{m=0}^\infty\big(\Phi_{m}''-\Phi_{m}'-\lda_m \Phi_{m} -e^{-\beta t}f_m\big)X_m=0,
\]
and hence 
\[
\Phi_{m}''-\Phi_{m}'-\lda_m \Phi_{m} -e^{-\beta t}f_m=0 \quad \mbox{ on }\mathbb R_-. 
\]
The two characteristic roots are given by
\[
\frac12\left(1\pm \sqrt{1+4\lda_m}\right). 
\]
Moreover, since $\lambda_m=i(i+1)$ for some integer $i\ge 0$, 
the corresponding characteristic roots are in fact $i+1$ and $-i$. 
We now discard the negative characteristic root, since the solution $\Phi$ is assumed to be bounded. 
In the homogeneous case when $f=0$, we then have $\Phi_m(t)=c_ie^{(i+1)t}$ for some constant $c_i$. In conclusion, 
any bounded solution $\Phi$ of $L\Phi=0$ can be expressed as 
$$\Phi(t)=c_0+\sum_{i=0}^{\infty} Y_{i}e^{(i+1)t},$$
where $c_0$ is a constant and $Y_i$ are spherical harmonics of degree $i$. The $L^\infty$ version of the expansion stated in this lemma now follows from standard ODE analysis. Finally, the $C^{1,\alpha}$ expansion may be obtained by scaled interior Schauder estimates.
\end{proof}

We now treat the asymptotic linear equation associated with the twist potential function $v$, 
from the harmonic map system. 

\begin{lem}\label{lem:v-linear-pde} 
Let $\xi\in C^{3,\alpha}(\mathbb R_-\times \mathbb{S}^2)$, $\alpha\in(0,1)$ be a bounded solution of 
\begin{align*}
\pa_t^2 \xi+3 \pa_t \xi  +\bar \w^4\mathrm{div}(\bar \w^{-4} \nabla_{g_0}\xi) =f 
&\quad \text{ on }\mathbb R_-\times \mathbb{S}^2,\\
\xi(t,N)=\xi(t,S)=0&\quad\text{ for }t\in \mathbb R_-,\end{align*} 
where for some positive constants $\beta$ and $K$ the function $f\in C^{1,\alpha}(\mathbb{R}_- \times \mathbb{S}^2)$ satisfies 
\begin{align}\label{eq-assumption-decay-f}
| f|\le Ke^{\beta t} \sin^2 \theta\quad\text{ on }\mathbb R_-\times\mathbb S^2.
\end{align} 
If $\beta\neq 1$ then 
\begin{align}\label{eq-conclusion-decay-xi}
\begin{split}
|\xi|&\le Ce^{\beta t}\sin^{3+\alpha} \theta \quad \text{ for }\beta< 1,\\
|\xi-c_1 e^{t} \bar \w^4 |&\le Ce^{\gamma t}\sin^{3+\alpha} \theta \quad \text{ for }\beta>1,
\end{split}
\end{align}
where $c_1$, $C>0$, and $\gamma>1$ are constants
depending only on $\alpha$, $\beta$, $K$, 
and the $L^\infty$-norm of $\xi$ on $\mathbb R_-\times \mathbb{S}^2$. 
If additionally $\partial_tf$ satisfies \eqref{eq-assumption-decay-f}, 
then $\partial_t\xi$ satisfies \eqref{eq-conclusion-decay-xi}.
Moreover if $\beta=1$, then $\xi$ may have a factor of $t$ in the expansion at order $e^t$.
\end{lem}

\begin{proof}  By Corollary 4.1 of \cite{LT} there exists a constant $C$ such that 
\begin{equation} \label{eq:li-tian-cor}
\int_{\mathbb{S}^2} w^2 \bar\w^{-6} \,  d  vol_{g_0} 
\le C \int_{\mathbb{S}^2} |\nabla_{g_0}w|^2 \bar\w^{-4} \,  d  vol_{g_0},
\end{equation}
for all $w\in H_0^{1}( \mathbb{S}^2, \bar\w^{-4})$. Then, for any $2<p<\infty$,
H\"older's inequality and the Sobolev inequality on $\mathbb{S}^2$ produce
\begin{align*}
\int_{\mathbb{S}^2} |w|^p \bar\w^{-4} \,  d  vol_{g_0}& \le 
\Big(\int_{\mathbb{S}^2} w^2 \bar\w^{-6} \,  d  vol_{g_0}\Big)^{2/3} 
\Big(\int_{\mathbb{S}^2} w^{3p-4} \,  d  vol_{g_0}\Big)^{1/3} \\
&
\le C \Big( \int_{\mathbb{S}^2} |\nabla_{g_0}w|^2 \bar\w^{-4} \,  d  vol_{g_0} \Big)^{p/2},
\end{align*}
where the constant $C$ depends only on $p$. This implies that $H_0^{1}( \mathbb{S}^2, \bar\w^{-4}) $ is compactly embedded 
in $L^2( \mathbb{S}^2, \bar\w^{-4})$, where this latter space consists of
$L^2$-functions on $\mathbb{S}^2$ with respect to the measure $\bar\w^{-4}  d  vol_{g_0}$. These observations may 
be used to show that the standard $L^2$-theory for elliptic equations in divergence form applies to the degenerate equation
\begin{equation*}
\bar\w^{4} \mathrm{div}_{g_0}(\bar\w^{-4} \nabla_{g_0}w)=h \quad \text{ on }\mathbb{S}^2.
\end{equation*}
In particular, for any $h\in L^2( \mathbb{S}^2, \bar\w^{-4})$ there exists a unique solution 
$w\in H_0^{1}( \mathbb{S}^2, \bar\w^{-4})$, and 
\begin{equation*}
\|\nabla_{g_0} w\|_{ L^2( \mathbb{S}^2, \bar\w^{-4})} \le C \|h\|_{ L^2( \mathbb{S}^2, \bar\w^{-4})}
\end{equation*}
for some universal constant $C$. Moreover, for the eigenvalue problem 
\begin{equation*}
\bar\w^{4} \mathrm{div}_{g_0}(\bar\w^{-4} \nabla_{g_0} w ) = -\mu w\quad \mbox{on }\mathbb{S}^2,
\end{equation*}
there exists a sequence of real eigenvalues $\mu_1<\mu_2\le \cdots $ converging to infinity, along with  
an associated sequence of eigenfunctions $\{\varphi_i\}_{i=1}^{\infty}$ that form an orthonormal basis of
$L^2( \mathbb{S}^2, \bar\w^{-4})$.  

Consider an eigenpair $(\mu, w)$. We may apply a De Giorgi iteration as in the proof of
\cite[Theorem 1]{N} to obtain
\begin{equation*}
\sup_{\mathbb{S}^2} |w| 
\le  C \|w\|_{ L^2( \mathbb{S}^2, \bar\w^{-4})}.
\end{equation*}
It then follows from Lemma \ref{lem:LT} that $w\in C^{3,\alpha}(\mathbb S^2)$, for any $\alpha\in (0,1)$, 
and 
\begin{equation*}
|w|\le C\bar\w^{3+\alpha}\quad\text{on }\mathbb S^2
\end{equation*}
where $C$ is a constant depending only on $\alpha$, $\mu$, and the $L^\infty$-norm of $w$. 
Note that $\bar\w^{4}\in H_0^{1}( \mathbb{S}^2, \bar\w^{-4})$ and 
\begin{equation*}
\bar\w^{4} \mathrm{div}_{g_0}(\bar\w^{-4} \nabla_{g_0} \bar\w^{4} ) = -4\bar\w^{4}.
\end{equation*}
Hence, 4 is an eigenvalue and $\bar\w^4$ is a corresponding eigenfunction. 
Since $\bar\w^4>0$ on $\mathbb S^2\setminus\{N, S\}$, we
conclude that 4 is the smallest eigenvalue, that is $\mu_1=4$.  

Next, we will establish an $L^2$-decay based on expansions in terms of the orthonormal basis 
$\{\varphi_i\}_{i=1}^{\infty}$. 
Write
\begin{equation*}
\xi(t)= \sum_{i=1}^\infty \xi_{i}(t) \varphi_i, \quad\quad f(t)= \sum_{i=1}^\infty f_{i}(t) \varphi_i,
\end{equation*}
and observe that
\begin{equation*}
\sum_{i=1}^\infty (\xi_{i}''(t)+3\xi_i'(t)-\mu_i \xi_i(t)+f_i(t))  \varphi_i =0, 
\end{equation*}
so that
\begin{equation*}
\xi_{i}''(t)+3\xi_i'(t)-\mu_i \xi_i(t)+f_i(t)=0. 
\end{equation*}
The two characteristic roots of this equation are 
\begin{equation*}
\frac12\big(-3\pm \sqrt{9+4\mu_i}\big). 
\end{equation*}
In particular, these become $1$ and $-4$ when $i=1$. 
By standard ODE analysis, we then have 
\begin{align*}
\|\xi(t)\|_{  L^2( \mathbb{S}^2, \bar\w^{-4})}&\le C e^{\beta t} \quad \mbox{ if }\beta< 1,\\
\|\xi(t)-c_1 e^{t} \bar \w^4 \|_{ L^2( \mathbb{S}^2, \bar\w^{-4})}&\le Ce^{\gamma t} \quad \mbox{ if }\beta>1,
\end{align*}
for some constants $c_1$, $C$, and $\gamma>1$.

The $L^2$-decay may be transformed into pointwise decay in the following way. 
For any $t_0<-1$ and $\tau\in (-1,1)$ set 
\begin{equation*}
\widetilde \xi(\tau,\cdot)= \begin{cases}
e^{-\beta t_0} \xi(t_0+\tau,\cdot) \quad \text{if }\beta<1,\\
e^{-\gamma t_0}  (\xi(t_0+\tau,\cdot)-c_3 e^{t_0+\tau} \bar \w^4)  \quad \text{if }\beta>1,
\end{cases}
\end{equation*}
and 
\begin{equation*}
\widetilde f(\tau,\cdot)= \begin{cases}
e^{-\beta t_0} f(t_0+\tau,\cdot) \quad \text{if }\beta<1,\\
e^{-\gamma t_0}  f(t_0+\tau,\cdot)  \quad \text{if }\beta>1.
\end{cases}
\end{equation*}
Then $\|\widetilde \xi(\tau)\| _{L^2( \mathbb{S}^2, \bar\w^{-4})} \le C$, $|\widetilde f|\le C\bar \w^2$, and
\begin{equation*}
\partial _\tau (\bar \w^{-4} e^{3\tau} \partial_\tau \widetilde \xi)
+ \mathrm{div}_{g_0}(\bar \w^{-4}e^{3\tau} \nabla_{g_0}\widetilde \xi)
= \bar \w^{-4} e^{3\tau} \widetilde  f.
\end{equation*}
Thus, a De Giorgi iteration as in \cite[Theorem 1]{N} produces
\begin{equation*}
\sup_{(-1/2,1/2) \times \mathbb{S}^2} |\widetilde \xi| 
\le  C\big(\|\widetilde \xi/\bar \w^2\|_{L^2 ((-1,1) \times \mathbb{S}^2)}
+\|\widetilde f/\bar \w^2 \|_{L^\infty((-1,1) \times \mathbb{S}^2)} \big)\le C.
\end{equation*}
We may then employ \cite[Lemma 4.4]{LT}, for any $\alpha\in (0,1)$, to obtain
\begin{equation*}
\sup_{(-1/4,1/4) \times \mathbb{S}^2} (|\widetilde \xi| \bar \w^{-3-\alpha})\le C.
\end{equation*}
This yields the desired result \eqref{eq-conclusion-decay-xi}. 

Finally, assume that $\partial_tf$ satisfies \eqref{eq-assumption-decay-f}, and observe that
\begin{equation*}
\partial _\tau (\bar \w^{-4} e^{3\tau} \partial_\tau \widetilde  \xi_\tau)
+ \mathrm{div}_{g_0}(\bar \w^{-4}e^{3\tau} \nabla_{g_0}\widetilde  \xi_\tau )
= \bar \w^{-4} e^{3\tau} \widetilde   f_\tau,
\end{equation*}
where $\widetilde  \xi_\tau=\pa_\tau \widetilde  \xi$ and $\widetilde  f_\tau=\pa_\tau \widetilde  f$. 
By the energy estimates for the equation of $\widetilde  \xi$, it is straightforward to see that 
\begin{equation*}
\|\widetilde  \xi_\tau/\bar \w^2\|_{L^2 ((-3/4,3/4) \times \mathbb{S}^2)}\le C.
\end{equation*}
Similarly, we have 
\begin{equation*}
\sup_{(-1/4,1/4) \times \mathbb{S}^2} (|\widetilde  \xi_\tau| \bar \w^{-3-\alpha})\le C.
\end{equation*}
Therefore, $\partial_t\xi$ satisfies \eqref{eq-conclusion-decay-xi}. 
\end{proof}

\begin{rem}\label{rem:v-linear-pde} It should be noted that for any constant $a\in\mathbb{R}$, the function
\begin{align}\label{eq-definition-v0}
v_0(\theta)=\frac12
a\cos \theta (3-\cos^2 \theta )
\end{align}
appearing in the extreme Kerr harmonic map, satisfies the zero eigenvalue equation from the proof of 
Lemma \ref{lem:v-linear-pde}. However, since $v_0$ does not vanish at the poles it does not lie in the space $L^2( \mathbb{S}^2, \bar\w^{-4})$, 
and is thus not considered as a member of the orthonormal basis $\{\varphi_i\}_{i=1}^{\infty}$.
\end{rem}

\begin{rem}
We shall now compute some eigenvalues and eigenfunctions that are releveant to the previous lemma. Consider
\begin{equation} \label{efunction-eq}
\bar\w^{4} \mathrm{div}_{g_0}(\bar\w^{-4} \nabla_{g_0} w)=-\lambda w \quad \text{ on }\mathbb{S}^2\setminus\{N,S\}, \quad w\in L^2( \mathbb{S}^2, \bar\w^{-4}),
\end{equation}
and use the separation of variables ansatz $w(\theta,\phi)=\eta(\theta)\cos(m(\phi-\phi_0))$ for $m=0,1,\ldots$ where $\phi_0$ is a constant. Under this assumption equation~\eqref{efunction-eq} may rewritten as
\begin{equation*}
\sin^5\theta \frac{d}{d\theta}\left(\frac1{\sin^3\theta}\,\frac{d\eta}{d\theta}\right) + \lambda\sin^2\theta\, \eta - m^2 \eta= 0, \quad\quad \eta \in {\mathcal H}=L^2([0,\pi],\bar\w^{-3}).
\end{equation*}
Making the substitutions $x=\cos\theta$ and $\lambda=\ell(\ell+3)$ produces
\begin{equation} \label{axisymm-efunction-eq}
    (1-x^2)^2\eta'' + 2x(1-x^2)\,\eta' + \big(\ell(\ell+3)(1-x^2)-m^2\bigr)\eta = 0,
\end{equation}
where $\eta'=\tfrac{d\eta}{dx}$. It is straightforward to check that $\eta(x)=(1-x^2)\,y(x)$ satisfies~\eqref{axisymm-efunction-eq} if and only if $y(x)$ satisfies the general Legendre equation of degree $\ell+1$ and order $n=\sqrt{m^2+4}$, that is
\begin{equation*}
(1-x^2)^2y'' - 2x(1-x^2)\, y' + \bigl((\ell+1)(\ell+2)(1-x^2) - n^2\bigr) \, y = 0.
\end{equation*}
Thus we conclude that $\eta(x)=(1-x^2) \bigl( \mathrm{a}\,P_{\ell+1}^n(x) + \mathrm{b}\,Q_{\ell+1}^n(x) \bigr)$, where $P_{\ell+1}^n$ and $Q_{\ell+1}^n$ are the associated Legendre functions of the first and second kind respectively, and $\mathrm{a}$, $\mathrm{b}$ are constants. For applications to stationary vacuum black holes, the relevant harmonic maps are axisymmetric, in which case we may restrict attention to $m=0$. In this situation, $n=2$ and $Q^2_{\ell+1}$ has a simple pole at the endpoint $x=1$. It follows that $(1-x^2)Q_{\ell+1}^2(1)\ne0$, and hence cannot be in $\mathcal H$. Similarly, if $\ell\not\in \mathbb N$ then $P^2_{\ell+1}$ has a simple pole at the endpoint $x=-1$. We conclude that $\lambda_\ell=\ell(\ell+3)$ are eigenvalues for each for $\ell\in\mathbb N$, with eigenfunctions 
$w_\ell(\theta)=\sin^2\theta P^2_{\ell+1}(\cos\theta)$. Moreover, by Rodrigues' formula 
\begin{equation*}
P^2_{\ell+1}(x) = \frac{1-x^2}{2^{\ell+1}(\ell+1)!}\,\frac{d^{\ell+3}}{dx^{\ell+3}}(x^2-1)^{\ell+1},
\end{equation*}
and thus the eigenfunctions may be rewritten as 
\begin{equation*}
w_\ell(\theta) = \frac{\sin^4\theta}{2^{\ell+1}(\ell+1)!}\, p_\ell(\cos\theta),
\end{equation*}
where $p_\ell$ is a polynomial of degree $\ell-1$. Lastly, note that since $\mu_2\leq \lambda_2 =10$, the decay constant $\gamma$ of Lemma \ref{lem:v-linear-pde} is not greater than 2.
\end{rem}

We are ready to prove the main result in this section, which offers an asymptotic expansion at infinity
for the renoramlized harmonic map.

\begin{thm}\label{thm:main-2a} Let $(\Phi,v)\in H^1_{\mathrm{loc}}(\mathbb{R}_-\times\mathbb S^2)$ be a weak solution of \eqref{eq:main-2z} 
satisfying \eqref{eq:ubd-1z}, \eqref{eq:ea-1z}, and \eqref{eq:trace-1z} for some positive constants $\Lambda$ and $a$. 
Then there exist constants $c_0$, $c_1$, $C$ 
and degree 0, 1, and 2 spherical harmonics $Y_0$, $Y_1$, and $Y_2$ respectively, 
such that for all $t\le -1$ the following expansions hold
\begin{equation} \label{eq:exter-main-conv-2}
|\Phi(t) - c_0-e^{t} Y_0 -e^{2t} Y_1-e^{3t} Y_2|_{C^{3}(\mathbb{S}^2)} \le Ce^{(3+\beta)t}, 
\end{equation}
and 
\begin{equation} \label{eq:exter-main-conv-3}
|v(t) - v_0(\theta)-c_1 e^{t} \bar \w^4 |_{C^3(\mathbb{S}^2)}\le C e^{(1+\beta) t}\bar{\w}^{3+\alpha},
\end{equation}
for some $\beta\in (0,1)$, any $\alpha\in(0,1)$, and where $v_0$ is given by \eqref{eq-definition-v0}. 
Moreover, corresponding asymptotics are valid for the $t$-derivatives of $(\Phi,v)$.
\end{thm} 

\begin{proof}  Write the first equation of \eqref{eq:main-2z} as $L\Phi= f$, where 
\begin{equation*}
f= 2e^{4u} |\nabla_g v|^2= e^{4t} \bar \w^{-4} e^{4\Phi} \big(|\partial_tv|^2+|\nabla_{g_0}v|^2).   
\end{equation*}
By the estimates for $\Phi$ in Proposition \ref{prop:reg-exterior}, and the estimates for $v$ in Proposition \ref{prop:reg-exterior-improved}, we have
\begin{equation*}
|f|\le Ce^{4t} \bar \w^{-4} \bar\w^{4+2\alpha}\le Ce^{4t} \bar\w^{2\alpha}\le Ce^{4t}
\end{equation*}
for any $\alpha\in (0,1)$. Moreover, by Lemma \ref{lem:phi-linear-pde} there exist constants $c_0$, $C$, $Y_0$
and degree 1 and 2 spherical harmonics $Y_1$ and $Y_2$ respectively, 
such that 
\begin{equation*}
|\Phi(t) - c_0-e^{t} Y_0 -e^{2t}Y_1 -e^{3t}Y_2 |_{C^{1,\sigma}(\mathbb{S}^2)} \le Ce^{(3+\beta)t}
\end{equation*}
for any $\sigma,\beta\in (0,1)$ and all $t\le -1$.
In order to improve this estimate to involve $C^3$-norms, we note that
\begin{equation*}
|\nabla_{g_0}f|\le Ce^{4t} \bar \w^{-5} \bar \w^{4+2\alpha}\le Ce^{4t} \bar\w^{2\alpha-1}\le Ce^{4t}
\end{equation*}
if $\alpha\in (1/2, 1)$. In fact, by choosing $\sigma\in (0,1)$ sufficiently small depending on $\alpha$, we find 
\begin{equation*}
|\nabla_{g_0}f|_{C^{\sigma}(\mathbb{S}^2)}\le  Ce^{4t}.
\end{equation*}
Hence the methods of Lemma \ref{lem:phi-linear-pde} yield
\begin{equation*}
|\Phi(t) - c_0-e^{t}Y_0 -e^{2t}Y_1 - e^{3t}Y_2|_{C^{3,\sigma}(\mathbb{S}^2)} \le Ce^{(3+\beta)t},
\end{equation*}
which gives \eqref{eq:exter-main-conv-2}. Furthermore, estimates of $t$-derivatives may be obtained in a similar manner. 
For example we have 
\begin{equation*} 
|\partial_t\Phi(t) -e^{t}Y_0 -2e^{2t}Y_1 -3 e^{3t}Y_2|_{C^{3,\sigma}(\mathbb{S}^2)} \le Ce^{(3+\beta)t},
\end{equation*}
along with an analogous estimate for $\partial_t^2\Phi$. In particular, it follows that
\begin{equation}\label{eq:exter-main-conv-2a}
|\partial_t\Phi|+|\partial^2_t\Phi| \le Ce^{t}, \quad \quad
|\nabla_{g_0}\Phi|+|\nabla_{g_0}\partial_t\Phi| \le Ce^{2t}. 
\end{equation}

Consider now the equation satisfied by $v$. Observe that 
\begin{align*}
L v +4 \nabla_g u \cdot  \nabla_g v&= \pa_t^2 v+3 \pa_t v+ \Delta_{g_0} v-4\nabla_{g_0} \ln \bar \w \cdot \nabla_{g_0} v+4\nabla_g \Phi \cdot \nabla_g v \\&
=\pa_t^2 v+3 \pa_t v + \bar \w^4\mathrm{div}_{g_0}(\bar \w ^{-4} \nabla_{g_0}v) +4\nabla_g \Phi \cdot \nabla_g v,
\end{align*}
so that the second equation of \eqref{eq:main-2z} may be expressed as
\begin{equation*}
\pa_t^2 v+3 \pa_t v + \bar \w^4\mathrm{div}_{g_0}(\bar \w ^{-4} \nabla_{g_0}v) =h,
\end{equation*}
where 
\begin{equation*}
h=-4\nabla_g \Phi \cdot \nabla_g v=-4\nabla_{g_0} \Phi \cdot \nabla_{g_0} v
-4\pa_t \Phi  \pa_t v.
\end{equation*}
By Remark \ref{rem:v-linear-pde}, we have $\bar \w ^4\mathrm{div}_{g_0}(\bar \w^{-4} \nabla_{g_0}v_0)=0$. 
Therefore, if $\xi=v-v_0$ then
\begin{equation*}
\pa_t^2 \xi+3 \pa_t \xi +\bar \w^4\mathrm{div}_{g_0}(\bar \w^{-4} \nabla_{g_0}\xi) =h. 
\end{equation*}
Notice that \eqref{eq:hat-v-est} gives
\begin{equation}\label{eq:exter-main-conv-2b}
|\nabla_{g_0} v|+|\nabla_{g_0}\pa_t v|\le C\bar\w^{2+\alpha}, \quad \quad
|\pa_t v|+|\pa^2_t v|\le C\bar\w^{3+\alpha},
\end{equation}
which together with \eqref{eq:exter-main-conv-2a} produces
\begin{equation*}
|h| \le  C e^{2t}\bar \w^{2+\alpha} +Ce^{t}\bar\w^{3+\alpha} \le Ce^{t}\bar \w^{2+\alpha}, 
\end{equation*}
and similarly 
\begin{equation*}
|\partial_th| \le   Ce^{t}\bar \w^{2+\alpha} . 
\end{equation*}
We may now apply Lemma \ref{lem:v-linear-pde} to obtain
\begin{equation*}
|\pa_t v|=|\pa_t \xi| \le  Ce^{\delta t} \bar \w^{3+\alpha},
\end{equation*}
for any $\delta\in (0,1)$. Moreover, by using the above estimate of $\partial_t v$ 
instead of the one in \eqref{eq:exter-main-conv-2b}, it follows that
\begin{equation*}
|h| \le  C e^{(1+\delta)t}\bar \w^{2+\alpha}. 
\end{equation*}
Applying Lemma \ref{lem:v-linear-pde} again prouces
\begin{equation*}
|\xi-c_1 e^{t} \bar \w^4 |\le C e^{(1+\beta) t}\bar \w^{3+\alpha},
\end{equation*}
for some constants $\beta\in (0,1)$, $c_1$, and $C>0$.
The desired expansion \eqref{eq:exter-main-conv-3} with $C^3$-norms now follows from Lemma \ref{lem:w2p-}. 
Similar arguments yield the corresponding estimates for the $t$-derivatives.
\end{proof}

\begin{proof}[Proof of Theorem \ref{thm:main-2}] 
The hypotheses guarantee the applicability of Theorems \ref{thm:reg} 
and \ref{thm:main-2a}, from which the desired result follows.
\end{proof}

\section{Near Horizon Geometry and Conical Singularities}
\label{sec7}

In this section we will give an application of the asymptotic analysis presented in previous sections to the classification of near horizon geometries, and will give a geometric interpretation of the parameter $b$ appearing in the expression of tangent maps at punctures. For the notation and background, we refer to the last portion of Section \ref{sec2}.

\subsection{Near horizon limit: proof of Theorem \ref{NHG}.}

Near horizon geometries (NHGs) arise through a limiting process in the vicinity of a degenerate (extreme) black hole, which infinitely magnifies the spacetime geometry at the horizon, see \cite{KunduriLucietti} for a detailed review. In the physics literature, a rigorous justification of the so called near horizon limit that gives rise to the NHG, is often not addressed. Here we will prove that this limit exists as a consequence of the asymptotic analysis studied in this work. 

Consider a puncture $p_l$ at the intersection of two axis rod closures $\overline{\Gamma}_l$ and $\overline{\Gamma}_{l+1}$, lying to the south and north of the puncture. We may introduce polar coordinates $(r,\theta)$ centered at this point, such that $\rho=r\sin\theta$ and $z=r\cos\theta$. As usual, write $u=\Phi-\ln\sin\theta$ so that $U=\Phi+\ln r$, then the spacetime metric becomes
\begin{equation}\label{aqoijfoinhh}
\mathbf{g} = - r^2 e^{2\Phi}d\tau^2 + e^{-2\Phi} \sin^2 \theta (d \phi +wd\tau)^2 +  e^{-2\Phi+2 \boldsymbol{\alpha}} \left(\frac{dr^2}{r^2} + d\theta^2 \right).
\end{equation}
The NHG metric is obtained from the scaling $r=\epsilon \bar{r}$, $\tau=\epsilon^{-1} \bar{\tau}$, $\phi=\bar{\phi}+\Omega \epsilon^{-1}\bar{\tau}$ by letting $\epsilon\rightarrow 0$, where $\Omega$ is the angular velocity of the black hole. Moreover, justification and computation of this near horizon limit will follow from the expansion of the metric coefficients that is implied by Theorems \ref{thm:main-1} and \ref{thm:classification-non-degeneracy}. To begin, observe that without loss of generality it may be assumed that $v=\pm a$ on $\Gamma_{l+1}$, $\Gamma_{l}$ respectively, for some $a>0$. Then for all sufficiently small $r$ this result yields
\begin{align}\label{oahaoqoihnwqh}
\begin{split}
\Phi(r,\theta) &=\bar{\Phi}(\theta)+O(r^{\beta})=-\frac{1}{2} \ln \frac{2a \sqrt{1-b^2}}{1+\cos^2 \theta+2b \cos \theta} +O(r^{\beta}), \\
v(r,\theta)&= \bar{v}(\theta)+O(r^{\beta})=a\cdot \frac{b+b \cos ^2 \theta +2\cos \theta}{1+ \cos ^2 \theta +2b \cos \theta}
+O(r^{\beta}),
\end{split}
\end{align}
where $\beta>0$ and $b\in(-1,1)$.

It remains to calculate the expansions for $\boldsymbol{\alpha}$ and $w$.  Using the relations
\begin{equation*}
\partial_{\rho}=\sin\theta \partial_r +\frac{\cos\theta}{r}\partial_{\theta},\quad\quad\quad
\partial_z =\cos\theta \partial_r -\frac{\sin\theta}{r}\partial_{\theta},
\end{equation*}
together with \eqref{alphaequation1}, and the estimates 
\begin{equation}\label{aoihfoihqoign}
r|\partial_r\Phi|+|\partial_{\theta}(\Phi-\bar{\Phi})|=O(r^{\beta}),\quad\quad\quad (\sin\theta)^{-7/2}r|\partial_r v|
+(\sin\theta)^{-5/2}|\partial_{\theta}(v-\bar{v})|=O(r^{\beta}),
\end{equation}
which follow from Proposition \ref{prop:difference-linear}, we have
\begin{align*}
r\partial_r \boldsymbol{\alpha}
&=r\sin\theta\partial_{\rho}\boldsymbol{\alpha} +r\cos\theta\partial_z \boldsymbol{\alpha}\\
&=r\sin\theta\Big[2u_{\rho}\!+\!\tfrac{1}{\rho}\!+\!\rho(u_{\rho}^2\! -\!u_z^2\!+\!e^{4u}(v_{\rho}^2\! -\!v_z^2))\Big]\!+\!r\cos\theta\big[2u_z \!+\!\rho(2u_{\rho}u_z \! +\!2e^{4u}v_{\rho}v_z)\big]\\
&=1-\sin^2 \theta\big[(\bar{\Phi}_{\theta}-\cot\theta)^2+e^{4u}\bar{v}_{\theta}^2\big]+O(r^{\beta})\\
&=O(r^{\beta}), 
\end{align*}
where in the last equality we employed the expression for $\bar{\Phi}$ as well as \eqref{eq:first-int} to find
\begin{align*}
\bar{\Phi}_{\theta}&=-\frac{\sin\theta(\cos\theta +b)}{1+\cos^2\theta+2b\cos\theta},\\
e^{4u}\bar{v}_{\theta}^2&=\tfrac{\sin^2 \theta}{4a^2}e^{-4\bar{\Phi}}+O(r^{\beta})=\frac{a(1-b^2)\sin^2\theta}{(1+\cos^2\theta +2b\cos\theta)^2}+O(r^{\beta}).
\end{align*}
Similar manipulations give rise to
\begin{align*}
\begin{split}
\partial_{\theta}\boldsymbol{\alpha}=&-\sin\theta\cos\theta+2\sin^2\theta \bar{\Phi}_{\theta}
+\sin\theta\cos\theta\left(\bar{\Phi}_{\theta}^2 +e^{4u}\bar{v}_{\theta}^2 \right)+O(r^{\beta})\\
=&-\sin\theta\cos\theta -\frac{\sin^3\theta(\cos\theta+2b)}{1+\cos^2\theta+2b\cos\theta}+O(r^{\beta}).
\end{split}
\end{align*}
It follows that
\begin{equation}\label{alphaexpansionskka}
\boldsymbol{\alpha}(r,\theta)=\ln\left(1+\cos^2 \theta +2b\cos\theta\right)+\boldsymbol{\alpha}_0
+O(r^{\beta})=:\bar{\boldsymbol{\alpha}}(\theta)+O(r^{\beta}),
\end{equation}
where $\boldsymbol{\alpha}_0$ is a constant of integration.

We will now compute the expansion of $w$. Observe that with the help of \eqref{wequations1} and \eqref{aoihfoihqoign} we have
\begin{align*}
\begin{split}
r\partial_r w &=r \sin\theta \partial_{\rho} w+r\cos\theta \partial_z w\\
&=r\sin\theta\left(2\rho e^{4u}v_z\right)-r\cos\theta\left(2\rho e^{4u}v_{\rho}\right)\\
&=2r^2 \sin\theta e^{4u}\left[\sin\theta\left(\cos\theta v_r -\tfrac{\sin\theta}{r} v_{\theta}\right)
-\cos\theta\left(\sin\theta v_r +\tfrac{\cos\theta}{r} v_{\theta}\right)\right]\\
&=-2r e^{4\bar{\Phi}}\frac{\bar{v}_{\theta}}{\sin^3 \theta}+O\Big(\frac{r^{1+\beta}}{\sin^{1/2}\theta}\Big),
\end{split}
\end{align*}
and similarly
\begin{align*}
\begin{split}
\partial_{\theta} w &=r \cos\theta \partial_{\rho} w-r\sin\theta \partial_z w\\
&=r\cos\theta\left(2\rho e^{4u}v_z\right)+r\sin\theta\left(2\rho e^{4u}v_{\rho}\right)\\
&=2r^2 \sin\theta e^{4u}\left[\cos\theta\left(\cos\theta v_r -\tfrac{\sin\theta}{r} v_{\theta}\right)
+\sin\theta\left(\sin\theta v_r +\tfrac{\cos\theta}{r} v_{\theta}\right)\right]\\
&=2r^2 e^{4\Phi} \frac{v_r}{\sin^3 \theta}\\
&=O(r^{1+\beta}).
\end{split}
\end{align*}
It follows that
\begin{equation}\label{alphaexpansionskka1}
w=w_0-2re^{4\bar{\Phi}}\frac{\bar{v}_{\theta}}{\sin^{3}\theta}+O\Big(\frac{r^{1+\beta}}{\sin^{1/2}\theta}\Big)=w_0+\frac{r}{a}+O\Big(\frac{r^{1+\beta}}{\sin^{1/2}\theta}\Big),
\end{equation}
where \eqref{eq:first-int} has been used in the last equality and $w_0$ is an integration constant that represents the sign reversed angular velocity of the black hole, $-\Omega$. We may now apply the expansions \eqref{oahaoqoihnwqh}, \eqref{alphaexpansionskka}, and \eqref{alphaexpansionskka1} to take the near horizon limit in \eqref{aqoijfoinhh} and obtain the near horizon metric
\begin{align*}
\mathbf{g}_{NH} &= - \bar{r}^2 e^{2\bar{\Phi}}d\bar{\tau}^2 + e^{-2\bar{\Phi}} \sin^2 \theta \Big(d \bar{\phi} +\frac{\bar{r}}{a}d\bar{\tau}\Big)^2 +  e^{-2\bar{\Phi}+2 \bar{\alpha}} \Big(\frac{d\bar{r}^2}{\bar{r}^2} + d\theta^2 \Big)\\
&=-\bar{r}^2\Big(\frac{1+\cos^2 \theta +2b\cos\theta}{2a\sqrt{1-b^2}}\Big)d\bar{\tau}^2
+\frac{2a\sqrt{1-b^2}\sin^2 \theta}{1+\cos^2 \theta+2b\cos\theta}\Big(d\bar{\phi}+\frac{\bar{r}}{a}d\bar{\tau}\Big)^2\\
&\qquad+\frac{a\sqrt{1-b^2}}{2}(1+\cos^2 \theta+2b\cos\theta)\Big(\frac{d\bar{r}^2}{\bar{r}^2}+d\theta^2 \Big).
\end{align*}
The integration constant of \eqref{alphaexpansionskka} has been chosen so that $e^{2\boldsymbol{\alpha}_0}=1/4$, which yields the near horizon metric of the extreme Kerr black hole (extremal Kerr throat metric) \cite{BHorowitz} when $a=2\mathcal{J}$ and $b=0$, where $\mathcal{J}$ denotes angular momentum.

\subsection{Angle defect at punctures: proof of Theorem \ref{conicalsing}.}

We will now show that the parameter $b\in(-1,1)$, that appears in the tangent map and near horizon geometry of an extreme black hole, completely determines the difference of logarithmic angle defects associated with the two neighboring axis rods.
Namely, consider the two axis rods $\Gamma_l$ and $\Gamma_{l+1}$ lying directly to the south and north of the puncture $p_l$, each having logarithmic angle defect $\mathbf{b}_l$ and $\mathbf{b}_{l+1}$, respectively. According to \eqref{alphaequation111} and \eqref{alphaexpansionskka} we find that
\begin{equation*}
\mathbf{b}_{l+1}-\mathbf{b}_{l}=\lim_{r\rightarrow 0}\left(\boldsymbol{\alpha}(r,0)-\boldsymbol{\alpha}(r,\pi)\right)
=\ln\left(\frac{1+b}{1-b}
\right)\quad \Leftrightarrow \quad b=\tanh\left(\frac{\mathbf{b}_{l+1}-\mathbf{b}_l}{2}\right).
\end{equation*}
In particular, the case $b=0$ corresponds to a `balanced horizon', in that conical singularities on both sides of the black hole may be relieved by choosing the constant $\boldsymbol{\alpha}_0$ in \eqref{alphaexpansionskka} appropriately.

\subsection*{Declarations}
{Q. Han acknowledges the support of NSF Grant DMS-2305038. M. Khuri acknowledges the support of NSF Grants DMS-2104229, DMS-2405045, and Simons Foundation Fellowship 681443. J. Xiong acknowledges the partial support of NSFC grants 12325104 and 12271028.}

\end{document}